\documentclass[a4paper,reqno]{amsart}
\usepackage{amssymb}
\usepackage{euscript}
\usepackage{pstricks}
\usepackage[dvipdfm]{graphicx}
\usepackage{bmpsize}
\usepackage{boxedminipage}
\usepackage{url}

\newcommand{\C}{\mathbb C}
\newcommand{\F}{\mathbb F}
\newcommand{\Z}{\mathbb Z}
\newcommand{\Q}{\mathbb Q}
\newcommand{\N}{\mathbb N}

\newcommand{\R}{\mathbb R}
\newcommand{\T}{\mathbb T}

\newcommand{\Cal}{\EuScript}

\newcommand{\ord}{\operatorname{ord}}

\newcommand{\mlt}{\operatorname{mult}}

\newcommand{\md}{\mathbin{\mathsf{mod}}}
\newcommand{\nm}{\operatorname{\mathsf {num}}}
\newcommand{\wrd}{\operatorname{\mathsf {wrd}}}
\newcommand{\mon}{\operatorname{\mathsf {mon}}}

\DeclareMathOperator{\wt}{\mathsf{wt}}

\renewcommand{\:}{\colon}
\renewcommand{\>}{\rightarrow}

\newtheorem{thm}{Theorem}[section] 
\newtheorem{prop}[thm]{Proposition}
\newtheorem{cor}[thm]{Corollary}

\theoremstyle{definition}

\theoremstyle{remark}
   \newtheorem*{rmk}{Remark}

\theoremstyle{plain}
\newtheorem{lem}[thm]{Lemma}

\newtheorem*{OpQu*}{Open question}
\theoremstyle{definition}
\newtheorem{defn}[thm]{Definition}
\newtheorem*{defn*}{Definition}
\theoremstyle{remark}
\newtheorem{note}[thm]{Note}
\newtheorem*{note*}{Note}
\newtheorem{exmp}[thm]{Example}

\newtheorem*{exmp*}{Example}
\newtheorem*{exmps*}{Examples}

\begin{document}

\title[Quantization causes waves]{Quantization causes waves:\\
Smooth finitely computable functions are affine
}
\author{Vladimir Anashin}
\date{\today} 
\maketitle

\begin{abstract}
Given an automaton (a letter-to-letter transducer)
$\mathfrak A$  whose input and output alphabets are $\F_p=\{0,1,\ldots,p-1\}$, 
one visualizes  word
transformations performed by $\mathfrak A$ by a point set
$\mathbf P(\mathfrak
A)$ of 
real plane $\R^2$ 
as follows: 
To an $m$-letter non-empty word $v=\gamma_{m-1}\gamma_{m-2}\ldots\gamma_0$
over the alphabet $\Cal A$
put into the correspondence a rational number $0.
v$ whose base-$p$ expansion is 
$0.\gamma_{m-1}\gamma_{m-2}\ldots\gamma_0$; then
to every  $m$-letter input word $w=\alpha_{m-1}\alpha_{m-2}\cdots\alpha _0$ of the automaton $\mathfrak
A$ and to the 
respective $m$-letter output word $\mathfrak a(w)=\beta_{m-1}\beta_{m-2}\cdots \beta_0$ 
(rightmost letters are feeded to/outputted from the automaton prior to leftmost
ones) 
there corresponds a point $(0.
w;0.
{\mathfrak
a(w)})$
of the real unit square $[0,1]^2$;
denote
$\mathbf P(\mathfrak A)$  a closure  (in the topology of 
$\R^2$) of the point set $(0.
w;0.
{\mathfrak
a(w)})$
where $w$ ranges over the set $\Cal W$ of all non-empty words over the alphabet $\F_p$.

For a finite-state automaton $\mathfrak A$, it is shown that once some points of $\mathbf P(\mathfrak A)$ constitute
a smooth (of a class $C^2$)  curve in $\mathbb R^2$, the curve is a segment of a straight line
with a rational slope; and there are only finitely many straight lines whose
segments are in $\mathbf{P}(\mathfrak A)$. Moreover, when identifying $\mathbf P(\mathfrak A)$
with a subset of a  2-dimensional torus $\mathbb T^2\subset\mathbb R^3$ (under a natural
mapping of the real unit square $[0,1]^2$ onto $\mathbb T^2$)
the smooth curves  from $\mathbf P(\mathfrak A)$ constitute a collection of  torus windings. In cylindrical coordinates either of the windings can be ascribed
to a complex-valued function $\psi(x)=e^{i(Ax-2\pi B(t))}$ $(x\in\R)$ for
suitable rational $A,B(t)$. 
Since $\psi(x)$ is a standard expression for a matter wave in quantum theory
(where $B(t)=tB(t_0)$), and since transducers can be regarded as a mathematical
formalization for  causal discrete systems, 
the
main result of the paper  might serve as a
mathematical reasoning why wave phenomena are inherent in quantum systems: This is
because of causality principle and the discreteness of matter.
\end{abstract}

\section{Introduction}
\label{sec:intro}
In the paper, we examine $C^2$-smooth real functions which
can be  computed (in some new but natural meaning which is rigorously defined below) on finite automata, i.e., on  sequential machines that have only finite number of states.
We show that all these functions are affine and, moreover,  that they can
be expressed as complex functions $e^{i(Ax+B)}$ and thus can be ascribed (also in some natural rigorous meaning) to  matter waves from quantum theory.

A general problem of evaluation of real functions on abstract discrete machines naturally
arose at the very moment the first digital computers had been invented. There are a number of various mathematical statements of the problem 
which depend both on  specific mathematical model of a digital computer
(the abstract machine)  and on the representation of reals in some `digital' form.
For instance, real number computations on  Turing machines  constitute a core of
theory
of 
constructive reals and computable functions. The theory demonstrates intensive development for  during more than half
a century, see e.g. \cite{ComputAnalPhys} and references therein. 
Sequential machines (also known as Mealy automata,
or as finite-state letter-to-letter transducers) are, speaking loosely,
 Turing machines whose heads move only in one direction. Sequential machines
 are   therefore less power computers compared
to general Turing machines; however, a number of real world phenomena and processes
can be modelled by sequential machines since the latter can be considered
as (non-autonomous) discrete dynamical systems. That is why the theory of functions
computable by sequential machines, which  constitutes a substantial part of automata theory,  has numerous applications not only in mathematics itself
(e.g., in real analysis, $p$-adic analysis, number theory, complexity theory,
dynamics,
etc.) but also in computer science, physics,
linguistics and in many other sciences, see e.g. monographs 
\cite{Allouche-Shall,AnKhr,Bra,Car-Long_Automata,Eilenberg_Auto,AlgCombinWords,Yb-eng} for details and references.



The paper was motivated by empirical data  obtained during a research project related to
an applied problem which assumed
intensive computer
experiments with automata modelling of various cryptographic
primitives used in stream ciphers, hash functions, etc. Word transformations
performed by the automata where visualised, namely, represented by points
of the unit square
$\mathbb I^2=[0,1]\times[0,1]$ in  real plane $\R^2$ so that coordinates of the points  relate
numerical (radix) representations of  input words to the numerical representations
of
corresponding output words. It was noticed that once the modelled system
was finite-state,
and once input words were taken sufficiently long, 
some linear structures (looking like segments of straight lines and somewhat
resembling
pictures from a double-slit experiment in quantum physics, cf. Figures \ref{fig:Plot-16}--\ref{fig:Plot-17}
and Figure \ref{fig:2slit}) may appear in the graph, but more complicated structures
like smooth curves of higher order had never been observed. A particular aim of the paper is  to give  mathematical
explanation of the phenomenon  and to characterize these linear structures.

But during the research it became evident that the problem (which
actually is a question what smooth real functions can be modelled on finite
automata) has applications not only to cryptography (see e.g. \cite[Chapter 11]{AnKhr}) but
also may be related to mathematical formalism of quantum theory. As a matter
of fact, the latter relation (which we believe does exist) can be regarded
as 
a yet another 
answer to the following question discussed by A.~Khrennikov in a series of papers devoted to so-called Prequantum Classical Statistical Field theory, see e.g. \cite{Khren-quant-av,Khren-quick}:
\emph{Why mathematical formalism of quantum theory} (which is based on the theory of linear operators on Hilbert
spaces) \emph{is essentially linear although
a number of quantum phenomena demonstrate an extremely non-linear behavior?}
 
Thus the goal of the paper is twofold: 
\begin{itemize}
\item firstly, to characterize
real functions which can be computed by  finite automata;
and
\item secondly,  to give (using obtained description of the functions) some
mathematical
reasoning why 
wave phenomena are inherent in quantum systems.
\end{itemize} 

The major part of the paper  focuses on real functions which can be computed by
finite automata while the said mathematical reasoning 
is considered in a closing section which contains a discussion of  possible
applications of mathematical results of the paper to quantum theory. We are
not going to discuss cryptographic applications here; they will be postponed
to forthcoming papers.

In the paper, by a general automaton (whose set of states is not necessarily
finite) we mean a machine which performs
letter-by-letter transformations of words over input alphabet into 
words over output alphabet: Once a letter is feeded to the  automaton, the automaton updates its current state (which initially is  fixed and so is the same for all input words) to the next one and produces 
corresponding output letter. Both the next state and the output
letter depend both on the current state and on the input letter. 
Therefore each
letter of output
word depends \emph{only} on those letters of  input word which have already been
feeded to the automaton. An input word is a finite sequence of letters;
the letters can naturally be ascribed
to `causes' while letters of the corresponding output word can be regarded
as `effects. `Causality' just means that effects depend only on causes that `already have happened'; therefore an automaton is  an adequate mathematical formalism for a specific manifestation of causality principle once we assume that there exist only finitely many causes
and effects, cf., e.g.,\cite{Vuillem_fin,Vuillem_DigNum}.


When studying real functions that can be computed by  an automaton $\mathfrak A$ whose input/output
alphabets are $\Cal A=\{0,1,\ldots,p-1\}$ (where $p>1$ is an integer from
$\N=\{1,2,3,\ldots\}$) most  authors  
follow common approach which described in e.g. \cite[Chapter XIII, Section 4]{Eilenberg_Auto}: They
associate an infinite
word $\alpha_1\alpha_2\ldots\alpha_n\ldots$ over $\Cal A$ to a
real number whose base-$p$ expansion is 
$0.\alpha_1\alpha_2\ldots\alpha_n\ldots=\sum_{i=1}^\infty\alpha_ip^{-i}$
and consider a real function $d_{\mathfrak A}$ defined as follows: Given $x\in[0,1]$, 
take its base-$p$ expansion $x=\sum_{i=1}^\infty\alpha_ip^{-i}$; then
produce an infinite output sequence $\beta_1\beta_2\ldots\beta_n\ldots$ of $\mathfrak A$ by successfully feeding
the automaton  with the letters $\alpha_1$, $\alpha_2$, etc., and put 
$d_{\mathfrak A}(x)=\sum_{i=1}^\infty\beta_ip^{-i}$.
Being
feeded by infinite input sequence $\alpha_1\alpha_2\ldots\alpha_n\ldots$,
the automaton $\mathfrak A$
produces a unique infinite output sequence $\beta_1\beta_2\ldots\beta_n\ldots$;
therefore
the function $d_{\mathfrak A}$ is well defined everywhere on the real closed
unit  interval (segment) $\mathbb I=[0,1]$
with the exception of maybe a countable set $D\subset[0,1]$ of points; namely, of those having two
distinct base-$p$ expansions
$0.\gamma_1\gamma_2\ldots\gamma_n0\ldots0\ldots=0.\gamma_1\gamma_2\ldots\gamma_{n-1}(\gamma_n-1)(p-1)\ldots
(p-1)\ldots$.
 The point set $\mathbf M(\mathfrak A)=\{(x;d_{\mathfrak A}(x))\in\mathbb R^2\colon x\in[0,1]\}$ can be considered as
a graph of the real function $d_{\mathfrak A}$ specified by the automaton $\mathfrak A$ (every time, before being feeded by the very first letter of each infinite
input word  the automaton $\mathfrak A$ is assumed to be in a fixed state $s_0$,
the \emph{initial state}).
Indeed,
 $d_{\mathfrak A}(x)$ is defined uniquely  for $x\in [0,1]\setminus D$ and $d_{\mathfrak A}(x)$
can be ascribed to at most two values for $x\in D$; so $d_{\mathfrak A}$ can be treated a real function which is defined on the unit segment $[0,1]$ and has not more that a countable number points of discontinuity in $[0,1]$.
In the sequel we 
refer  $\mathbf M(\mathfrak A)$ as to the \emph{Monna graph} of
the automaton $\mathfrak A$, cf. Subsection \ref{ssec:plots}.

The said common approach (and its various generalisations) is utilised in numerous
papers, see e.g.
\cite{Cherep_Approx,Cherepov_Approx-contf,Konech_Aff,Lisovik_Realfunk,Shkar_AffineAuto}.
Speaking loosely, the common approach looks as if one feeds
the automaton $\mathfrak A$ 
by a base-$p$ expansion of a real number $x\in[0,1]$ so that \emph{leftmost
\textup{(i.e., the most significant)} digits are feeded to the automaton prior to rightmost ones} and observes output as real numbers since the automaton outputs accordingly
leftmost
digits
of the base-$p$ expansion of $d_{\mathfrak A}(x)\in[0,1]$ 
 prior to rightmost
ones thus ascribing to the automaton $\mathfrak
A$
the real function $d_{\mathfrak A}$. We stress  that  the function $d_{\mathfrak A}$
is well defined almost everywhere on $[0,1]$ due to \emph{namely that order} in which
digits of base-$p$ expansion are feeded to (and outputted from) the automaton
$\mathfrak A$. 

A crucial difference of  the approach used in our paper from the mentioned one  is 
that \emph{the order we feed digits to  (and read digits from) the automaton is
inverse}: Namely, 
\begin{enumerate}
\item given
a real number $x\in[0,1]$, we 
represent $x$ via base-$p$ expansion
$x=0.\alpha_1\alpha_2\ldots\alpha_n\ldots$ (we take both expansions if $x$ has
two distinct ones); 
\item from the base-$p$ expansion $0.\alpha_1\alpha_2\ldots\alpha_n\ldots$
we derive corresponding sequence $\alpha_1, \alpha_1\alpha_2,
\alpha_1\alpha_2\alpha_3, \ldots$ of words; then 
\item feeding the automaton $\mathfrak
A$ successively  by
the words $\alpha_1, \alpha_1\alpha_2,
\alpha_1\alpha_2\alpha_3, \ldots$ so that \emph{rightmost letters are feeded to $\mathfrak A$  prior
to leftmost ones} we obtain corresponding output word sequence 
$\zeta_{11},\zeta_{12}\zeta_{22},\zeta_{13}\zeta_{23}\zeta_{33},\ldots$;
\item  to the  output sequence
 we put into a correspondence the sequence $\mathcal S(x)$ of rational
numbers whose base-$p$
expansions are 
$0.\zeta_{11},0.\zeta_{12}\zeta_{22},0.\zeta_{13}\zeta_{23}\zeta_{33},\ldots$
thus
obtaining a point set $\Cal X(x)=\{(0.\alpha_1\ldots\alpha_i;
0.\zeta_{1i}\zeta_{2i}\ldots\zeta_{ii})\:i=1,2,\ldots\}$ in the real unit square $\mathbb I^2=[0,1]\times[0,1]$; after that
\item we consider the set $\mathcal F(x)$ of \emph{all} cluster points of the sequence $\Cal S(x)$; 
\item finally, we specify  a \emph{real plot} (or, briefly, a plot) of the automaton $\mathfrak
A$ as a union  
$\mathbf{P}(\mathfrak A)=\cup_{x\in[0,1],y\in\mathcal F(x)}((x;y)\cup\Cal X(x))$. 
\end{enumerate}
In other words, $\mathbf P(\mathfrak A)$ is a closure in the unit square
$\mathbb I^2$ of the union $\cup_{i=1}^\infty\mathbf L_i(\mathfrak A)$ where
$\mathbf L_i(\mathfrak A)=\{(0.\alpha_1\ldots\alpha_i;
0.\zeta_{1i}\zeta_{2i}\ldots\zeta_{ii})\: x\in\mathbb I\}$ is the \emph{$i$-th
layer} of the plot $\mathbf P(\mathfrak A)$. That is, the plot  $\mathbf
P(\mathfrak A)$ can be considered as a `limit' of the sequence  of sets 
$\cup_{i=1}^n\mathbf L_i(\mathfrak A)$, the \emph{approximate plots} at word
length $N$, while $N\to\infty$   (see more formal definitions in Subsection \ref{ssec:plots}).
Note that according to automata 0-1 law (cf. \cite[Proposition 11.15]{AnKhr}
and \cite{me:Discr_Syst})
 the plot $\mathbf P(\mathfrak A)$ of arbitrary  automaton $\mathfrak A$
can be of two
kinds only: Either $\mathbf{P}(\mathfrak A)=\mathbb I^2$ or $\mathbf{P}(\mathfrak A)$ is a (Lebesgue) measure-0 closed subset of $\R^2$.
Moreover, if the number of states of the automaton $\mathfrak A$ is finite (further
in the paper these automata are referred to as \emph{finite} ones), then the second case takes place.

We stress crucial advantage of real plots over Monna graphs: In a contrast to the Monna graph $\mathbf M(\mathfrak A)$, a real plot  $\mathbf{P}(\mathfrak A)$ is capable of showing true long-term behavior of
 automaton $\mathfrak A$ (i.e., when $\mathfrak A$ is feeded by sufficiently
long words)
rather than a short-term behaviour displayed by the Monna graph $\mathbf
M(\mathfrak A)$
since due to the very construction of the real plot the higher order (i.e.,
the most significant) digits of the real number represented by the output
word
are formed by the latest outputted letters of the output word whereas  the construction
of the Monna graph assumes that the higher order digits are formed by the
earliest outputted letters. This results in a drastically different
appearances of  the real
plot and of the Monna graph:
Real plot  clearly demonstrates that corresponding automaton is `ultimately linear' (that
is, exhibits
linear long-term behavior), cf. Figures \ref{fig:Plot-16}--\ref{fig:Plot};
whereas the Monna graph is incapable to reveal this important
feature of the automaton, cf. Figure \ref{fig:Monna}. This is the main reason why in the paper we focus  on real plots of automata rather
than on their Monna graphs.

 \begin{figure}[ht]
 \begin{minipage}[b]{.45\linewidth}
 \includegraphics[width=0.92\textwidth,natwidth=610,natheight=642]{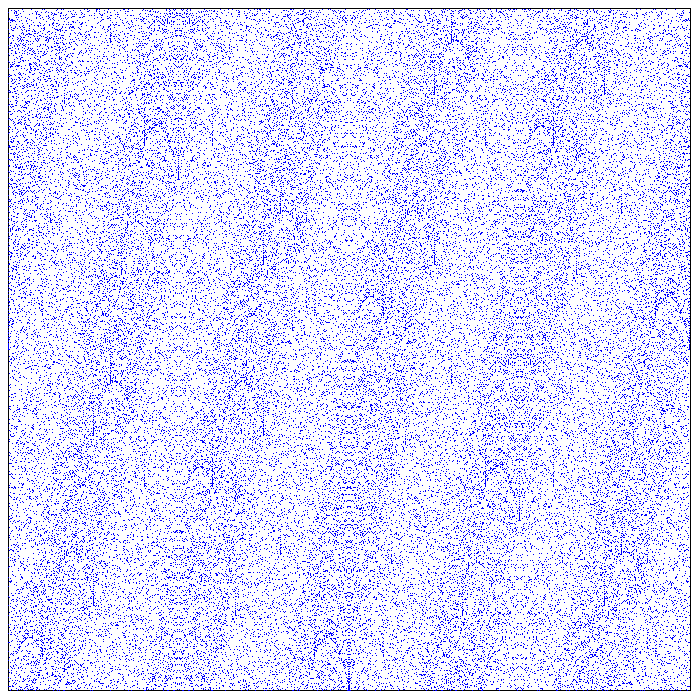}
 \caption{\quad Approximate plot of an 
 automaton at word length 16 }
 \label{fig:Plot-16}
 \end{minipage}\hfill
 \begin{minipage}[b]{.45\linewidth}
 \includegraphics[width=0.92\textwidth,natwidth=610,natheight=642]{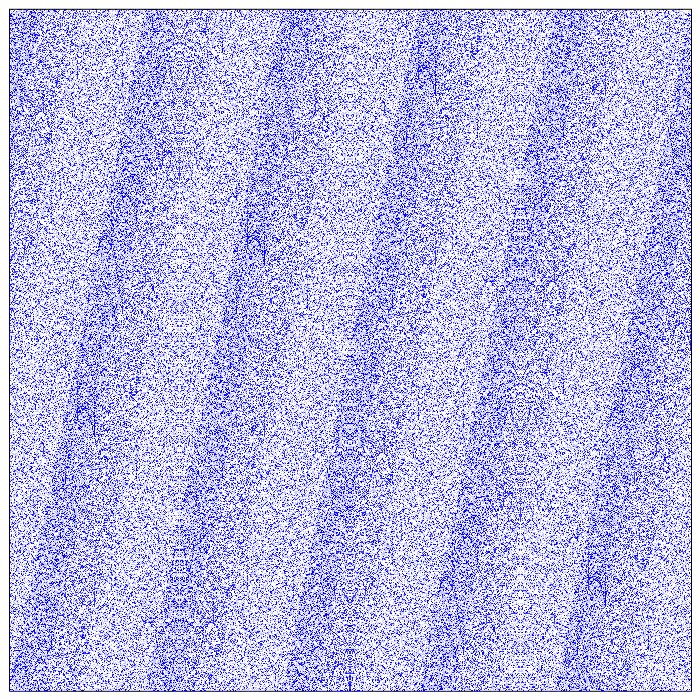}
 \caption{\quad Approximate plot of the same automaton at word length 17}
 \label{fig:Plot-17}
 \end{minipage}\hfill
 \begin{minipage}[b]{.45\linewidth}
 \includegraphics[width=0.92\textwidth,natwidth=610,natheight=642]{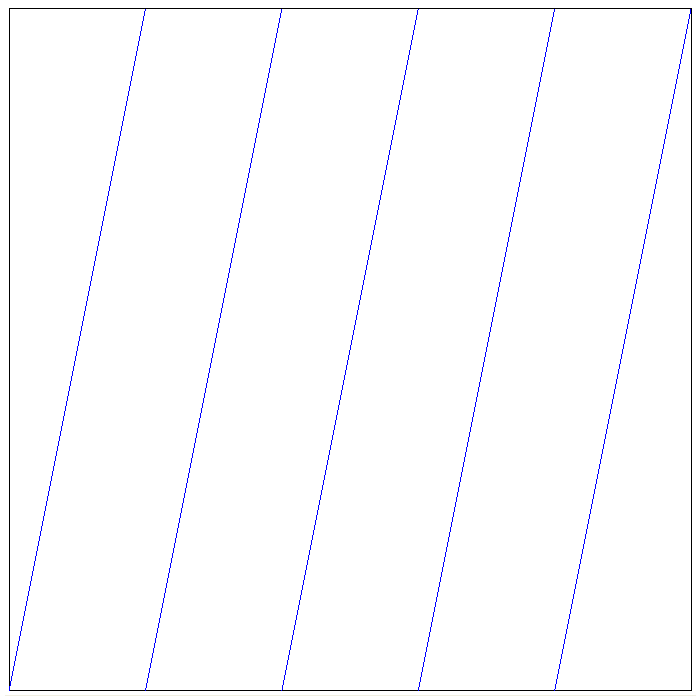}
 \caption{Cluster points of the plot of the same automaton \qquad \qquad}
 \label{fig:Plot}
 \end{minipage}\hfill
 \begin{minipage}[b]{.45\linewidth}
 \includegraphics[width=0.92\textwidth,natwidth=610,natheight=642]{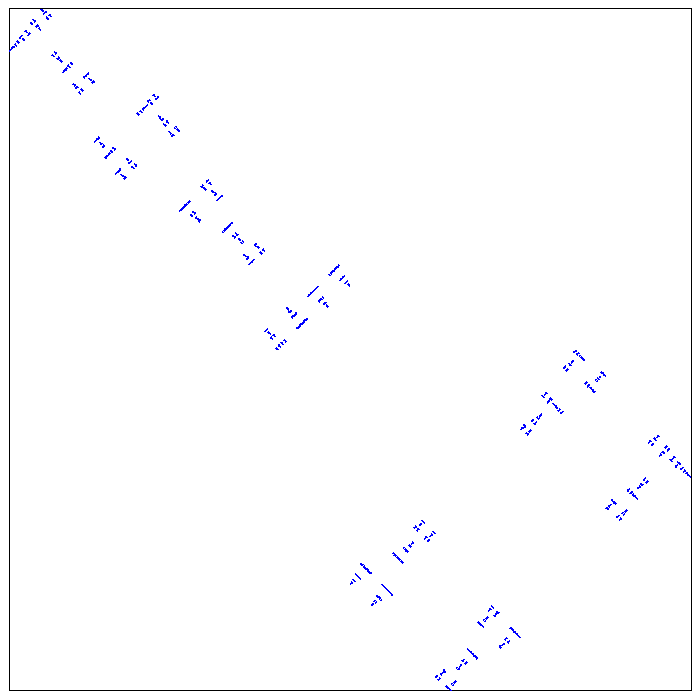} \caption{The Monna graph of the same automaton}
 \label{fig:Monna}
 \end{minipage}
 \end{figure}

Therefore when specifying a notion of computability  of a real-valued function $g\:G\>[0,1]$ (where $G\subset[0,1]$) on 
automata, at least two different approaches do exist: The first
one is to speak of the case 
when the graph $\mathbf G(g)=\{(x;g(x))\colon x\in G\}$ of the function $G$ lie completely in $\mathbf M(\mathfrak
A)$ for some automaton $\mathfrak A$ while the second one is to consider the case when $\mathbf G(g)\subset\mathbf{P}(\mathfrak A)$. Papers 
\cite{Cherep_Approx,Cherepov_Approx-contf,Konech_Aff,Lisovik_Realfunk,Shkar_AffineAuto}
mentioned above
basically deal with  the  computability in the first meaning whereas our's  paper deals
with the computability of the second kind. Note that classes of real functions which are computable
on finite automata are different depending on the meaning: For instance,  the
function $\lfloor px\rfloor$ (where $\lfloor a \rfloor$ stands for the integral
part of $a\in\R$, i.e., for the biggest integer not exceeding $a$) is not
computable in the first meaning but is computable in the second one whilst
the function $p^{-1}x$ is computable in the first meaning but is not computable
in the second one. 

To the best of our knowledge, our approach (which is based on  real plots rather than on Monna graphs) was originally used in \cite{AnKhr} and  was never considered before by other authors.

In the sequel
we refer real functions $g\:G\>[0,1]$ with domain $G\subset[0,1]$ as to \emph{finitely
computable} if there exists a finite automaton $\mathfrak A$ whose real plot
contains the graph  of the function $g$; i.e., if
$\mathbf G(g)\subset\mathbf P(\mathfrak A)$.
Main result of our paper is Theorem \ref{thm:main} which characterizes all
finitely computable   $C^2$-functions $g$ 
defined on  a sub-segment $D=[a,b)\subset[0,1]$: The theorem
yields that 
\emph{if a finitely computable function $g\:D\>[0,1]$ is twice differentiable and if its second derivative is continuous
everywhere on $D$ then $g$ is necessarily affine of the form $g(x)=Ax+B$  for suitable rational $p$-adic $A,B$} (that is, for $A$, $B$ which can be represented by irreducible fractions whose denominators are
co-prime to $p$). Moreover, this is true 
in $n$-dimensional case as well (Theorem \ref{thm:main-mult}).


In view of Theorem \ref{thm:main} it is noteworthy that despite the classes
of functions computable on finite automata are different depending on the meaning the
computability is understood, nonetheless
if a function $g\:[0,1]\>[0,1]$ is everywhere differentiable on $(0,1)$ and
$\mathbf G(g)\subset\mathbf M(\mathfrak A)$ for some finite automaton $\mathfrak
A$ with  binary input/output alphabets then $g$ is necessarily affine, see
\cite{Lisovik_Realfunk}.  In \cite{Konech_Aff} it is shown that a similar assertion holds
for multivariate continuously differentiable functions and arbitrary finite
alphabets. 
Therefore finite automata should be judged as rather `weak computers' in
all meanings  since 
only quite simple
real functions can be evaluated on these devices.  From this view, results
of the current paper are
some contribution to the theory of computable real functions.


It is worth mentioning
right now
that actually our proof reveals a basic reason why smooth functions which can be represented
by finite automata are necessarily affine: This is because \emph{squaring
can not be performed by a finite automaton}; that is, an automaton which,  being feeded by a base-$p$ expansion
of $n$, outputs a base-$p$ expansion of $n^2$ for every positive integer
$n$, can not be finite (the
latter is a well-known fact from automata theory, see e.g., \cite[Theorem
2.2.3]{Bra}).

It is also worth mentioning that the question when $\mathbf G(g)\subset
\mathbf M(\mathfrak A)$ is  somewhat easier to handle than the question when
$\mathbf G(g)\subset\mathbf{P}(\mathfrak A)$. Indeed,  in the first case
once
 the
automaton $\mathfrak A$  is feeded by an infinite word $\ldots\alpha_3\alpha_2\alpha_1$, the word is treated
as a base-$p$ expansion of a unique real number $x=0.\alpha_1\alpha_2\alpha_3\ldots\in[0,1]$,
corresponding output of $\mathfrak A$ is also an infinite word 
$\ldots\beta_3\beta_2\beta_1$ which also is treated as a base-$p$ expansion
of a unique
real
number $y=0.\beta_1\beta_2\beta_3\ldots\in[0,1]$. This results in a unique point $(x;y)$ of the unit square $\mathbb
I^2$ in the first case; whilst in the second case the automaton  $\mathfrak A$, being
feeded by the infinite word $\ldots\alpha_3\alpha_2\alpha_1$, produces
generally   an infinite  point set  of a cardinality continuum: The set is a closure of the point set 
$\{(0.\alpha_n\alpha_{n-1}\ldots\alpha_1;0.\beta_n\beta_{n-1}\ldots\beta_1)\:n=1,2,\ldots\}$
in $\mathbb I^2$.  
Due to this reason
during the proofs we have to use more complicated techniques from real analysis
which in some cases we combine with methods of  $p$-adic
analysis. Therefore some proofs are involved; but to make  general idea
of a proof as transparent as possible in the sequel we explain it in loose terms when appropriate.

Last but not least: Our approach reveals another important feature of smooth
functions which can be computed on finite automata. From Figure \ref{fig:Plot}
it can be clearly observed that limit points of the plot constitute a torus winding if one converts a unit square into torus by gluing together opposite
sides of the square. This is not occasional: Our Theorem \ref{thm:main}
yields that \emph{if the unit square $\mathbb I^2$ is mapped onto a torus
$\T^2\subset\mathbb R^3$, the smooth curves from the plot become torus windings}; and these \emph{windings after being represented in cylindrical coordinates are described
by
complex-valued functions $e^{i(Ax+B)}$} ($x\in[0,1]$), see Corollary
\ref{cor:mult-add-compl}. But in quantum theory the
latter  exponential functions are ascribed to matter waves 
(cf., de Broglie waves); therefore, since automata can be considered
as models for discrete
casual systems, the \emph{results of our paper give some mathematical evidence that  matter waves 
are inherent in quantum systems merely due to causality principle and discreteness of matter} (quantization). We discuss these possible connections
to physics in Section
\ref{sec:Concl}.

Note that for
not to overload the paper with extra calculations  we consider only automata whose input
and output alphabets consist of $p$ letters $0,1,\ldots,p-1$ where $p>1$ is
a prime number  
though our approach  can be expanded
to the case when $p$ is arbitrary integer greater than 1 (and even to
the case when $p$ is not necessarily an integer, see Section \ref{sec:Concl}). For a prime $p$,
we naturally associate when necessary letters of the alphabet $0,1,\ldots,p-1$ to residues
modulo $p$, i.e., to elements of a finite field $\F_p$.
 
The paper is organized
as follows:
\begin{itemize}
\item In Section \ref{sec:Prelm} we recall basic definitions as well as some
(mostly known) facts from combinatorics
of words, from automata theory, from $p$-adic analysis, and from knot theory. 
Also in this section we formally introduce the notion of  real
plot of  automaton and examine its basic properties.
\item In Section \ref{sec:aff} we completely describe  cluster points of
real plots of finite autonomous
automata and of finite affine
automata: We show that the points  constitute  links of torus knots.
\item In Section \ref{sec:fin-comp} we prove numerous (mostly technical)
results on finitely computable functions; that is, on real functions whose
graphs lie in plots of finite automata. Loosely speaking, in the section
we (rigorously) develop  techniques to examine real functions computed on finite automata as if the automata are feeded
by base-$p$ expansions of real arguments of the functions so that less significant
digits are feeded to  automaton prior to more significant ones.
\item Section \ref{sec:main} contains main results of the paper: We prove
that once a finitely computable function is $C^2$-smooth than it is affine
and may be associated to a finite collection of complex-valued functions $\Psi(x,\ell)=e^{i(Ax-2\pi p^\ell B)}$, $(x\in\R; \ell\in\N_0)$ for suitable
rational numbers $A,B$ which are $p$-adic integers. We prove a multivariate
version of the theorem as well.
\item In Section \ref{sec:Concl} we discuss possible connections of the main results
to informational interpretation of quantum theory. We argue  that the results
show that wave function is a mathematical consequence of two basic assumptions
which are causality principle and discreteness of matter: We show that using
$\beta$-expansions of real numbers (where $\beta=1+\tau$ and $\tau>0$ is small) rather than base-$p$
expansions for positive integer $p>1$, main results
of the paper imply that a quantum system may be considered as a finite  automaton
which calculates functions $e^{i(Ax-2\pi \beta^\ell B)}$; but the functions
are approximately equal to $a\cdot e^{i(Ax-2\pi tB)}$ when $t=\ell\tau$ since $\tau$ is small and thus $(1+\tau)^\ell\approx 1+\ell\tau$;  moreover, $\beta=1+\tau$ implies that both input and output alphabets of the automaton
must be necessarily binary, i.e., $\{0,1\}$.  Therefore one may say that the automaton produces waves
$a\cdot e^{i(Ax-2\pi tB)}$ (since variables $x,t\in\R$  may be regarded as `position' and `time' respectively)
from bits. This may serve a mathematical evidence in favour of J.~A.~Wheeler's \emph{It from bit} doctrine which suggests that all things physical (`its') are
information-theoretic in origin (`from bits'), \cite{Wheeler_IFB}.
\end{itemize}

%
%
\section{Preliminaries}
\label{sec:Prelm}
Technically the paper is a sort of interplay between real analysis and $p$-adic
analysis; but although  real analysis is the tool we mostly use in proofs, in some important places we also use $p$-adic analysis to examine specific
properties of automata maps since the maps actually are 1-Lipschitz functions
w.r.t. $p$-adic metric. 
This is why we first recall some facts about words over a finite alphabet,
$p$-adic integers, and automata.
\subsection{Few words about words}
\label{ssec:word}
An \emph{alphabet}   is just a finite non-empty set $\Cal A$; further  in the paper usually $\Cal A=\{0,1,\ldots,p-1\}=\F_p$.   Elements of $\Cal A$
 elements are called \emph{symbols},
or \emph{letters}. 
By the definition, a \emph{word of length $n$ over alphabet $\Cal A$} is a finite sequence (stretching from
right to left)
$\alpha_{n-1}\cdots\alpha_1\alpha_0$, where $\alpha_{n-1},\ldots,\alpha_1,\alpha_0\in\Cal
A$. The number $n$ is called the \emph{length} of the word $w=\alpha_{n-1}\cdots\alpha_1\alpha_0$
 and  is denoted via $\Lambda(w)$. The \emph{empty word} $\phi$ is a sequence
of length 0, that is, the one that contains no symbols. 
Given a word $w=\alpha_{n-1}\cdots\alpha_1\alpha_0$,
any word $v=\alpha_{k-1}\cdots\alpha_1\alpha_0$, $k\le n$, is called a \emph{prefix}
of the word $w$; whereas any word $u=\alpha_{n-1}\cdots\alpha_{i+1}\alpha_i$,
$0\le i\le n-1$ is called a \emph{suffix} of the word $w$. Every word $\alpha_{j}\cdots\alpha_{i+1}\alpha_i$
where $n-1\ge j\ge i\ge 0$ is called a \emph{subword} of the word $w=\alpha_{n-1}\cdots\alpha_1\alpha_0$. Given words $a=\alpha_{n-1}\cdots\alpha_1\alpha_0$
and $b=\beta_{k-1}\cdots\beta_1\beta_0$, the
\emph{concatenation} $ab$ is the following word (of length $n+k$):
\[
ab=\alpha_{n-1}\cdots\alpha_1\alpha_0\beta_{k-1}\cdots\beta_1\beta_0.
\]
Given a word $w$, its $k$-times concatenation  is denoted
via $(w)^k$:
$$
(w)^k=\underbrace{ww\ldots w}_{k\ \text{times}}.
$$
We denote via $\Cal W$ the set of all non-empty words over $\Cal A=\{0,1,\ldots,p-1\}$
and  via $\Cal W_\phi$ the set of all words including the empty word $\phi$. In the sequel the set of all $n$-letter words
over the alphabet $\F_p$ we denote as $\Cal W_n$; so $\Cal W=\cup_{n=1}^\infty
\Cal W_n$.  
To
every word $w=\alpha_{n-1}\cdots\alpha_1\alpha_0$ we put into the correspondence
a non-negative integer $\nm(w)=\alpha_0+\alpha_1\cdot p+\cdots+\alpha_{n-1}\cdot
p^{n-1}$. 
Thus $\nm$ maps the set $\Cal W$
of all non-empty finite words over the alphabet $\Cal A$ 
onto the set $\N_0=\{0,1,2,\ldots\}$ of all non-negative integers. We will
also consider a map $\rho$ of the set $\Cal W$ into the real unit half-open
interval $[0,1)$; the map $\rho$ is defined as follows: Given $w=\beta_{r-1}\ldots\beta_0\in\Cal W$, put 
\begin{equation}
\label{eq:rho}
\rho(w)=\nm(w)\cdot p^{-\Lambda(w)}=
\frac{\beta_0+\beta_1p+\cdots+\beta_{r-1}p^{r-1}}{p^r}
=0.\beta_{r-1}\ldots\beta_0\in[0,1).
\end{equation}
We also use the notation $0.w$ for $0.\beta_{r-1}\ldots\beta_0$.

Along with finite words we also consider (left-)infinite words over the alphabet
$\Cal A$; the ones are the infinite sequences of the form $\ldots\alpha_2\alpha_1\alpha_0$
where $\alpha_i\in\Cal A$, $i\in\N_0$. For infinite words the notion of a
prefix and of a subword are defined in the same way as for finite words;
whilst suffix is not defined. Let an infinite word $w$ be eventually periodic,
that is, let $w=\ldots\beta_{t-1}\beta_{t-2}\ldots\beta_0
\beta_{t-1}\beta_{t-2}\ldots\beta_0\alpha_{r-1}\alpha_{r-2}\ldots\alpha_0$
for $\alpha_i\beta_j\in\Cal A$; then the subword $\beta_{t-1}\beta_{t-2}\ldots\beta_0$
is called a \emph{period} of the word $w$ and the suffix $\alpha_{r-2}\ldots\alpha_0$
is called the pre-period of the word $w$. Note that a pre-period may be an
empty word while a period can not. We write the eventually periodic word
$w$ as $w=(\beta_{t-1}\beta_{t-2}\ldots\beta_0)^\infty\alpha_{r-1}\alpha_{r-2}\ldots\alpha_0$.
\subsection{$p$-adic numbers}
\label{ssec:p-adic}
See \cite{Gouvea:1997,Kat,Kobl} for introduction to $p$-adic analysis
or comprehensive
monographs \cite{Mah,Sch} for further reading.
 
Fix a prime number $p$ and denote respectively via $\N=\{1,2,\ldots\}$ and $\Z=\{0,\pm1,\pm
2,\ldots\}$ the set of all positive rational integers and the ring of all
rational integers. Given $n\in\N=\N_0\setminus\{0\}$, the \emph{$p$-adic
absolute value} of $n$ is $|n|_p=p^{-\ord_pn}$, where $p^{\ord_pn}$ is the
largest power of $p$ which is a factor of $n$; so $n=n^\prime \cdot p^{\ord_pn}$
where $n^\prime\in\N$ is co-prime to $p$. By putting $|0|_p=0$, $|-n|_p=|n|_p$
and $|n/m|_p=|n|_p/|m|_p$ for $n,m\in\Z$, $m\ne0$ we expand the $p$-adic
absolute value to the whole field $\Q$ of rational numbers. Given an absolute value
$|~|_p$, we  define a metric  in a standard way: $|a-b|_p$ is a \emph{$p$-adic
metric} on $\Q$.   
The field $\Q_p$ of \emph{$p$-adic
numbers} is a completion of the field $\Q$ of rational numbers w.r.t. the
$p$-adic metric while the ring $\Z_p$ of \emph{$p$-adic
integers} is a ring of integers of $\Q_p$; and the ring $\Z_p$ is a completion
of $\Z$ w.r.t. the $p$-adic metric.  The ring $\Z_p$ is compact w.r.t.
the $p$-adic metric: Actually $\Z_p$ is a ball of radius 1 centered at 0; namely $\Z_p=\{r\in\Q_p\colon |r|_p\le1\}$. Balls in $\Q_p$ are \emph{clopen};
that is, both closed and open w.r.t. the $p$-adic metric.

A $p$-adic number $r\in\Q_p\setminus\{0\}$ admits a unique
\emph{$p$-adic canonical expansion} $r=\sum_{i=k}^\infty\alpha_ip^i$ where $\alpha_i\in\{0,1,\ldots,
p-1\}$, $k\in\Z$, 
$\alpha_k\ne0$. Note that then any $p$-adic integer $z\in\Z_p$
admits a unique representation $z=\sum_{i=0}^\infty\alpha_ip^i$ for suitable
$\alpha_i\in\{0,1,\ldots,p-1\}$. 
The latter representation is
called a \emph{canonical form} (or, a \emph{canonical representation}) of the $p$-adic integer $z\in\Z_p$;
the $i$-th coefficient $\alpha_i$ of the expansion will be referred to as \emph{the $i$-th $p$-adic digit} of $z$ and denoted via $\alpha_i=\delta_i(z)$.
It is clear that once $z\in\N_0$, the $i$-th $p$-adic digit $\delta_i(z)$
of $z$ is just the $i$-th digit in  the base-$p$ expansion of $z$. Note also
that a $p$-adic integer $z\in\Z_p$ is a \emph{unity} of $\Z_p$ (i.e., has
a multiplicative inverse   $z^{-1}\in\Z_p$) if and only if $\delta_0(z)\ne 0$; so any $p$-adic number $z\in\Q_p$ has
a unique representation of the form $z=z^\prime\cdot |z|_p^{-1}$ where $z^\prime\in\Z_p$
is a unity. 

The $p$-adic integers may be associated to infinite words over the alphabet
$\F_p=\{0,1,\ldots, p-1\}$ as follows: Given a $p$-adic integer $z\in\Z_p$,
consider its canonical expansion $z=\sum_{i=0}^\infty\alpha_i\cdot p^i$;
then denote via $\wrd(z)$ the infinite word $\ldots\alpha_2\alpha_1\alpha_0$
(allowing some freedom of saying we will sometimes refer  $\wrd(z)$ as to a \emph{base-$p$ expansion of $z\in\Z_p$}). Vice versa, given a left-infinite
word $w=\ldots\alpha_2\alpha_1\alpha_0$ we denote via 
$\nm(w)=\sum_{i=0}^\infty\alpha_i\cdot p^i$ corresponding $p$-adic integer
whose base-$p$ expansion is $w$ thus expanding the mapping $\nm$ defined
in Subsection \ref{ssec:word} to the case of infinite words as well.
It is worth noticing here that addition and multiplication of $p$-adic integers
can be performed by using the same school-textbook algorithms for addition/multiplication
of non-negative integers represented via their base-$p$ expansions with the
only difference: The algorithms are applied
to infinite words that correspond to $p$-adic canonical forms of summands/multipliers
rather than to a finite words which are base-$p$ expansions
of summands/multipliers.

Given $n\in\N$ and a canonical expansion $z=\sum_{i=0}^\infty\alpha_ip^i$ for $z\in\Z_p$, denote $z\md{p^n}=\sum_{i=0}^{n-1}\alpha_ip^i$. The mapping $\md{p^n}\:z\mapsto
z\md{p^n}$ is a ring epimorphism of $\Z_p$ onto the residue ring $\Z/p^n\Z$
(under a natural representation of elements of the residue ring by the least
non-negative residues $\{0,1\ldots,p^n-1\}$).

The series in the right-hand side of the canonical form converges w.r.t.
the $p$-adic metric; that is, the sequence of partial sums $z\md{p^n}$ converges
to $z$ w.r.t. the $p$-adic metric: $\lim_{n\to\infty}^p(z\md p^n)=z$. It is worth noticing here that arbitrary infinite series $\sum_{i=0}^\infty r_i$ where
$r_i\in\Q_p$ converges in $\Q_p$ (i.e., w.r.t. $p$-adic metric) if and
only if $\lim_{i\to\infty}|r_i|_p=0$ since  $p$-adic metric is \emph{non-Archimedean};
that is, it satisfies \emph{strong triangle inequality} $|x-y|_p\le\max\{|x-z|_p,|z-y|_p\}$ for all $x,y,z\in\Q_p$.

Note that $z\in\N_0$ if and only if all but a finite number
of coefficients $\alpha_i$ in the canonical form are 0 while $z\in\{-1,-2,-3,\ldots\}$
if and only if all but a finite number of $\alpha_i$ are $p-1$.
Further we will need a special representation for \emph{$p$-adic integer
rationals}; that is, for those rational numbers $z$ which at the same time
are $p$-adic integers, i.e., for $z\in\Z_p\cap\Q$.  Note that $z\in\Z_p\cap\Q$ if and only if $z$ can be represented by an irreducible
fraction $z=a/b$, $a\in\Z, b\in\N$ where $b$ is co-prime to $p$.
The following proposition is well known, cf., e.g., \cite[Theorem 10]{Frougny_Rat-base-p}:
\begin{prop}
\label{prop:p-adic-rat}
A $p$-adic integer $z$ is rational \textup{(i.e., $z\in\Z_p\cap\Q$)}
if and only if the sequence of coefficients of its
canonical form is eventually periodic:
\begin{multline}
\label{eq:p-adic-rat}
z=\alpha_0+\alpha_1p+\cdots+\alpha_{r-1}p^{r-1}+(\beta_0+\beta_1p+\cdots+\beta_{t-1}p^{t-1})p^r+\\
(\beta_0+\beta_1p+\cdots+\beta_{t-1}p^{t-1})p^{r+t}+(\beta_0+\beta_1p+\cdots+
\beta_{t-1}p^{t-1})p^{r+2t}+\cdots
\end{multline}
for suitable $\alpha_j,\beta_i\in\{0,1,\ldots,p-1\}$, $r\in\N_0$, $t\in\N$
\textup{(the sum  $\alpha_0+\alpha_1p+\cdots+\alpha_{r-1}p^{r-1}$ is absent
in the above expression once $r=0$)}.
\end{prop}
In other words, once a $p$-adic integer $z$ is represented in its canonical
form, $z=\sum_{i+0}^\infty\gamma_ip^i$, the corresponding infinite word $\ldots\gamma_1\gamma_0$
is eventually periodic: $\ldots\gamma_1\gamma_0=(\beta_{t-1}\ldots\beta_{0})^\infty\alpha_{r-1}\ldots\alpha_{0}$.
It is clear that given $z\in\Z_p\cap\Q$, both $r$ and  $t$ are not unique:
For instance, 
$$
(\beta_{t-1}\ldots\beta_{0})^\infty\alpha_{r-1}\ldots\alpha_{0}=
(\beta_0\beta_{t-1}\ldots\beta_1\beta_0\beta_{t-1}\ldots\beta_1)^\infty\alpha_{r}\alpha_{r-1}\ldots\alpha_{0},
$$
where $\alpha_r=\beta_0$. But once both pre-periodic and periodic parts 
(the prefix $\alpha_{r-1}\ldots\alpha_{0}$ and the word $\beta_{t-1}\ldots\beta_{0}$
) are taken the shortest possible, both the \emph{pre-period length} $r$ and
the \emph{period length} $t$ are unique for a given $p$-adic rational integer
$z\in\Z_p\cap\Q$; we refer to $\alpha_{r-1}\ldots\alpha_{0}$ and to 
$\beta_{t-1}\beta_{t-2}\ldots\beta_{1}\beta_0$ as to \emph{pre-period} of $z$ and
\emph{period} of $z$ accordingly.

Given $z\in\Z_p\cap\Q$ we mostly  assume further that in the representation 
$z=\alpha_0+\cdots+\alpha_{r-1}p^{r-1}+(\beta_0+\cdots+\beta_{t-1}p^{t-1})\cdot\sum_{j=0}^\infty
p^{r+tj}$ (respectively, in eventually periodic infinite word 
$\wrd(z)=(\beta_{t-1}\ldots\beta_{0})^\infty\alpha_{r-1}\ldots\alpha_{0}$ that corresponds
to $z$) $r$ is a pre-period length and $t$ is a period length. Note that
a pre-period may be an empty word (i.e., of length 0) while a period can not.

Rational $p$-adic integers can also be represented as fractions of a special kind:
\begin{prop}
\label{prop:p-repr}
A $p$-adic integer $z\in\Z_p$ is rational  if and
only if there exist $t\in\N$, $c\in\Z$, $d\in\{0,1,\ldots,p^t-2\}$ 
such
that 
\begin{equation}
\label{eq:pq-frac-rep}
z=c+\frac{d}{p^t-1}.
\end{equation}
\end{prop} 
\begin{proof}
Indeed, $z\in\Z_p\cap\Q$ if and only if $z$ is of the form \eqref{eq:p-adic-rat};
therefore
\begin{multline}
\label{eq:p-repr}
z=
(\alpha_0+\alpha_1p+\cdots+\alpha_{r-1}p^{r-1}-p^r)+p^r
\left(1-\frac{\beta_0+\beta_1p+\cdots+\beta_{t-1}p^{t-1}}{p^t-1}\right)=\\
(\alpha_0+\alpha_1p+\cdots+\alpha_{r-1}p^{r-1} -p^r+q)+
\frac{\zeta_0+\zeta_1p+\cdots+\zeta_{t-1}p^{t-1}}{p^t-1}
\end{multline}
where $\zeta_0+\zeta_1p+\cdots+\zeta_{t-1}p^{t-1}$ is a base-$p$ expansion
of  the least non-negative
residue $s$ of $p^r(p^t-1-(\beta_0+\beta_1p+\cdots+\beta_{t-1}p^{t-1}))=(p^t-1)q+s$ modulo $p^t-1$.  
\end{proof}
\begin{note}
\label{note:inverse}
Recall that
$(1-p^m)^{-1}=\sum_{i=0}^\infty p^{mi}\in \Z_p$, for every $m\in\N$.
\end{note}
\begin{note}
\label{note:pq-frac-rep}
Note that once in \eqref{eq:p-repr} $r$ is a pre-period length and $t$ is
a period length of $z\in\Z_p\cap\Q$, the representation \eqref{eq:pq-frac-rep} is unique; that is, the  choice of $c$ and $d$ in \eqref{eq:pq-frac-rep}
is unique.
\end{note}
In the sequel we often use base-$p$ expansions  of $p$-adic rational
integers reduced modulo 1 (recall that if $y\in\R$ then by the definition
$y\md 1=y-\lfloor y\rfloor\in[0,1)\subset\R$)
along with their $p$-adic canonical forms. For reader's convenience, we now summarize some facts on connections
between  these representations.

It is very well known that a base-$p$ expansion of a rational number is eventually
periodic; that is, given $x\in\Q\cap[0,1]$, the base-$p$ expansion for $x$ is
\begin{multline}
\label{eq:r-repr}
x=0.\chi_0\ldots\chi_{k-1}(\xi_0\ldots\xi_{n-1})^\infty=\\
\chi_0p^{-1}+\chi_1p^{-2}+\cdots+\chi_{k-1}p^{-k}+
\xi_0p^{-k-1}+\xi_1p^{-k-2}+\cdots+\xi_{n-1}p^{-k-n}+\\
\xi_0p^{-k-1-n}+\xi_1p^{-k-2-n}+\cdots+\xi_{n-1}p^{-k-2n}+\cdots=\\
\frac{1}{p^k}(\chi_0p^{k-1}+\chi_1p^{k-2}+\cdots+\chi_{k-1})+
\frac{1}{p^k}\cdot\frac{\xi_0p^{n-1}+\xi_1p^{n-2}+\cdots+\xi_{n-1}}{p^n-1},
\end{multline}
where $\chi_i,\xi_j\in\{0,1,\ldots,p-1\}$. Note that in the base-$p$ expansions
of rational integers from $[0,1]$ we use \emph{right}-infinite words rather
than left-infinite ones that correspond to canonical expansions 
of $p$-adic
integers. 
\begin{prop}
\label{prop:r-p-repr-qz}
Given $z\in\Z_p\cap\Q$, represent $z$ in the form \eqref{eq:p-adic-rat};
then
$$z\md1=0.(\hat\beta_{t-1-\bar r}\hat\beta_{t-2-\bar r}\ldots
\hat\beta_{0}\hat\beta_{t-1}\hat\beta_{t-2}\ldots\hat\beta_{t-\bar
r})^\infty\md1,$$
where $\hat\beta=p-1-\beta$ for  $\beta\in\{0,1,\ldots,p-1\}$
and $\bar r$ is the least non-negative residue of $r$ modulo $t$ if $t>1$
or $\bar r=0$ if otherwise.
\end{prop}
\begin{proof}
Indeed, by Note \ref{note:inverse}, $\sum_{j=0}^\infty p^{r+tj}=-p^r(p^t-1)^{-1}$ in $\Z_p$;
so $z=u-vp^r(p^t-1)^{-1}$ where $u=\alpha_0+\alpha_1p+\cdots+\alpha_{r-1}p^{r-1}$
and $v=\beta_0+\beta_1p+\cdots+\beta_{t-1}p^{t-1}$. Therefore 
$$
z\md1=\left(-\frac{vp^r
}{p^t-1}\right)\md1.
$$ 
But $(p^t-1)^{-1}=p^{-t}+p^{-2t}+p^{-3t}+\cdots$ in $\R$; so 
$$
(p^t-1)^{-1}=0.(\underbrace{00\ldots0}_{t-1}1)^\infty
$$
and thus 
$-v\cdot(p^t-1)^{-1}=-0.(\beta_{t-1}\beta_{t-2}\ldots\beta_{0})^\infty$.

Now just note that 
$$
(p-1-\gamma_0)+(p-1-\gamma_1)p+\cdots+(p-1-\gamma_{s-1})p^{s-1}=p^s-1-
(\gamma_0+\gamma_1p+\cdots+\gamma_{s-1}p^{s-1})
$$
for $\gamma_0, \gamma_1,\ldots\in\{0,1,\ldots,p-1\}$, $s\in\N$;
so 
\begin{multline*}
\frac{(p-1-\gamma_0)+(p-1-\gamma_1)p+\cdots+(p-1-\gamma_{s-1})p^{s-1}}{p^s-1}=\\
1-\frac{\gamma_0+\gamma_1p+\cdots+\gamma_{s-1}p^{s-1}}{p^s-1}
\end{multline*}
and therefore 
\begin{equation}
\label{eq:-md1}
(-0.(\gamma_{s-1}\gamma_{s-2}\ldots\gamma_0)^\infty)\md1=
(0.(\hat\gamma_{s-1}\hat\gamma_{s-2}\ldots\hat\gamma_0)^\infty)\md1
\end{equation}
where $\hat\gamma=p-1-\gamma$ for $\gamma\in\{0,1,\ldots,p-1\}$.
\end{proof} 
Combining
\eqref{eq:r-repr} with Proposition \ref{prop:p-repr} we see that   all real
numbers whose base-$p$ expansions are purely periodic must lie in $\Z_p\cap\Q$;
therefore the following criterion is true:
\begin{cor}
\label{cor:r-repr}
A real number $x$ is in $\Z_p\cap\Q$ if and only if  base-$p$ expansion of
$x\md1$ is purely periodic: $x\md1=0.(\chi_0\ldots\chi_{n-1})^\infty$ for
suitable $\chi_0,\ldots,\chi_{n-1}\in\F_p$.
\end{cor}
The following corollary expresses base-$p$ expansion of a $p$-adic rational
integer via its representation in the form given by Proposition \ref{prop:p-repr}:
\begin{cor}
\label{cor:r-p-repr-qz}
Once a $p$-adic rational integer $z\in\Z_p\cap\Q$ is represented in the form
as of Proposition \ref{prop:p-repr} then 
$z\md1=0.(\zeta_{t-1}\zeta_{t-2}\ldots\zeta_{0})^\infty$ where
$d=\zeta_0+\zeta_1p+\cdots+\zeta_{t-1}p^{t-1}$. 
\end{cor}
\begin{proof}
Indeed, under notation of Proposition \ref{prop:p-repr}, $z\md1=(d\cdot(p^t-1)^{-1})\md
1$
and the result follows since $(p^t-1)^{-1}=p^{-t}+p^{-2t}+p^{-3t}+\cdots$ in $\R$. 
\end{proof}
Now we can find a period length of $z\in\Z_p\cap\Q$ provided $z$ is represented
as an irreducible fraction $z=a/b$, where $a\in\Z$, $b\in\N$.
\begin{prop}
\label{prop:mult-per}
Once a $p$-adic rational integer $z\ne 0$ is represented as an irreducible
fraction $z=a/b$, and if  $b>1$, then the period length $t$ of $z$ is equal to the multiplicative
order of $p$ modulo $b$ \textup{(i.e., to the smallest $\ell\in\N$ such that
$p^\ell\equiv 1\pmod b$)}. 
\end{prop}
\begin{proof}
Note that the multiplicative order $\ell$ of $p$ modulo $b$ is the smallest
positive integer such that $p^\ell(a/b)\equiv a/b\pmod 1$.
Indeed, $p^\ell=eb+1$ for a suitable $e\in \Z$; so $p^\ell (a/b)=ea+(a/b)$.
On the other hand, if $p^s(a/b)=m+(a/b)$ for some $m\in\Z$ then $a(p^s-1)=mb$
and thus $p^s-1\equiv 0\pmod b$ since $a$ is co-prime to $b$ (as the fraction
$a/b$ is supposed to be irreducible). 

Now, from Corollary \ref{cor:r-p-repr-qz} it immediately follows that $(p^t z)\md1=z\md1$
once $t$ is a period length of $z$ and that $t$ is the smallest positive
integer with that property. Finally we conclude that $\ell=t$. 
\end{proof}
Now given $b\in\N$, $b$ co-prime to $p$, we denote via $\mlt_bp$ the multiplicative
order of $p$ modulo $b$  if $b>1$ or put $\mlt_bp=1$ once $b=1$. Then $\mlt_bp$
is the period length of $z\in\Z_p\cap\Q$ once $z$ is represented as an irreducible
fraction $z=a/b$ where $a\in\Z$ and $b\in\N$. Note that we consider here
only infinite words that correspond to $p$-adic rational integers; thus to,
e.g.,  0 there corresponds a word $(0)^\infty$ (so a period of 0 is 0 and a pre-period
is empty) and  the respective base-$p$ expansion
of 0 is $0.(0)^\infty$. Also, $1=1+0\cdot p+0\cdot p^2+\cdots$, the corresponding
infinite word is $(0)^\infty1$; therefore
$1$ is a pre-period of 1, $0$ is a period of 1, and the representation of 1 in the form \eqref{eq:pq-frac-rep} is $1=1+(0/p-1)$.
\begin{exmp}
Let $p=2$; then $1/3=1\cdot 1+1\cdot 2+0\cdot 4+1\cdot8+0\cdot16+\cdots=1-2\cdot
3^{-1}$ is a canonical 2-adic expansion of $1/3$; so  the corresponding infinite binary word is $(01)^\infty1$.
Therefore the period length of $1/3$ is 2 (and note that the multiplicative order of $2$ modulo $3$ is indeed $2$), the period is $01$, the pre-period
is $1$. Also,  $c=0$ and $d=1$ once $1/3$ is represented in the form of
Proposition \ref{prop:p-repr}; $1/3=0.(01)^\infty$ is a base-2 expansion
of $1/3$, cf. Proposition \ref{prop:r-p-repr-qz} and Corollary \ref{cor:r-p-repr-qz}.

\end{exmp}

\subsection{Automata: Basics}
\label{ssec:auto}
Here we remind some basic facts from automata theory (see e.g. monographs
\cite{Bra,Car-Long_Automata,Eilenberg_Auto}). 

By the definition, a (non-initial) automaton is a 5-tuple $\mathfrak A=\langle\Cal I,\Cal S,\Cal O,S,O\rangle$
where $\Cal I$ is a finite set, the \emph{input alphabet}; $\Cal O$  is a finite set, the \emph{output alphabet}; $\Cal S$ is a non-empty (possibly, infinite) set of \emph{states}; 
$S\colon\Cal I\times\Cal S\to \Cal S$ is  a \emph{state transition function}; $O\colon\Cal I\times\Cal S\to \Cal O$ is an \emph{output function}.  An automaton
where both input alphabet $\Cal I$ and output alphabet $\Cal O$ are non-empty
is called a \emph{transducer}, see e.g.
\cite{Allouche-Shall,Car-Long_Automata}.
The \emph{initial automaton} $\mathfrak A(s_0)=\langle\Cal I,\Cal S,\Cal O,S,O, s_0\rangle$ is an automaton
$\mathfrak A$ where one state $s_0\in\Cal S$ is fixed; it is called the \emph{initial
state}.  We
stress that the definition of  an initial automaton $\mathfrak A(s_0)$ is nearly the same as the one of
\emph{Mealy automaton} (see e.g. \cite{Bra,Car-Long_Automata})
with the only important difference: 
the set of states
$\Cal S$ of $\mathfrak A(s_0)$ is \emph{not necessarily finite}. Note also
that in literature the automata we consider in the paper are also referred
to as \emph{(letter-to-letter) transducers};  in the sequel we use terms `automaton'
and `transducer' as synonyms. 

Given an input word $w=\chi_{n-1}\cdots\chi_1\chi_0$ over the alphabet $\Cal
I$, an initial transducer $\mathfrak A(s_0)=\langle\Cal I,\Cal S,\Cal O,S,O, s_0\rangle$
transforms $w$ to output word $w^\prime=\xi_{n-1}\cdots\xi_1\xi_0$ over the
output alphabet $\Cal O$ as follows (cf. Figure \ref{fig:Transd-sc}): 
Initially the transducer  $\mathfrak A(s_0)$ is at the state
$s_0$; accepting the input symbol $\chi_0\in\Cal I$, the transducer outputs the
symbol $\xi_0=O(\chi_0,s_o)\in\Cal O$ and reaches the state
$s_1=S(\chi_0,s_0)\in\Cal S$; then the transducer accepts the next input symbol
$\chi_1\in\Cal I$, reaches the state
$s_2=S(\chi_1,s_1)\in\Cal S$, outputs $\xi_1=O(\chi_1,s_1)\in\Cal O$, and the routine repeats. This way the transducer $\mathfrak A=\mathfrak A(s_0)$
defines a mapping $\mathfrak a=\mathfrak a_{s_0}$  of the set $\Cal W_n(\Cal I)$ of all $n$-letter words over
the input alphabet $\Cal I$ to the set  $\Cal W_n(\Cal O)$ of all $n$-letter words over
the output alphabet $\Cal O$; thus $\mathfrak A$ defines a map of the set
$\Cal W(\Cal I)$ of all  non-empty words over the alphabet $\Cal I$ to the
set $\Cal W(\Cal O)$ of all  non-empty words over the alphabet $\Cal O$.
We will denote the latter map by the same symbol $\mathfrak a$ (or by $\mathfrak a_{s_0}$ if we want to stress what initial state is meant), and when it is
clear from the context what alphabet $\Cal A$ is meant we use notation
$\Cal W$ rather than $\Cal W(\Cal A)$.

\begin{figure}[h]
\begin{quote}\psset{unit=0.5cm}
 \begin{pspicture}(1,2.5)(20,13)
\pscircle[fillstyle=solid,
linewidth=2pt](12,10){1}
\pscircle[fillstyle=solid,
linewidth=2pt](12,4){1}
  \psline(12,11)(12,12)
  \psline{->}(18,7)(13,7)
  \psline{->}(6,10)(11,10)
  \psline{->}(6,4)(11,4)
  \psline{->}(4,7)(6,7)
  \psline(6,4)(6,10)
  \psline(18,12)(18,7)
  \psline(12,12)(18,12)
  \psline{->}(12,8)(12,9)
  \psline(12,3)(12,2)
  \psline{->}(12,2)(18,2)
  \psline{<-}(12,5)(12,6)
  \psframe[fillstyle=solid,
  linewidth=2pt](11,6)(13,8)
  \uput{0}[180](12.4,7){$s_i$}
  \uput{0}[180](3.5,7){$\cdots\chi_{i+1}\chi_i$}
  \uput{0}[90](12,9.7){$S$}
  \uput{1}[90](12,2.7){$O$}
  \uput{1}[0](11.7,11.3){$s_{i+1}=S(\chi_i,s_i)$}
 \uput{1}[0](11.7,12.7){{\sf state transition}}
 \uput{1}[0](0,8){{\sf input}}
  \uput{1}[0](11.5,2.5){$\xi_i=O(\chi_i,s_i)$}
  \uput{1}[0](17.5,2){$\xi_i\xi_{i-1}\cdots\xi_0$}
  \uput{1}[0](18,3){{\sf output}}
 \end{pspicture}
\end{quote}
\caption{Initial transducer, schematically}
\label{fig:Transd-sc}
\end{figure}
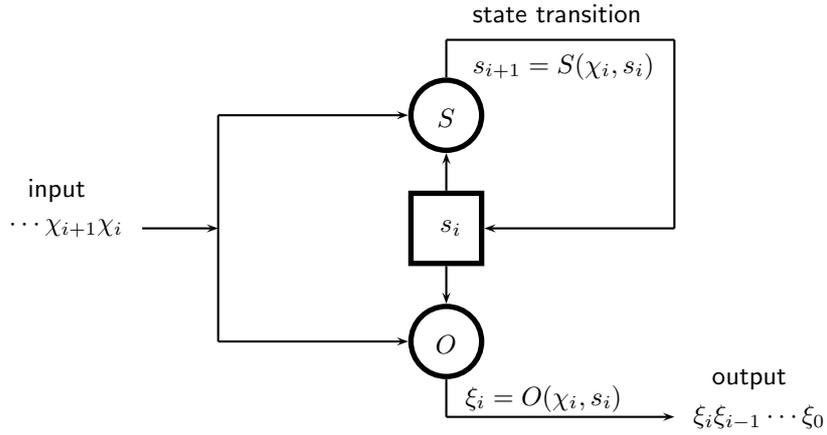
Throughout the paper, `automaton' mostly stands for `initial automaton';
 we make  corresponding remarks if not. Further in the paper we mostly
 consider transducers. Furthermore, throughout the paper  we
 consider \emph{reachable transducers} only; that is, \emph{we assume 
that all states of the initial transducer $\mathfrak A(s_0)$  are 
\emph{reachable
from} the initial state $s_0$}: 
Given $s\in\Cal S$, there exists
input word $w$ over alphabet $\Cal I$ such that after the word $w$ has been
feeded to the automaton $\mathfrak A(s_0)$, the automaton reaches the state $s$. A reachable transducer is called \emph{finite}\ if its set $\Cal S$ of
states is finite, and  transducer is called \emph{infinite} if otherwise.

To  
 the initial automaton
$\mathfrak A(s_0)$ we put into a correspondence a family $\Cal F(\mathfrak
A)$ of all \emph{sub-automata} $\mathfrak
A(s)=\langle\Cal I,\tilde{\Cal S},\Cal O,\tilde S,\tilde O, s\rangle$, $s\in\Cal S$, where
$\tilde{\Cal S}=\tilde{\Cal S}(s)\subset\Cal S$ is the set of all states that are reachable from the state
$s$ and $\tilde S,\tilde O$ are respective restrictions of the state transition and
output functions $S, O$ on $\Cal I\times \tilde{\Cal S}$. A sub-automaton
$\mathfrak A(s)$ is called \emph{proper} if  the set $\tilde{\Cal S}$ of all its states is a proper subset
of $\Cal S$. A sub-automaton 
$\mathfrak A(s)$ 
is called \emph{minimal} if it contains no proper sub-automata. It is obvious that a finite sub-automaton is minimal if and only if every its state is reachable from any other its state. The set of all states of a minimal sub-automaton
of the automaton $\mathfrak A$ is called an \emph{ergodic component} of the (set of all states) of the automaton $\mathfrak A$. It is clear that once
the automaton is in a state that belongs to an ergodic component, all its
further states will also be in the same ergodic component. Therefore all
states of a finite automaton  are of two types only: The \emph{transient states}
which belong to no ergodic component, and \emph{ergodic states} which belong
to ergodic components. It is clear that the set of all ergodic states is
a disjoint union of ergodic components. Note that we use the term `minimal
automaton'  in a different meaning  compared to the one used in automata
theory, see, e.g., \cite{Eilenberg_Auto}: Our terminology here is from the
theory of
Markov chains, see, e.g., \cite{Kem-Snell_FinMaCh} (since to the graph of
state transitions of every automaton
there corresponds a Markov chain).

Hereinafter in the paper the word `automaton' stands for a letter-to-letter initial transducer whose input and output alphabet consists of $p$ symbols,
and we mostly assume that $p$ is
a prime. 
 Thus, for every $n=1,2,3,\ldots$ the automaton
$\mathfrak A(s_0)=\langle\F_p,\Cal S,\F_p,S,O,s_0\rangle$ maps $n$-letter words over $\F_p$ to $n$-letter words over
$\F_p$ according to the procedure described above,
cf. Figure \ref{fig:Transd-sc}. Given two such automata $\mathfrak A=\mathfrak
A(s_0)$ and $\mathfrak B=\mathfrak B(t_0)$, their  \emph{sequential
composition} (or briefly, a \emph{composition}) $\mathfrak C=\mathfrak B\circ\mathfrak
A$ can be defined in a natural way via sending output of the automaton
$\mathfrak A$ to input of the automaton $\mathfrak B$  so that  the mapping $\mathfrak c\:\Cal W\to\Cal W$
the automaton $\mathfrak C$ performs is just a composite mapping $\mathfrak
b\circ\mathfrak a$ (cf. any of  monographs \cite{Bra,Car-Long_Automata,Eilenberg_Auto}
for exact definition and further facts mentioned in the subsection). Note that a composition of finite automata
is a finite automaton.  

In a similar manner one can consider automata with multiply inputs/outputs;
these can be also treated as automata whose input/output alphabets are Cartesian
powers of $\F_p$: For instance, and automaton with $m$ inputs and $n$ outputs
over alphabet $\F_p$ can be considered as an automaton with a single input
over the alphabet $\F_p^m$ and a single output over the alphabet $\F_p^n$.
Moreover, as the letters of the alphabet $\F_p^k$ are in a one-to-one correspondence
with residues modulo $p^k$; the automaton with $m$ inputs and $n$ outputs
can be considered (if necessary)  as an automaton with a single input over
the alphabet $\Z/p^m\Z$ and a single output over alphabet $\Z/p^n\Z$. 

is an automaton with 2 inputs and 1 output which   


Compositions  of automata with multiple inputs/outputs can also be naturally defined: For instance, given automata $\mathfrak A_1$, $\mathfrak A_2$, and $\mathfrak
A_3$ with $m_1,m_2,m_3$ inputs and $n_1,n_2,n_3$ outputs respectively, in
the case when $m_3=n_1+n_2$ one can consider a composition of these automata
by connecting every output of automata $\mathfrak A_1$ and $\mathfrak A_2$
to some input of the automaton $\mathfrak A_3$ so that  every input of the
automaton $\mathfrak A_3$ is connected to a unique output which belongs either
to $\mathfrak A_1$ or to $\mathfrak A_2$ but not to the both. This way one
obtains various compositions of automata   $\mathfrak A_1$ and $\mathfrak A_2$, with the automaton $\mathfrak
A_3$, and either of these compositions is an automaton with $m_1+m_2$ inputs
and $n_3$ outputs. Moreover, either of the compositions is a finite automaton
if all three automata  $\mathfrak A_1$, $\mathfrak A_2$,  $\mathfrak
A_3$ are finite.


Automata can be considered as  (generally) non-autonomous dynamical systems
on different configuration spaces (e.g., $\Cal W_n$, $\Cal W$, etc.);
the system is autonomous  when neither the state transition function $\Cal S$ nor
the output function $\Cal O$ depend on input; in this case the automaton
is called \emph{autonomous} as well. For purposes of the paper it is convenient
to consider automata with input/output alphabets $\Cal A=\F_p$ as dynamical systems on the space $\Z_p$ of $p$-adic integers, i.e., to relate an automaton
$\mathfrak A$
to a special map $f_\mathfrak A\:\Z_p\to\Z_p$. In the next subsection we recall some facts about the map $f_\mathfrak A$.

\subsection{Automata maps: the $p$-adic view}
\label{ssec:a-map}
We identify $n$-letter words over $\F_p$
with non-negative integers in a natural way: Given an $n$-letter word 
$w=\chi_{n-1}\chi_{n-2}\cdots\chi_0$ (i.e., $\chi_i\in\F_p$ for $i=0,1,2,\ldots,n-1$),
we consider $w$ as a base-$p$ expansion of the number $\nm(w)=\chi_0+\chi_1\cdot
p+\cdots+\chi_{n-1}\cdot p^{n-1}\in\N_0$. In turn, the latter number can be considered
as an element of the residue ring $\Z/p^n\Z$ modulo $p^n$. We denote via $\wrd_n$ an inverse mapping to $\nm$. The mapping $\wrd_n$ is a bijection
of the set $\{0,1\ldots,p^n-1\}\subset\N_0$ onto the set $\Cal W_n$ of all
$n$-letter words over $\F_p$. 

As the set $\{0,1\ldots,p^n-1\}$ is the set of all non-negative residues
modulo $p^n$ , to \emph{every
automaton $\mathfrak A=\mathfrak A(s)$ there corresponds a map $f_{n,\mathfrak A}$ from $\Z/p^n\Z$ to $\Z/p^n\Z$}, for every $n=1,2,3,\ldots$. Namely, for
$r\in\Z/p^n\Z$ put  $f_{n,\mathfrak A}(r)=\mathfrak
\nm(\mathfrak a(\wrd_n(r)))$,  where $\mathfrak a$ is a word transformation of $\Cal
W_n$ performed by the automaton $\mathfrak A$, cf. Subsection \ref{ssec:auto}. 

Speaking
less formally, the mapping $f_{n,\mathfrak A}$ can be defined as follows: given $r\in\{0,1,\ldots,p^n-1\}$, consider a base-$p$ expansion of
$r$, read it as a $n$-letter word over $\F_p=\{0,1,\ldots,p-1\}$ (put additional
zeroes
on higher order positions if necessary) and then feed the word to the automaton
so that letters that are on lower order positions (`less significant digits')
are feeded prior to
ones on higher order positions (`more significant digits'). Then read the corresponding output $n$-letter
word as a base-$p$ expansion of a number from $\N_0$ keeping  the same order, i.e.
when the earliest outputted
letters correspond to lowest order digits in the base-$p$ expansion. 

We stress the following determinative property of the mapping $f_{n,\mathfrak A}$ which follows directly from the definition: Given $a,b\in\{0,1,\ldots,p^n-1\}$, \emph{whenever $a\equiv b\pmod{p^k}$ for
some $k\in\N$ then necessarily $f_{n,\mathfrak A}(a)\equiv f_{n,\mathfrak A}(b)\pmod{p^k}$}. This implication may be re-stated in  terms of $p$-adic metric as follows:
\begin{equation}
\label{eq:n-compat}
|f_{n,\mathfrak A}(a)-f_{n,\mathfrak A}(b)|_p\le|a-b|_p.
\end{equation}

Furthermost, \emph{every automaton $\mathfrak A=\mathfrak A(s_0)$ defines  a mapping $f_{\mathfrak A}$ from  $\Z_p$ to $\Z_p$} which can be specified in a manner similar to the one of the mapping $f_{n,\mathfrak A}$:
Given an infinite word $w=\ldots\chi_{n-1}\chi_{n-2}\cdots\chi_0$ (that is,
an infinite  sequence) over $\F_p$ we consider
a $p$-adic integer whose $p$-adic canonical expansion is $z=z(w)=\chi_0+\chi_1\cdot
p+\cdots+\chi_{n-1}\cdot p^{n-1}+\cdots$; so, by the definition,  for every $z\in\Z_p$ we put
\begin{equation}
\label{eq:auto-f-def}
\delta_i(f_{\mathfrak A}(z))=O(\delta_i(z),s_i)\qquad  (i=0,1,2,\ldots),
\end{equation}
where $s_i=S(\delta_{i-1}(z),s_{i-1})$, $i=1,2,\ldots$, and $\delta_i(z)$
is the
$i$-th $p$-adic digit of $z$; that is,  the $i$-th term coefficient
in the $p$-adic canonical representation of $z$: $\delta_i(z)=\chi_i\in\F_p$,
$i=0,1,2,\ldots$ (see Subsection \ref{ssec:p-adic}).
The so defined map $f_{\mathfrak A}$ is called the \emph{automaton function}\index{automata function}
(or, the \emph{automaton map})
of the automaton $\mathfrak A$. Note that from \eqref{eq:auto-f-def} it follows
that
\begin{equation}
\label{eq:f-coord}
\delta_i(f_\mathfrak A(z))=\Phi_i(\delta_0(z),\ldots,\delta_i(z)),
\end{equation}
where $\Phi_i$ is a map  from the $(i+1)$-th Cartesian power $\F_p^{i+1}$ of $\F_p$
into  $\F_p$.

More formally,  given $z\in\Z_p$, define $f_{\mathfrak A}(z)$ as follows:
Consider a sequence $(z\md p^n)_{n=1}^\infty$ and a corresponding sequence
$(f_{n,\mathfrak A}(z\md p^n))_{n=1}^\infty$; then, as the sequence $(z\md p^n)_{n=1}^\infty$ converges
to $z$ w.r.t. $p$-adic metric (cf. Subsection \ref{ssec:p-adic}), the sequence $(f_{n,\mathfrak A}(z\md p^n))_{n=1}^\infty$
in view \eqref{eq:n-compat} also converges w.r.t. the $p$-adic metric (since
the latter sequence is fundamental and $\Z_p$ is closed in $\Q_p$ which is a complete metric
space). Now we just put $f_{\mathfrak A}(z)$ to be a limit point of the sequence
$(f_{n,\mathfrak A}(z\md p^n))_{n=1}^\infty$. Thus, the mapping $f_{\mathfrak
A}$ is a well-defined  function with domain  $\Z_p$ and values in $\Z_p$; by \eqref{eq:n-compat}
the function $f_{\mathfrak
A}$
satisfies Lipschitz condition with a constant 1
w.r.t. $p$-adic metric.


The point is that \emph{the class of all automata functions that correspond to
automata with $p$-letter input/output alphabets coincides with the
class of all  maps from $\Z_p$ to $\Z_p$ that satisfy the $p$-adic Lipschitz condition
with a constant 1} (the 1-Lipschitz maps, for brevity), cf., e.g., \cite{me-auto_fin}.
We note that the claim can also be derived from a more
general result on asynchronous automata \cite[Proposition 3.7]{Grigorch_auto}; for $p=2$ the claim was proved in  \cite{Vuillem_circ}. 

Further we need more detailed information about \emph{finite automata functions},
that is, about functions $f_{\mathfrak A}\:\Z_p\>\Z_p$ where $\mathfrak A=\mathfrak
A(s_0)$
is a finite automaton (i.e., with a finite set $\Cal S$ of states). It is
well known (cf. previous subsection \ref{ssec:auto}) that \emph{the class of finite automata functions is
closed w.r.t. composition of functions and a sum of functions}: Once $f,g\:\Z_p\>\Z_p$
are finite automata functions, either of mappings $x\mapsto f(g(x))$ and $x\mapsto
f(x)+g(x)$ $(x\in\Z_p)$ is a finite automaton function. Another important
property of finite automata functions is that \emph{any finite automaton
function maps $\Z_p\cap\Q$ into itself}. In view of \eqref{eq:p-adic-rat}, the latter property is just a re-statement of a
a well-known property of finite automata which yields that any finite automaton
feeded by an eventually periodic  sequence outputs
an eventually periodic sequence, cf., e.g., \cite[Corollary 2.6.9]{Bra},
\cite[Chapter XIII, Theorem 2.2.]{Eilenberg_Auto}.
Since further we often use that property of finite automata, we state it
as a lemma for future references:
\begin{lem}
\label{le:per}
If a finite automaton $\mathfrak A$ is  being feeded
by a left-infinite periodic word $w^\infty$, where  $w\in\Cal W$ is a
finite non-empty word, then  the corresponding
output left-infinite word is  eventually periodic; i.e., it is of the form
$u^\infty v$, where $u\in\Cal W$, $v\in\Cal W_\phi$. To put it in other words, if a
finite automaton is being feeded by an eventually periodic finite
word $(w)^kt$, where  $w\in\Cal W$, $t\in\Cal W_\phi$, and $k\in\N$ is sufficiently large, then the output word is of the form $r(u)^\ell v$, where $\ell\in\N$, $u\in\Cal
W$, $r,v\in\Cal W_\phi$ and $r$ is either empty or a prefix of $u$: $u=hr$
for a suitable $h\in\Cal W_\phi$. Therefore the output word is of the form
$(\bar u)^\ell v^\prime$, where $\bar u$ is a cyclically shifted word $u$.
\end{lem}
To study finite automata functions it is convenient sometimes to represent  1-Lipschitz maps from $\Z_p$ to $\Z_p$ as  special
convergent $p$-adic series, the \emph{van der Put series}. 
Details about the latter series may be found in, e.g., \cite{Mah,Sch};
here we only briefly recall some basic facts.
Given a continuous 
function 
$f\: \Z_p\rightarrow \Z_p$, 
there exists 
a unique sequence 
$B_0,B_1,B_2, \ldots $ 
of $p$-adic integers such that 

\begin{equation}
\label{vdp}
f(z)=\sum_{m=0}^{\infty}
B_m \chi(m,z) 
\end{equation}
for all $z \in \Z_p$, where 
\begin{displaymath}
\chi(m,z)=\left\{ \begin{array}{cl}
1, &\text{if}
\ \left|z-m \right|_p \leq p^{-n} \\
0, & \text{otherwise}
\end{array} \right.
\end{displaymath}
and $n=1$ if $m=0$; $n$ is uniquely defined by the inequality 
$p^{n-1}\leq m \leq p^n-1$ otherwise. The right side series in \eqref{vdp} is called the \emph{van der Put series} of the function $f$. Note that
the sequence $B_0, B_1,\ldots,B_m,\ldots$ of \emph{van der Put coefficients} of
the function $f$ tends $p$-adically to $0$ as $m\to\infty$, and the series
converges uniformly on $\Z_p$. Vice versa, if a sequence $B_0, B_1,\ldots,B_m,\ldots$
of $p$-adic integers tends $p$-adically to $0$ as $m\to\infty$, then the the series in the right
part of \eqref{vdp} converges uniformly on $\Z_p$ and thus define a continuous
function $f\colon \Z_p\to\Z_p$.

The number $n$ in the definition of $\chi(m,z)$ has a very natural meaning;
it is just the number of digits in a base-$p$ expansion of $m\in\N_0$: 
\[
\left\lfloor  \log_p m \right\rfloor 
=
\left(\text{the number of digits in a base-} p \;\text{expansion for} \;m\right)-1;
\]
therefore $n=\left\lfloor  \log_p m \right\rfloor+1$ for all $m\in\N_0$ (that
is why we assume  $\left\lfloor  \log_p 0 \right\rfloor=0$). 

 Note that 
coefficients $B_m$ are
related to the values of the function $f$ in the following way:
Let 
$m=m_0+ \ldots +m_{n-2} p^{n-2}+m_{n-1} p^{n-1}$ be a base-$p$ expansion
for $m$, i.e., 
 $ m_j\in \left\{0,\ldots ,p-1\right\}$, $j=0,1,\ldots,n-1$ and $m_{n-1}\neq 0$, then
\begin{equation}
\label{eq:vdp-coef}
B_m=
\begin{cases}
f(m)-f(m-m_{n-1} p^{n-1}),\ &\text{if}\ m\geq p;\\
f(m),\ &\text{if otherwise}.
 
\end{cases}
\end{equation}
It worth noticing  also that $\chi (m,z)$ is merely  a characteristic function of the ball $\mathbf B_{p^{-\left\lfloor  \log_p m \right\rfloor-1}}(m)=m+p^{\left\lfloor  \log_p m \right\rfloor-1}\Z_p$
of radius $p^{-\left\lfloor  \log_p m \right\rfloor-1}$ centered at $m\in\N_0$:
\begin{equation}
\label{eq:chi}
\chi(m,z)=\begin{cases}
1,\ &\text{if}\ z\equiv m \pmod{p^{\left\lfloor  \log_p m \right\rfloor+1}};\\
0,\ &\text{if otherwise}
 
\end{cases}
 =
\begin{cases}
1,\ &\text{if}\ x\in\mathbf B_{p^{-\left\lfloor  \log_p m \right\rfloor-1}}(m);\\
0,\ &\text{if otherwise}
 
\end{cases}
\end{equation}

\begin{thm}[
cf. \cite{AKY-DAN}]
\label{thm:vdp-comp}
A function $f\: \Z_p\rightarrow \Z_p$ is 1-Lipschitz \textup{(that is, an
automaton function)} if and only if
$f$ can be represented as
\begin{equation}
\label{eq:vdp-comp}
f(z)=\sum_{m=0}^{\infty}b_m
p^{\left\lfloor \log_p m \right\rfloor} \chi(m,z),
\end{equation}
where $b_m\in \Z_p$ for $m=0,1,2,\ldots$
\end{thm}
By using the  van der Put series it is possible to determine whether a mapping
$f\:\Z_p\>\Z_p$ is an automaton function of a finite automaton. We first remind some notions and facts from the theory of automata sequences following
\cite{Allouche-Shall}.

An infinite sequence $\mathbf a=(a_i)_{i=0}^\infty$ over a finite alphabet $\Cal A$, $\#\Cal A=L<\infty$, is called \emph{$p$-automatic} if there
exists a finite transducer $\mathfrak T=\langle\F_p,\Cal S,\Cal A,S,O,s_0\rangle$ such
that for all $n=0,1,2,\ldots$, if $\mathfrak T$ is feeded by the word $\chi_k\chi_{k-1}\cdots\chi_0$
which is a base-$p$ expansion of $n=\chi_0+\chi_1p+\cdots\chi_kp^k$, $\chi_k\ne
0$ if $n\ne 0$, then the $k$-th output symbol of $\mathfrak T$ is $a_n$;
or, in other words, such that $\delta_{k}^{\Cal A}(f_\mathfrak T(n))=a_n$ for all $n\in\N_0$,
where $k=\lfloor\log_p n\rfloor$ 
and $\delta_k^{\Cal A}(r)$ stands for the $k$-th digit in the base-$L$ expansion of $r\in\N_0$.

A \emph{$p$-kernel} of the sequence $\mathbf a$ is a set $\ker_p(\mathbf
a)$ of all subsequences $(a_{jp^m+t})_{j=0}^\infty$, $m=0,1,2,\ldots$,
$0\le t<p^m$. 
\begin{thm}[Automaticity criterion, cf.
\protect{\cite[Theorem 6.6.2]{Allouche-Shall}}] 
\label{thm:p-ker}
A sequence $\mathbf a$ is $p$-automatic if and only if its $p$-kernel
is finite.
\end{thm}
\begin{thm}[Finiteness criterion, cf. \cite{me-auto_fin}]
\label{thm:fin-auto}
Let a 1-Lipschitz function $f\:\Z_p\>\Z_p$ be represented by van der Put series \eqref{eq:vdp-comp}. 
The function $f$ is a  finite automaton function 
if and only if the following conditions hold simultaneously:
\begin{enumerate}
\item all coefficients $b_m$, $m=0,1,2,\ldots$, constitute a finite subset $B_f\subset\Q\cap\Z_p$,
and
\item the $p$-kernel of the sequence $(b_m)_{m=0}^\infty$ is finite.
\end{enumerate}
\end{thm}

\begin{note} 
Condition (ii) of the theorem is equivalent to the condition that
the  sequence $(b_m)_{m=0}^\infty$ is $p$-automatic, cf. Theorem \ref{thm:p-ker}.
\end{note}
Criteria to determine if an automaton function  is finite which are based on  expansions other than van der Put are
also known, cf. \cite{Smyshl-fin-auto,Vuillem_fin}.

In literature, automata with multiple inputs and outputs over the same alphabet
are also studied.
We remark that in the case when the alphabet is $\F_p$, the automata can
be considered as automata whose input/output  alphabets are Cartesian powers $\F_p^n$ and $\F_p^m$, for suitable $m,n\in\N$. For these automata a theory
similar to that of automata with a single input/output can be developed: Corresponding automata function are then 1-Lipshitz mappings from $\Z_p^n$ to $\Z_p^m$ w.r.t. $p$-adic metrics. Recall that $p$-adic absolute value
on $\Z_p^k$ is defined as follows: Given $(z_1,\ldots,z_k)\in\Z_p^k$, put
$|(z_1,\ldots,z_k)|_p=\max\{|z_i|_p\colon i=1,2,\ldots,k\}$. The so defined
absolute value (and
the corresponding metric) are non-Archimedean as well.  The main theorem
of the paper holds
(after a proper re-statement) for these automata as well, see Theorem \ref{thm:main-mult}.

It is worth recalling  here a well-known fact (which also can be proved by
using Theorem \ref{thm:fin-auto}) that \emph{addition of two $p$-adic integers
can be performed by a finite automaton with two inputs and one output}: Actually
the automaton just finds successively (digit after digit)  the sum  by
a standard addition-with-carry algorithm which
is used to find a sum of two non-negative integers represented by base-$p$
expansions thus calculating the sum with arbitrarily high accuracy w.r.t.
the $p$-adic metric.
 On the contrary, \emph{no finite automaton can perform multiplication
of two  arbitrary $p$-adic integers} since it is well known that no finite automaton can calculate a
base-$p$ expansion of a square of an arbitrary non-negative integer given
a base-$p$ expansion of the latter, cf., e.g., \cite[Theorem
2.2.3]{Bra}.

From these remarks combined with  Theorem \ref{thm:fin-auto} 
the following properties of finite automata
functions can  be deduced:
\begin{prop}
\label{prop:fin-auto}
Let $\mathfrak A,\mathfrak B$ be  finite automata, let $a,b\in\Z_p\cap \Q$ be $p$-adic rational integers. Then the following is true:
\begin{enumerate}
\item the mapping $z\mapsto f_{\mathfrak A}(z)+f_{\mathfrak B}(z)$ of $\Z_p$
into $\Z_p$ is a finite automaton function;
\item a composite function $f(z)=a\cdot f_{\mathfrak A}(z)+b$, $(z\in\Z_p)$,
is a finite automaton
function;
\item a constant function $f(z)=c$ is a finite automaton function if and
only if $c\in\Z_p\cap\Q$;
\item  an affine mapping $f(z)=c\cdot z+d$ is a finite automaton function
if and only if $c,d\in\Z_p\cap\Q$.
\end{enumerate}
\end{prop}
\begin{proof}
Note that the van der Put expansion of the constant function $z\mapsto c$ is
\begin{equation}
\label{eq:vdp-const}
c=c\chi(0,z)+c\chi(1,z)+\cdots+c\chi(p-1,z)+0\chi(p,z)+0\chi(p+1,z)+\cdots,
\end{equation}
while the van der Put expansion for the identity function $z\mapsto z$ is
\begin{equation}
\label{eq:vdp-id}
z=\sum_{i=0}^\infty\delta_{\lfloor\log_pi\rfloor}(i)p^{\lfloor\log_pi\rfloor}\chi(i,z),
\end{equation}
where $\delta_j(i)$ stands for the $j$-th  digit in the base-$p$ expansion
of $i$. Now all  statements of the proposition follow immediately from
Theorem \ref{thm:fin-auto} and the above mentioned facts from finite automata theory.
\end{proof}
Note that the statement of Proposition \ref{prop:fin-auto} is known: For
instance, it can  be deduced
from the old work \cite{Lunts} of A.~G.~Lunts.  To our best knowledge, Lunts
was the first who revealed connections between automata theory and $p$-adic analysis.
It is worth noticing that Lunts defines automata functions in a slightly different way than
we do: In his work, an automaton function is a 1-Lipschitz function $F\:\Q_p\>\Q_p$
such that $F(pz)=pF(z)$ for all $z\in\Q_p$. Also, Lunts' methods of proofs are completely different form the ones of
Proposition \ref{prop:fin-auto}. Unfortunately, most automata theorists seem
to be
unaware of the paper \cite{Lunts}   since it was never
translated into English and even was never reviewed by Mathematical Reviews.


Concluding the subsection, we remark that in literature (finite) automata functions are also
known under  names of \emph{(bounded) determinate functions}, or \emph{(bounded)
deterministic
functions}, cf., e.g., \cite{Yb-eng,Cherep_Approx,Cherepov_Approx-contf,Smyshl-fin-auto}.
\subsection{Real plots of automata functions vs Monna graphs.}
\label{ssec:plots}
Further in the paper we consider special representation of automata functions by
point sets of real and complex spaces. As we have already mentioned in previous
section,   several representations
of this sort were considered: Via the so-called limit sets (see e.g. \cite{Haeseler_Barbe}),
via the Monna graphs (see e.g. \cite{Cherep_Approx,Cherepov_Approx-contf,Konech_Aff,Lisovik_Realfunk,Shkar_AffineAuto} ) and via real plots which were originally introduced in \cite[Chapter 11]{AnKhr}.
In the paper we focus on real plots; however we will start this subsection with saying few
words about Monna graphs since in some meaning they are counterpart of real plots; and we will
not touch limit sets at all since they are  standing somewhat apart.

The Monna graphs are based on the Monna's representation of $p$-adic integers
via real numbers of the unit closed segment $[0,1]$ originally suggested
by Monna in \cite{Monna}: Given a canonical expansion $z=\sum_{i=0}^\infty\alpha_ip^i$ of  $p$-adic integer $z\in\Z_p$ (cf. Subsection \ref{ssec:p-adic}), consider
a real number $\mon(z)=\sum_{i=0}^\infty\alpha_ip^{-i-1}\in[0,1]\subset\R$. It is
clear that $\mon$ maps $\Z_p$ onto $[0,1]$, however, $\mon$ is not bijective:
The only points from the open interval $(0,1)$ that have more than one (actually, exactly two) pre-image w.r.t. $\mon$ are rational numbers of the form $\sum_{i=0}^\infty\alpha_ip^{-i-1}$
where $\alpha_i=p-1$ for some $i\ge i_0$ since 
\begin{gather}
\label{eq:discont}
\sum_{i=0}^\infty\alpha_ip^{-i-1}=\sum_{i=0}^\infty\beta_ip^{-i-1},\ \text{where}\\
\notag
\beta_j=\begin{cases}
         \alpha_j,&\text{if}\ j\le i_0-2;\\
         (\alpha_{i_0-1}+1)\md p,&\text{if}\ j=i_0;\\
         0,&\text{if}\ j\ge i_0+1         
         \end{cases}
\end{gather}
where $\alpha_j=\beta_j$ for all $j\le i_0-2$, $\beta_j=0$ for all $j\ge
i_0$ and $\beta_{i_0-1}=(\alpha_{i_0-1}+1)\md p$. We can naturally
associate the segment $[0,1]$ (or a half-open interval $[0,1)$) to the real circle $\mathbb S$ by \emph{reducing
$[0,1]$ modulo 1}; that is, by taking fractional parts of reals from $[0,1]$:
$\mathbb S=[0,1]\md 1$. Then in a similar manner we may consider a mapping of $\Z_p$ onto $\mathbb S$; we will
denote the mapping also via $\mon$ since there is no risk of misunderstanding.
Note that w.r.t. the latter mapping the point $0=1\in\mathbb S$ has exactly two pre-images
since
$\sum_{i=0}^\infty 0\cdot p^{-i-1}=0=1=\sum_{i=0}^\infty (p-1)\cdot p^{-i-1}$
in $\mathbb S$.

Now, given an automaton $\mathfrak A=\mathfrak A(s_0)$, we define the \emph{Monna graph} of
$\mathfrak A$ as follows: Let $f=f_{\mathfrak A}$ be a corresponding automaton
function, cf. Subsection \ref{ssec:a-map} (that is, $f\:\Z_p\>\Z_p$ is a
1-Lipschitz function w.r.t. $p$-adic metric). Then  the Monna graph $\mathbf
M(\mathfrak A)=\mathbf M(f)$ (or, which is the same, of the automaton function
$f$) is the point set $\mathbf M(\mathfrak A)=\mathbf M(f)=\{(\mon(z),\mon(f(z)))\:z\in\Z_p\}$.
Note that we can consider the Monna graph when convenient either as a subset of the unit real square $\mathbb I^2$,
a Cartesian square of a unit segment,
$\mathbb I^2=[0,1]\times[0,1]$, or as a subset of a 2-dimensional real torus $\mathbb T^2=\mathbb
S\times\mathbb S$, a Cartesian square of a real unit circle $\mathbb S$.
A Monna graph can be considered as a graph of a real function $f^{\mathfrak
A}$ 
defined on $[0,1]$ and valuated in $[0,1]$. Indeed, given a point $x\in[0,1]$,
which is not of the form \eqref{eq:discont}, there is a unique $z\in\Z_p$ such
that $\mon(z)=x$. Therefore,  $f^{\mathfrak A}$ is well defined at $x$ since there exists a unique $y\in[0,1]$ such that $y=\mon(f_{\mathfrak
A}(z))$; so we just put $f^{\mathfrak A}(x)=y$. Once $x$ is of the form \eqref{eq:discont},
then there exist exactly two $z_1,z_2\in\Z_p$, $z_1\ne z_2$ such that $\mon(z_1)=\mon(z_2)=x$.
As $f_{\mathfrak A}(z_1)$ is not necessarily equal to $f_{\mathfrak A}(z_2)$,
then $f^{\mathfrak A}$  may be not well defined at $x$: One have to assign
to $f^{\mathfrak A}(x)$ both $\mon(f_{\mathfrak A}(z_1))$ and $\mon(f_{\mathfrak A}(z_2))$ which may happen to be non-equal. To make $f^{\mathfrak A}$ well defined on
$[0,1]$  a usual way is to admit only representations of one (of two) types 
for $x$ of the form \eqref{eq:discont}; say, only those with finitely many
non-zero terms, cf., e.g., \cite{Cherep_Approx,Cherepov_Approx-contf}. In this case the function $f^{\mathfrak A}$ becomes well-defined
everywhere on $[0,1]$
and having points of discontinuity at maybe the points of type \eqref{eq:discont}
only.
A typical Monna graph of the function $f^{\mathfrak A}$  looks like the  one represented by Figure \ref{fig:Monna}.

Now we are going to introduce a notion of the \emph{real plot} of an automaton function; the latter notion is somehow `dual' to  the notion of  Monna
graph. Given an automaton $\mathfrak A=\mathfrak
A(s_0)$, we associate to an $m$-letter non-empty word $v=\gamma_{m-1}\gamma_{m-2}\ldots\gamma_0$
over the alphabet $\F_p$
 a rational number $0.
 v$ whose base-$p$ expansion is 
$$0.
v=0.\gamma_{m-1}\gamma_{m-2}\ldots\gamma_0=\sum_{i=0}^{m-1}\gamma_{m-i-1}p^{-i-1};$$ then
to every  $m$-letter input word $w=\alpha_{m-1}\alpha_{m-2}\cdots\alpha _0$ of the automaton $\mathfrak
A$ and to the 
respective $m$-letter output word $\mathfrak a(w)=\beta_{m-1}\beta_{m-2}\cdots \beta_0$ 
(rightmost letters are feeded to/outputted from the automaton prior to leftmost
ones) 
there corresponds a point $(0.
w;0.
{\mathfrak
a(w)})$
of the real unit square $\mathbb I^2$; then
we define
$\mathbf P(\mathfrak A)$ as a closure  in 
$\R^2$ of the point set $(0.
w;0.
{\mathfrak
a(w)})$
where $w$ ranges over the set $\Cal W$ of all finite non-empty words over the alphabet $\F_p$. 

Given an automaton function $f=f_{\mathfrak A}\:\Z_p\>\Z_p$ define a  set
$\mathbf P(f_{\mathfrak A})$ of points of
the real plane $\R^2$ 
as follows: For $k=1,2,\ldots$ denote
\begin{equation}
\label{eq:plot}
E_k(f)=\left\{
\left({\frac{{z\md p^k}
}{p^k};\frac{{f(z)\md p^k}
}{p^k}
}\right)\in\mathbb I^2\: z\in \Z_p\right\}
\end{equation} 
a point set in a unit real square $\mathbb I^2=[0,1]\times[0,1]$ and take a
union $E(f)=\cup_{k=1}^\infty E_k(f)$; then $\mathbf
P(f)$ is a closure (in topology of $\R^2$) of the set $E(f)$.
Note that if $z=\sum_{i=0}^\infty\gamma_ip^i$ is a $p$-adic canonical
expansion of $z\in\Z_p$ then $p^{-m}(z\md p^m)=0.\gamma_{m-1}\gamma_{m-2}\ldots\gamma_0$,
c.f. \eqref{eq:plot}; so $\mathbf P(\mathfrak A)\supset\mathbf P(f_{\mathfrak
A})$. Moreover, $\mathbf P(\mathfrak A)=\mathbf P(f_{\mathfrak
A})$, see further  Note \ref{note:plot-auto-in}.
\begin{defn}[Automata plots]
\label{def:plot-auto}
Given an 
automaton $\mathfrak A$,
we call a \emph{plot of the automaton} $\mathfrak A$ 
the set $\mathbf P(\mathfrak A)$. We call a \emph{limit plot} 
of the automaton $\mathfrak A$
the point set $\mathbf{LP}(\mathfrak A)$ 
which is defined as follows: A point $(x;y)\in\R^2$ lies in $\mathbf{LP}(\mathfrak A)$ if and only if there exist $z\in\Z_p$ and a strictly increasing
infinite sequence $k_1<k_2<\ldots$ of numbers from $\N$ such that simultaneously
\begin{equation}
\label{eq:def-LP}
\lim_{i\to\infty}\frac{z\md p^{k_i}}{p^{k_i}}=x;\ \lim_{i\to\infty}\frac{f_\mathfrak
A(z)\md p^{k_i}}{p^{k_i}}=y.
\end{equation}  
\end{defn}
%
\begin{note}
\label{note:plot-auto}
Further in the paper we consider $\mathbf{LP}(\mathfrak A)$ (as well as 
$\mathbf{P}(\mathfrak A)$ and $\mathbf{P}(f)$) either as a subset of the unit square $\mathbb I^2\subset\R^2$ or as a subset
of the unit torus $\mathbb T^2=\R^2/\Z^2$ when appropriate. Note that when
considering the plot on the unit torus we  reduce  coordinates of
the points modulo 1, that is, we just `glue together' 0 and 1 of the unit
segment $\mathbb I$ thus transforming it into the unit circle $\mathbb S$
(whose points we usually identify with the points of the half-open segment $[0,1)$ via a natural one-to-one correspondence, say, $\omega\leftrightarrow\sin^2(\omega/2)$). Also, sometimes we consider $\mathbf{LP}(\mathfrak A)$ (as well as
$\mathbf{P}(\mathfrak A)$ and $\mathbf{P}(f)$)
as a subset of the cylinder $\mathbb I\times\mathbb S$ or of the cylinder $\mathbb S\times\mathbb I$
by reducing modulo 1 either $y$- or $x$-coordinate respectively. We 
denote the corresponding plot via $\mathbf{LP}_\mathbb M(\mathfrak A)$ by using the subscript $\mathbb M\in\{\mathbb I^2,\mathbb T^2,\mathbb I\times\mathbb
S,\mathbb S\times\mathbb I\}$ and we omit the subscript when it is clear
(or when it is no difference)
on which of the  surfaces the plot is considered.
\end{note}

We take a moment to recall some well-known topological notions and to introduce
some notation. In the sequel, given a subset $S$ of a topological (in particular,
metric) space $\mathbb M$ which satisfies the Hausdorff axiom we denote via $\mathbf{AP}_\mathbb
M(S)$ the set of all accumulation points of $S$. Recall that the point $x\in\mathbb
M$
is called an \emph{accumulation point of} $S\subset \mathbb M$ once every neighborhood of
$x$ contains infinitely many points from $S$; and a point $y\in \mathbb M$
is called \emph{isolated  point of} $S$ (or, the point isolated from $S$;
or, the point isolated w.r.t. $S$) once there exists a neighborhood $U\ni y$ such that $U$
contains no points from $S$ other than (maybe) $y$. We may omit the subscript and use  notation $\mathbf{AP}(S)$ when it is clear from the context what metric
 space is meant. 
 
 We also write $\mathbf{AP}((a_i)_{i=0}^\infty)$ (or briefly
 $\mathbf{AP}(a_i)$, or $\mathbf{AP}(\Cal C))$ for the set of all limit
 points of the sequence $\Cal C=(a_i)_{i=0}^\infty$ over $\mathbb M$. Recall that
 a point $x\in\mathbb M$ is called a limit (or, cluster) point of the sequence $(a_i)_{i=0}^\infty$
 if every neighbourhood of $x$ contains infinitely many members of the sequence
 $(a_i)_{i=0}^\infty$; that is, given any neighborhood $U$ of $x$, the number
 of $i$ such that $a_i\in U$ is infinite (note that the very  $a_i\in U$ 
 are not assumed  to be pairwise distinct points of $\mathbb M$; some, or
 even all of them may be identical). Note that in topology the terms `accumulation point of a set' and `limit point of a set'
are used as synonyms; however to avoid possible misunderstanding we reserve
the term `limit point' only for sequences while for sets we use the term
`accumulation point'. 

\begin{note}
\label{note:plot-auto-in}
The definition of $\mathbf P(\mathfrak A)$  immediately implies that  $(x;y)\in\mathbf{P}(\mathfrak A)$ if and only if there
exists a sequence $(w_i)_{i=0}^\infty$ of finite non-empty words $w_i\in\Cal W$ such that $\Lambda(w_i)=k_i$
for all $i=0,1,2,\ldots$ and $\lim_{i\to\infty}\rho(w_i)=x$, 
$\lim_{i\to\infty}\rho(\mathfrak a(w_i))=y$. Note that once $(x;y)\in\mathbf{LP}(\mathfrak A)$ then there
exists a sequence $(w_i)_{i=0}^\infty$ of words such that the sequence $(\Lambda(w_i)=k_i)_{i=0}^\infty$
of their lengths is strictly increasing:
One just  may take 
$w_i=\wrd_{k_i}(z\md p^{k_i})$, cf. \eqref{eq:rho} and Subsection \ref{ssec:a-map}.
Therefore $\mathbf{LP}(\mathfrak A)\subset\mathbf{AP}(\mathbf P(f_{\mathfrak A}))$. Moreover, from Definition \ref{def:plot-auto} it readily follows that
$\mathbf{AP}(\mathbf P(f_{\mathfrak A}))=\mathbf{AP}(E(f_{\mathfrak A}))=\mathbf{AP}(\mathbf P({\mathfrak A}))$ since given a finite non-empty word $w$ and taking any $z\in\Z_p$ such that the prefix of the corresponding infinite word is $w$ (i.e., such that $w=\wrd_{\Lambda(w)}(z\md p^{\Lambda(w)})$) we see that 
$\rho(\mathfrak a(w))=((f_{\mathfrak A}(z))\md p^{\Lambda(w)})/p^{\Lambda(w)}$.
This implies that $\mathbf P({\mathfrak A})=\mathbf P(f_{\mathfrak A})$ since
$\mathbf P(f_{\mathfrak A})=E(f_{\mathfrak A})\cup\mathbf{AP}(E(f_{\mathfrak A}))=\mathbf P({\mathfrak A})$; so in the sequel we do not differ automata
plots from the plots of automata functions and use both $\mathbf P({\mathfrak A})$ and $\mathbf P(f_{\mathfrak A})$ as notation for the plot of the automaton $\mathfrak
A$. Also we may use notation $\mathbf{LP}(f_{\mathfrak A})$ along with $\mathbf{LP}(\mathfrak
A)$ to denote the limit plot of the automaton $\mathfrak A$.
\end{note}
We stress here once again a crucial difference in the construction of plots and of Monna graphs of automata: Given a canonical expansion of  $p$-adic integer 
$z=\sum_{i=0}^\infty\gamma_ip^i$ we put into a correspondence to $z$ a
\emph{single} real number $\mon(z)=\sum_{i=0}^\infty\gamma_ip^{-i-1}$ while constructing
Monna graphs; whereas in the construction of plots we put into a correspondence
to $z$ a \emph{whole set of all limit points} of the sequence
$(p^{-m}(z\md p^m))_{m=1}^\infty$, and the latter set may not consist of a
single point; moreover, `usually' the set never consists of a single point since with a probability 1 the set is 
a whole segment $[0,1]$. Therefore to study structure of plots we
need to work with  sets of all limit points of (usually non-convergent)
sequences rather
than with limits of convergent sequences as in the case of Monna maps.

\begin{prop}
\label{prop:no-iso}
Let $\mathfrak A$ be an arbitrary automaton; then
$\mathbf{LP}(\mathfrak A)$ contains no points isolated w.r.t. $E(f_\mathfrak
A)$
\textup{(cf. \eqref{eq:plot}
and the text thereafter)}. 
\end{prop}
\begin{proof}[Proof of Proposition \ref{prop:no-iso}]
Let  $(x;y)\in\mathbf{LP}(\mathfrak A)$
be a point isolated w.r.t. $E(f_\mathfrak A)$. As $(x;y)\in\mathbf{LP}(\mathfrak A)$, let $z=\sum_{j=0}^\infty\zeta_j\cdot p^j$ be a $p$-adic canonical representation
of the $p$-adic integer $z\in\Z_p$ mentioned in Definition \ref{def:plot-auto};
and let $f_\mathfrak A(z)=\sum_{j=0}^\infty\gamma_j\cdot p^j$ be a $p$-adic canonical
expansion of the $p$-adic integer $f_\mathfrak A(z)$.
Then as the point $(x;y)$ is isolated, there exists $I\in\N$ such that $z\md p^{k_i}/p^{k_i}=x$ and $f_\mathfrak A(z)\md p^{k_i}/p^{k_i}=y$ for all $i\ge
I$, cf. \eqref{eq:def-LP} (if otherwise, the point $(x,y)$ is not  isolated
w.r.t. $E(f_\mathfrak A)$).
Put $w_{i}=
\wrd_{k_i}\left(z\md p^{k_i}/p^{k_i}\right)=\zeta_{k_i-1}\zeta_{k_i-2}\ldots\zeta_0$, $u_{i}=
\wrd_{k_i}\left(f_\mathfrak A(z)\md p^{k_i}/p^{k_i}\right)=\gamma_{k_i-1}\gamma_{k_i-2}\ldots\gamma_0$; then
\begin{align}
\label{eq:mark-up-aa}
0.\zeta_{k_i-1}\ldots\zeta_{0}&=
x;\\
\label{eq:mark-up-bb}
0.\gamma_{k_i-1}\ldots\gamma_{0}&= 
y,
\end{align}
for all $i\ge I$. We claim that then necessarily both $z=0$ and $f_\mathfrak
A(z)=0$ (whence both $x=0$ and $y=0$). 

Indeed,
as the sequence $\Cal K=(k_i)_{i=0}^\infty$ is infinite and strictly increasing,
then taking $i=I$ in \eqref{eq:mark-up-aa} we conclude that necessarily
$\zeta_0=\zeta_1=\cdots=\zeta_{k_{I+M}-k_{I}-1}=0$  for all $M\in\N$. Therefore,
taking $M$ large enough so that $k_{I+M}-k_{I}\ge k_I$ (which is always
possible since the sequence $\Cal K$ is strictly increasing) we see that
$\zeta_0=\zeta_1=\cdots=\zeta_{k_{I}-1}=0$
and thus $\zeta_i=0$ for all $i\in\N_0$ since 
$0.\zeta_{k_i-1}\ldots\zeta_{0}=0.\zeta_{k_I-1}\ldots\zeta_{0}$
for all $i\ge I$ by \eqref{eq:mark-up-aa}. But this implies that $z=0$ (whence $x=0$). The same argument combined with \eqref{eq:mark-up-bb}
shows that $f_\mathfrak A(0)=0$ and  $y=0$.

Consider now an automaton $\mathfrak B$ whose automaton function  $f_\mathfrak
B$ is defined as follows: Given a $p$-adic canonical representation 
$z=\sum_{j=0}^\infty\zeta_j\cdot p^j$, let $\delta_0(f_\mathfrak B(z))=1$
and $\delta_j(f_\mathfrak B(z))=\delta_j(f_\mathfrak A(z))$ for $j>0$. Such
an automaton $\mathfrak B$ exists  since the so defined function $f_\mathfrak
B$ satisfies \eqref{eq:f-coord} and thus is 1-Lipschitz, cf. Subsection \ref{ssec:a-map}.
Actually the automaton being feeded by the input word $\ldots\zeta_2\zeta_1\zeta_0$
just put $1$ as the first output letter and put $\gamma_j$ for the $j$-th
output letter for $j>0$ where $\ldots\gamma_2\gamma_1\gamma_0$ is the output
word of the automaton $\mathfrak A$ feeded by the input word $\ldots\zeta_2\zeta_1\zeta_0$;
that is, $\mathfrak B$ outputs $\ldots\gamma_2\gamma_11$ being feeded by 
$\ldots\zeta_2\zeta_1\zeta_0$.

From \ref{eq:plot} and Definition \ref{def:plot-auto} it immediately follows
that $\mathbf{LP}(\mathfrak A)=\mathbf{LP}(\mathfrak B)$ and that a point
$(x;y)\in\R^2$ is an isolated point of $E(f_\mathfrak B)$ if and only if
it is an isolated point of $E(f_\mathfrak A)$. But by the claim we have proved
above, once $(x;y)$ is an isolated point of $E(f_\mathfrak B)$, then 
necessarily $f_\mathfrak B(0)=0$. But the first  letter of any output word
of  automaton $\mathfrak B$ is 1 by the construction of $f_\mathfrak B$;
thus
$\delta_0(f_\mathfrak B(0))=1$ and so $f_\mathfrak B(0)\ne 0$. From the
claim we have proved at the beginning of the proof it follows now that 
$\mathbf{LP}(\mathfrak B)$ cannot contain
isolated points of $E(f_\mathfrak B)$; thus $\mathbf{LP}(\mathfrak A)$ cannot contain
isolated points of $E(f_\mathfrak A)$.
\end{proof}
\begin{rmk}
Note that Proposition \ref{prop:no-iso} only states that $\mathbf{LP}(\mathfrak
A)$ contains no points isolated from $E(f_\mathfrak A)$, but of course $\mathbf{LP}(\mathfrak
A)$ may contain isolated points w.r.t. itself. For instance, let $\mathfrak
A$ be a $p$-adic odometer; that is, $f_\mathfrak
A(z)=z+1$ (the automaton $\mathfrak A$ may be taken a finite then). Then
the point $(1;0)\in\mathbb I^2$ is an isolated point of $\mathbf{LP}_{\mathbb I^2}(\mathfrak
A)$ w.r.t. $\mathbf{LP}_{\mathbb I^2}(\mathfrak
A)$; however $\mathbf{LP}_{\mathbb T^2}(\mathfrak
A)$ contains no points isolated w.r.t. $\mathbf{LP}_{\mathbb T^2}(\mathfrak
A)$.
\end{rmk}

\begin{thm}
\label{thm:AP=LP}
If 
automaton $\mathfrak A$ is finite and minimal then $\mathbf{AP}(E(f_\mathfrak
A))=\mathbf{LP}(\mathfrak A)$. 
\end{thm}
\begin{proof}[Proof of Theorem \ref{thm:AP=LP}]
By Proposition
\ref{prop:no-iso}, $\mathbf{AP}(E(f_\mathfrak A))\supset\mathbf{LP}(\mathfrak A)$; we need to prove that the inverse inclusion also holds.  Let $(x;y)\in\mathbf{AP}(E(f_\mathfrak A))$; then there exists a sequence $(z_i)_{i=0}^\infty$ of $p$-adic integers
and a sequence $(k_i)_{i=0}^\infty$ of integers from $\N$ such that 
all the points 
$$
\mathbf p_i=\left({\frac{{z_i\md p^{k_i}}
}{p^{k_i}};\frac{{f_\mathfrak A(z_i)\md p^{k_i}}
}{p^{k_i}}
}\right)\in\mathbb R^2
$$ 
are pairwise distinct and
\begin{align*}
\lim_{i\to\infty}\frac{z_i\md p^{k_i}}{p^{k_i}}&=x;\\
\lim_{i\to\infty}\frac{f_\mathfrak A(z_i)\md p^{k_i}}{p^{k_i}}&=y.
\end{align*}
We may assume that the sequence $(k_i)_{i=0}^\infty$ is  increasing
since otherwise in the point sequence $(\mathbf p_i)_{i=0}^\infty$ there
are only finitely many pairwise distinct points. Moreover, we may assume
that the sequence $(k_i)_{i=0}^\infty$ is  strictly increasing; we consider
corresponding infinite point subsequence of $(\mathbf p_i)_{i=0}^\infty$
if otherwise. So we see that there exists an infinite sequence of words
$h_i=\wrd_{k_i}(z_i\md p^{k_i})$ of strictly increasing lengths $k_i$ such
that
\begin{align}
\label{eq:AP=LP-1}
\lim_{i\to\infty}0.h_i&=x;\\
\label{eq:AP=LP-2}
\lim_{i\to\infty}0.\mathfrak a(h_i)&=y.
\end{align}

That is, there exists a sequence 
$\left(h_i\right)_{i=0}^\infty$ of
words $h_i=\alpha_{k_i-1}^{(i)}\ldots\alpha_{0}^{(i)}$ of strictly increasing lengths $1\le k_0<k_1<k_2<\ldots$ such that
$\lim_{i\to\infty}0.\alpha_{k_i-1}^{(i)}\ldots\alpha_{0}^{(i)}=x$. From here
it follows that once $i$ is sufficiently large (say, once $i\ge M_0\in\N_0$)
then $\alpha_{k_i-1}^{(i)}=\zeta_{0}$
for a suitable $\zeta_0\in\F_p$. By the same reason, $\alpha_{k_i-2}^{(i)}=\zeta_{1}$
for a suitable $\zeta_1\in\F_p$ once $i$ is large enough (say, once $i\ge
M_1\in\N_0$), etc. Moreover, we may assume that the sequence $(M_\ell)_{\ell=0}^\infty$
is strictly increasing.
Therefore, $x=0.\zeta_0\zeta_1\ldots$. Applying a similar argument to 
the sequence $\beta_{k_i-1}^{(i)}\beta_{k_i-2}^{(i)}\ldots\beta_0^{(i)}=\mathfrak a(\alpha_{k_i-1}^{(i)}\alpha_{k_i-2}^{(i)}\ldots\alpha_0^{(i)})$
 ($i=0,1,2,\ldots$) 
 we conclude
that there exists a strictly increasing sequence $(N_\ell)_{\ell=0}^\infty$
such that $\beta_{k_i-\ell-1}^{(i)}=\gamma_{\ell}\in\F_p$ once $i\ge N_\ell$
and therefore $y=0.\gamma_0\gamma_1\ldots$. Moreover, by the construction
of the sequences $(M_\ell)_{\ell=0}^\infty$ and $(N_\ell)_{\ell=0}^\infty$
we may assume that
$M_\ell=N_\ell$ for all $\ell\in\N_0$. Thus we have shown that
\begin{align*}
\alpha_{k_i-1}^{(i)}\ldots\alpha_{0}^{(i)}&=\zeta_0\zeta_1\ldots\zeta_{\ell}w^{(i)}_\ell\
\text{if}\ i\ge M_\ell\ (\ell=0,1,2,\ldots);\\
\beta_{k_i-1}^{(i)}\ldots\beta_{0}^{(i)}&=\gamma_0\gamma_1\ldots\gamma_{\ell}u^{(i)}_\ell\
\text{if}\ i\ge M_\ell\ (\ell=0,1,2,\ldots),
\end{align*}
where $w^{(i)}_\ell,u^{(i)}_\ell\in\Cal W_\phi$ and 
$\gamma_0\gamma_1\ldots\gamma_{\ell}u^{(i)}_\ell=\mathfrak a(\zeta_0\zeta_1\ldots\zeta_{\ell}w^{(i)}_\ell)$,
($\ell=0,1,2,\ldots$). Let $s_\ell^{(i)}$ be a state the automaton $\mathfrak A$ reaches
after being feeded by the input word $w^{(i)}_\ell$ (note that $s_\ell^{(i)}=s_0$, the initial
state, once $w^{(i)}_\ell$ is empty word). As the number of states of $\mathfrak
A$ is finite, at least one state $s\in\Cal S$ repeats in the sequence  
$\left(s_\ell^{(M_\ell)}\right)_{\ell=0}^\infty$ infinitely many times. 
Therefore
\begin{align}
\label{eq:mark-up_A}
\lim_{\ell\to\infty}
0.\zeta_0\ldots\zeta_{\ell}
&=
x;\ \text{whence}\ x=0.\zeta_0\zeta_1\zeta_2\ldots\\
\label{eq:mark-up_B}
\lim_{\ell\to\infty}
0.\mathfrak a_s(\zeta_0\ldots\zeta_{\ell})&= 
y; \text{whence}\ y=0.\gamma_0\gamma_1\gamma_2\ldots.
\end{align}
Denote $w_\ell=\zeta_0\ldots\zeta_{\ell}$, $v_\ell=\gamma_0\ldots\gamma_{\ell}$,
($\ell=0,1,2,\ldots$). As every state of the  automaton $\mathfrak A$ is reachable from the initial
state $s_0$, there exists a word $t_0\in\Cal
W_\phi$ such that the if the automaton $\mathfrak A$ (which is initially at
the state $s_0$) has been feeded by the word $t_0$, then $\mathfrak A$ outputs
the word $\bar t_0=\mathfrak a(t_0)$ and reaches the state $s$. Thus $\mathfrak
a(w_0t_0)=\gamma_0\bar t_0$, and the automaton $\mathfrak A$ after being  feeded
by the word $w_0t_0$ reaches the state $r^{(0)}$. As the automaton $\mathfrak
A$ is minimal, there exists a word $t_1\in\Cal W_\phi$ such that once the
automaton $\mathfrak A_1=\mathfrak A(r^{(0)})$ has been feeded by the word
$t_1$, the automaton reaches the state $s$. Now being feeded by the word $w_1$, the automaton
$\mathfrak A_s=\mathfrak A(s)$ outputs the word $v_1$ and reaches a state
$r^{(1)}$. By the minimality of $\mathfrak A$, there exists a word $t_2\in\Cal
W_\phi$ such that after $\mathfrak A(r^{(1)})$ has been feeded by the word
$t_2$, the automaton reaches the state $s$. Now after $\mathfrak A_s$ has been feeded by the word $w_2$,
the automaton $\mathfrak A_s$ reaches the state $r^{(2)}$, and we can find
a word $t_3$ in a manner similar to that of described. Now being feeded  by the
so constructed left-infinite word $\ldots w_2t_2w_1t_1w_0t_0$, the
automaton $\mathfrak A$ outputs the left-infinite word $\ldots v_2\bar t_2v_1\bar
t_1v_0\bar t_0$ where $\bar t_j=\mathfrak a_{r^{(j-1)}}(t_j)$, $j=1,2,3,\ldots$.
Now consider $p$-adic integers $z=\sum_{i=0}^\infty\chi_i\cdot p^i$ and 
$\bar z=\sum_{i=0}^\infty\xi_i\cdot p^i$ which correspond to  infinite
words $\ldots w_2t_2w_1t_1w_0t_0$ and $\ldots v_2\bar t_2v_1\bar
t_1v_0\bar t_0$ accordingly; that is,  $\ldots\chi_2\chi_1\chi_0=\ldots w_2t_2w_1t_1w_0t_0$ and $\ldots\xi_2\xi_1\xi_0=\ldots v_2\bar t_2v_1\bar
t_1v_0\bar t_0$. Then, by the construction we have that $\bar z=f_\mathfrak A(z)$, and from \eqref{eq:mark-up_A}--\eqref{eq:mark-up_B} it
follows that
\begin{align}
\label{eq:mark-up_AA}
\lim_{j\to\infty}\frac{z\md p^{K_j}}{p^{K_j}}&=x;\\
\label{eq:mark-up_BB}
\lim_{j\to\infty}\frac{f_\mathfrak A(z)\md p^{K_j}}{p^{K_j}}&=y,
\end{align}
where $K_j=\sum_{i=0}^j\Lambda(w_i)+\sum_{i=0}^j\Lambda(t_i)$. As the sequence
$(K_j)_{j=0}^\infty$ is strictly increasing, from \eqref{eq:mark-up_AA}--\eqref{eq:mark-up_BB}
it follows now that $(x;y)\in\mathbf{LP}(\mathfrak A)$ in view of Definition
\ref{def:plot-auto}.



\end{proof}
It is well known (see e.g. \cite[Ch.2, Exercise 2]{RealAnalys}) that the set of all accumulation
points of  a Hausdorff topological  space (the \emph{derived set} of the space)
is a closed subset of the space. From Theorem \ref{thm:AP=LP} it follows
 that once a finite automaton  is minimal then its limit plot is a derived set of its plot (whence, closed):
\begin{cor}
\label{cor:AP=LP}
Let an automaton $\mathfrak A$ be finite and minimal; then the set $\mathbf{LP}(\mathfrak
A)$ is a derived set of $\mathbf P(\mathfrak A)$ and therefore is closed in $\R^2$. A point $(x;y)\in\R^2$ belongs to $\mathbf{LP}(\mathfrak A)$ if and only if there exists a sequence $\left(\alpha_{k_i-1}^{(i)}\ldots\alpha_{0}^{(i)}\right)_{i=0}^\infty$
of finite non-empty words of strictly increasing lengths $k_0<k_1<k_2<\cdots$ such that the sequence
$\left(0.\alpha_{k_i-1}^{(i)}\alpha_{k_i-2}^{(i)}\ldots\alpha_0^{(i)}\right)_{i=0}^\infty$ tends to $x$
 and the corresponding sequence 
 $\left(0.\beta_{k_i-1}^{(i)}\beta_{k_i-2}^{(i)}\ldots\beta_0^{(i)}\right)_{i=0}^\infty$
 tends to $y$ as $i\to\infty$,
 where $\beta_{k_i-1}^{(i)}\ldots\beta_0^{(i)}$ are respective output words of the automaton
 $\mathfrak A$ that correspond
 to input words $\alpha_{k_i-1}^{(i)}\ldots\alpha_{0}^{(i)}$ 
 \textup{(i.e., $\beta_{k_i-1}^{(i)}\beta_{k_i-2}^{(i)}\ldots\beta_0^{(i)}=\mathfrak a(\alpha_{k_i-1}^{(i)}\alpha_{k_i-2}^{(i)}\ldots\alpha_0^{(i)})$,
 $i=0,1,2,\ldots$)}.
\end{cor}
We stress once again that
words $\alpha_{k_i-1}\ldots\alpha_{0}$ are feeded to the automaton $\mathfrak
A$ from right
to left; i.e. the letter $\alpha_0$ is feeded  to $\mathfrak A$ first, then the letter $\alpha_1$
is feeded to $\mathfrak A$, etc.
\begin{proof}[Proof of Corollary \ref{cor:AP=LP}]
By the definition, the  set $\mathbf{AP}(E(f_\mathfrak A))=\mathbf{AP}(\mathbf
P({\mathfrak A}))$ is a derived set of $\mathbf P(\mathfrak A)$; whence by Theorem \ref{thm:AP=LP} the set $\mathbf{LP}(\mathfrak
A)$ is a derived (thus, closed) set of $\mathbf P(\mathfrak
A)$. 

The necessity of conditions of the corollary follows immediately from Definition \ref{def:plot-auto} since once $(x;y)\in\mathbf{LP}(\mathfrak
A)$ then there exist a $p$-adic integer $z=\sum_{i=0}^\infty\alpha_i\cdot
p^i$ and a strictly increasing sequence $1\le k_1<k_2<\ldots$ over $\N$ such
that \eqref{eq:def-LP} holds; that is, we just put  
$\alpha_{k_i-1}^{(i)}\ldots\alpha_{0}^{(i)}=\wrd_{k_i}(z\md p^{k_i})$ and
$\beta_{k_i-1}^{(i)}\ldots\beta_{0}^{(i)}=\wrd_{k_i}(f(z)\md p^{k_i})$, where $f$
is an automaton function of $\mathfrak A$, cf. Note \ref{note:plot-auto-in}.

To prove sufficiency of the conditions note that the conditions just yield
that there exists an infinite sequence of words
$h_i=\alpha_{k_i-1}^{(i)}\ldots\alpha_{0}^{(i)}$ of strictly increasing lengths $k_i$ such
that \eqref{eq:AP=LP-1}--\eqref{eq:AP=LP-2} hold. The argument that follows
\eqref{eq:AP=LP-1}--\eqref{eq:AP=LP-2} of the proof of Theorem \ref{thm:AP=LP}
now proves the sufficiency.
\end{proof}
It is worth noticing here that the limit plot of a finite minimal automaton does not depend
on what state of the automaton is taken as initial:
\begin{note}
\label{note:AP=LP}
If $s,t$ are states of a finite minimal automaton $\mathfrak A$, $s\ne t$, then $\mathbf{LP}(\mathfrak
A(s))=\mathbf{LP}(\mathfrak A(t))$.
\end{note} 
Indeed, due to the minimality, every state of $\mathfrak A$ is reachable
from any other state of $\mathfrak A$. Therefore  if $(x;y)\in\mathbf{LP}(\mathfrak
A(t))$ then by Definition \ref{def:plot-auto} there exist $z\in\Z_p$ and a strictly increasing
infinite sequence $k_1<k_2<\ldots$ of numbers from $\N$ such that \eqref{eq:def-LP}
holds. By the minimality of $\mathfrak A$, there exists a finite word $w$ of length $K>0$
such that after the automaton $\mathfrak A(s)$ has been feeded by $w$, it
reaches the state $t$. Now by substituting in Definition \ref{def:plot-auto} $p^K\cdot z+\nm(w)$ for $z$ and $k_1+K<k_2+K<\ldots$ for $k_1<k_2<\ldots$
we see that \eqref{eq:def-LP} holds and therefore $(x;y)\in\mathbf{LP}(\mathfrak
A(s))$. 

Using an idea similar to that of Note \ref{note:AP=LP} it can be easily demonstrated
that if $\mathfrak B$ is a sub-automaton of $\mathfrak A$ then $\mathbf
P(\mathfrak B)\subset\mathbf P(\mathfrak A)$ since every state of
the automaton $\mathfrak A$ is reachable from its initial state:
\begin{note}
\label{note:sub-auto}
Let $\mathfrak B=\mathfrak B(s)$ be a sub-automaton
of the automaton $\mathfrak A$.  As the initial
state $s$ of the automaton $\mathfrak B$ is reachable from the initial state $s_0$ of the
automaton $\mathfrak A$, from the definition of the respective sets it immediately
follows that $\mathbf P(\mathfrak B)\subset\mathbf P(\mathfrak A)$, $\mathbf{LP}(\mathfrak B)\subset\mathbf{LP}(\mathfrak A)$,
and $\mathbf{AP}(\mathfrak B)\subset\mathbf{AP}(\mathfrak A)$.
\end{note}

The following useful lemma is a sort of a counter-part of Lemma \ref{le:per}
in terms of points from $\mathbf{LP}(\mathfrak A)$ rather than in terms of
words.
\begin{lem}
\label{le:qint-LP}
Given a finite automaton $\mathfrak A$ and a point $x\in\Z_p\cap\Q$, if $(x;y)\in\mathbf{LP}(\mathfrak
A)$ for some $y\in\R$ then $y\in\Z_p\cap\Q$.
\end{lem}
\begin{proof}[Proof of Lemma \ref{le:qint-LP}]
As $(x;y)\in\mathbf{LP}(\mathfrak A)$ then there exist $z\in\Z_p$ and a
strictly increasing sequence $k_0<k_1<\ldots$ over $\N$ such that \eqref{eq:def-LP}
holds.  Therefore there exists an infinite sequence of words
$h_i=\wrd_{k_i}(z\md p^{k_i})$ of strictly increasing lengths $k_i$ such
that \eqref{eq:AP=LP-1}--\eqref{eq:AP=LP-2} hold simultaneously. Now repeating
for the case $z_i=z$
the argument that follows \eqref{eq:AP=LP-1}--\eqref{eq:AP=LP-2} in the proof
of Theorem \ref{thm:AP=LP} we conclude that \eqref{eq:mark-up_A}--\eqref{eq:mark-up_B}
hold in our case as well (note that nowhere in the mentioned argument from the proof of Theorem \ref{thm:AP=LP} we used that $\mathfrak A$ is minimal).
Moreover, in the notation of the argument, there exists a strictly increasing
sequence $(M_\ell)_{\ell=0}^\infty$ over $\N$ such that
\begin{align}
\label{eq:qint-LP-1}
\alpha_{k_i-1}^{(i)}\ldots\alpha_{0}^{(i)}&=\zeta_0\zeta_1\ldots\zeta_{\ell}w^{(i)}_\ell\
\text{if}\ i\ge M_\ell\ (\ell=0,1,2,\ldots);\\
\label{eq:qint-LP-2}
\beta_{k_i-1}^{(i)}\ldots\beta_{0}^{(i)}&=\gamma_0\gamma_1\ldots\gamma_{\ell}u^{(i)}_\ell\
\text{if}\ i\ge M_\ell\ (\ell=0,1,2,\ldots);
\end{align}
But $\alpha_j^{(i)}, \beta_j^{(i)}$ do not depend on $i$  since in our case
$z_i=z=\sum_{n=0}^\infty\alpha_np^n$
(where $\alpha_0,\alpha_1,\ldots\in\F_p$) for
all $i$; therefore $\alpha_n^{(i)}=\alpha_n$ for all $n,i\in\N_0$.
As $x=0.\zeta_0\zeta_1\ldots$ (cf. \eqref{eq:mark-up_A}) and $x\in\Z_p\cap\Q$
then the right-infinite word $\zeta_0\zeta_1\ldots$ must be purely periodic
(cf. Corollary \ref{cor:r-repr}) with a period $\chi_0\ldots\chi_{t-1}$ of length $t>0$: that is, $\zeta_0\zeta_1\ldots=(\chi_0\ldots\chi_{t-1})^\infty$.
Now in \eqref{eq:qint-LP-1} put $\ell=mt-1$; then
for every $m\in\N$ we  have that 
$\alpha_{k_i-1}\ldots\alpha_{k_i-mt}=(\chi_0\ldots\chi_{t-1})^m$ 
for all $i\ge M_{mt}$. Now denote via $s_m^{(i)}$
the state the automaton $\mathfrak A$ reaches after have been feeded by the
word $w_{mt}^{(i)}$. Fix $m\in\N$ and denote $s_m$ a state which occurs in
the sequence $(w_{mt}^{(i)})_{i=M_{mt}}^\infty$ infinitely often; due to
the finiteness of the automaton $\mathfrak A$ such state exists. Denote $K_m$
the smallest $i\ge M_{mt}$ such that $s_m=s_m^{(i)}$ (therefore $K_m\ge M_{mt}$). 
And again due to the finiteness of the automaton $\mathfrak A$ in the sequence  $(s_m)_{m=1}^\infty$ some state (say, $s$) occurs
infinitely often. Let $(m_j)_{j=0}^\infty$ be the corresponding infinite
(thus, strictly increasing)
subsequence, i.e., $s_{m_j}=s$; then
as the sequence $(M_\ell)_{\ell=0}^\infty$ is strictly increasing and as
$K_m\ge M_{mt}$, in the sequence $(K_{m_j})_{j=0}^\infty$ there exists a
strictly increasing subsequence, say $(K_{m_{j_r}})_{r=0}^\infty$ (note that
the sequence $({m_{j_r}})_{r=0}^\infty$ is also strictly increasing). Now from
\eqref{eq:qint-LP-1}--\eqref{eq:qint-LP-2} it follows that once being feeded
successfully by purely periodic words 
$w_r=\alpha_{k_i-1}\ldots\alpha_{k_i-m_{j_r}t}=(\chi_0\ldots\chi_{t-1})^{m_{j_r}}$
for $i=K_{m_{j_r}}$, $r=0,1,2,\ldots$, the automaton $\mathfrak A(s)$ 
outputs the words $v_r=\gamma_0\gamma_1\ldots\gamma_{t\cdot
m_{j_r}-1}$. Now by combining Lemma \ref{le:per} with Corollary \ref{cor:r-repr} we conclude  that if $z^\prime\in\Z_p$ is such that 
$\wrd z^\prime=(\chi_0\ldots\chi_{t-1})^{\infty}$ then 
$\lim_{r\to\infty} (z^\prime\md p^{m_{j_r}})/p^{m_{j_r}}=x$ and 
$\lim_{r\to\infty} ((\mathfrak a_{s}(z^\prime))\md p^{m_{j_r}})/p^{m_{j_r}}=y\in\Z_p\cap\Q$.

\end{proof}

Yet one more  property of automata plots is their invariance with respect
to $p$-shifts. That is, \emph{given a point $(x;y)\in\mathbf{P}(\mathfrak
A)$, take base-$p$ expansions $x=0.\chi_1\chi_2\chi_3\ldots$, 
$y=0.\xi_1\xi_2\xi_3\ldots$ of coordinates $x,y$; then 
$(0.\chi_2\chi_3\ldots;0.\xi_2\xi_3\ldots)\in\mathbf{P}(\mathfrak
A)$.} To put it in other words, the following proposition is true:
\begin{prop}
\label{prop:p-shift}
For an arbitrary 
automaton $\mathfrak A$, if $(x;y)\in\mathbf{P}(\mathfrak
A)\subset\T^2$ \textup{(}resp., $(x;y)\in\mathbf{LP}(\mathfrak
A)\subset\T^2$\textup{)} then $((px)\md1;(py)\md1)\in\mathbf{P}(\mathfrak
A)$ \textup{(}resp., $((px)\md1;(py)\md1)\in\mathbf{LP}(\mathfrak
A)$\textup{)}.
\end{prop}
\begin{proof}[Proof of Proposition \ref{prop:p-shift}]
The first statement follows immediately from Note  \ref{note:plot-auto-in} since once 
$$(0.\alpha_{k_i}\ldots\alpha_0;0.\beta_{k_i}\ldots\beta_0)
\to(0.\chi_1\chi_2\ldots;0.\xi_1\xi_2\ldots)$$ as $i\to\infty$ 
then necessarily
$$(0.\alpha_{k_i-1}\ldots\alpha_0;0.\beta_{k_i-1}\ldots\beta_0)
\to(0.\chi_2\chi_3\ldots;0.\xi_2\xi_3\ldots)$$ as $i\to\infty$. 

To prove the second statement, let $f=f_\mathfrak A\:\Z_p\>\Z_p$ be  an automaton function of the automaton
$\mathfrak A$. As $(x;y)\in\mathbf{LP}(\mathfrak A)$, there exists $z\in\Z_p$
and a strictly increasing sequence $(k_i)_{i=0}^\infty$ over $\N$ such that
$x=\lim_{i\to\infty}(z\md p^{k_i})/p^{k_i}$ and $y=\lim_{i\to\infty}(f(z)\md p^{k_i})/p^{k_i}$, cf. Definition  \ref{def:plot-auto}. Therefore $(px)\md1=
(p\lim_{i\to\infty}(z\md p^{k_i})/p^{k_i})\md1=\lim_{i\to\infty}(p(z\md p^{k_i})/p^{k_i})\md1=
\lim_{i\to\infty}(z\md p^{k_i-1})/p^{k_i-1}$ as $(z\md p^{k_i})/p^{k_i}=
\zeta_{k_i-1}p^{-1}+\zeta_{k_i-2}p^{-2}+\cdots+\zeta_{0}p^{-k_i}$
once $z=\zeta_0+\zeta_1p+\cdots+\zeta_{k_i-1}p^{k_i-1}+\cdots$ is a $p$-adic
canonical representation for $z\in\Z_p$. By the same reason, 
$(py)\md1=
(p\lim_{i\to\infty}(f(z)\md p^{k_i})/p^{k_i})\md1=\lim_{i\to\infty}(p(f(z)\md p^{k_i})/p^{k_i})\md1=
\lim_{i\to\infty}(f(z)\md p^{k_i-1})/p^{k_i-1}$. Therefore 
$((px)\md1;(py)\md1)\in\mathbf{LP}(\mathfrak
A)$ by Definition \ref{def:plot-auto}. 
\end{proof}

It is known that the plot $\mathbf P(\mathfrak A)\subset\mathbb I^2$ of the
automaton $\mathfrak A$ can be of two types only; namely,  \emph{given
an automaton $\mathfrak A$, the set $\mathbf P(\mathfrak A)$ either coincides
with the whole unit square $\mathbb I^2$ or $\mathbf P(\mathfrak A)$ is nowhere dense in $\mathbb
I^2$}: Being closed in $\R^2$, the set $\mathbf P(\mathfrak A)$
is measurable w.r.t.  Lebesgue measure on $\R^2$, and the measure of $\mathbf P(\mathfrak A)$ is 1 if and only if $\mathbf P(\mathfrak A)=\mathbb I^2$
and is 0 if otherwise: The later assertion is a statement of  automata 0-1 law, cf. \cite[Proposition 11.15]{AnKhr} and \cite{me:Discr_Syst}. Moreover,
\emph{once an automaton $\mathfrak A$ is finite, the measure of $\mathbf
P(\mathfrak A)$ is 0 and  $\mathbf P(\mathfrak A)$ is nowhere dense in $\mathbb
I^2$} (cf. op. cit.). Therefore, \emph{plots of finite automata are  Lebesgue
measure 0 nowhere
dense closed subsets of the unit square $\mathbb I^2$}; thus they can not contain sets of positive
measure, but they may contain lines. The goal of the paper is to prove
that if $\mathfrak A$ is a finite automaton then smooth curves which lies completely
in  $\mathbf
P(\mathfrak A)$ (thus in $\mathbf{LP}(\mathfrak A)$, cf. further Theorem \ref{thm:main}) can only be straight lines. Moreover, we
will prove that \emph{if
finite automata plots are considered as subsets of  the unit torus $\T^2$ in $\R^3$ then smooth
curves lying in the plots 
can only be torus windings}. For this purpose we will need some extra information
(which follows)
about torus knots.
\subsection{Torus knots, torus links and linear flows on torus}
\label{ssec:knot}
Further in the paper we will need only few concepts concerning torus
knots theory; details may be found in numerous books on knot theory, see
e.g. \cite{Fox,Mansurov-knots}. For our purposes it is enough to recall only
two notions, the knot and the link. Recall that a knot is a smooth embedding of a circle
$\mathbb S$ into $\R^3$ and a link is a smooth embedding of several disjoint
circles in $\R^3$, cf. \cite{Mansurov-knots}. We will  consider only special
types of knots and links, namely, torus knots and torus links. Informally, a torus knot is a smooth
closed curve without intersections which lies completely in the surface of
a torus $\mathbb T^2\subset \R^3$, and a link (of torus knots)  is a collection
 of (possibly knotted) torus  knots, see e.g. \cite[Section 26]{DuFomNov_ModGeo}
for formal definitions.  

We also need a notion of a cable
of torus. Formally, a cable of torus is  any geodesic on torus. Recall that geodesics on
torus $\T^2$ are images of straight lines in $\R^2$ under the mapping $(x;y)\mapsto(x\md1;y\md1)$
of $\R^2$ onto $\T^2=\R^2/\Z\times\Z$,
cf., e.g., \cite[Section 5.4]{Misch-Fom}. 
\begin{defn}[Cable of the torus]
\label{def:cable}
A \emph{cable of the torus} is an image of a straight line in $\R^2$ under
the map $\md 1\colon (x;y)\mapsto (x\md 1;y\md 1)$ of the Euclidean plain
$\R^2$ onto the 2-dimensional real torus $\mathbb T^2=\R^2/\Z\times\Z=\mathbb
S\times\mathbb S\subset \R^3$. If the line is defined by the equation $y=ax+b$ we say
that $a$ is a \emph{slope} of the cable $\mathbf C(a,b)$. We denote via  $\mathbf
C(\infty,b)$ a cable which corresponds to the line $x=b$, the \emph{meridian},
and say that the
slope is $\infty$ in this case. Cables $\mathbf C(0,b)$ of slope 0 (i.e., the ones that correspond
to straight lines $y=b$) are called  \emph{parallels}.
\end{defn}
In dynamics, cables of torus $\T^2$ are  viewed as orbits of \emph{linear
flows on torus}; that is, of dynamical systems on $\T^2$ defined by a pair
of differential equations of the form $\frac{dx}{dt}=\beta;\frac{dy}{dt}=\alpha$ on $\T^2$, whence,
by a pair of parametric equations $x=(\beta t+\tau)\md1; y=(\alpha t+\sigma)\md1$
in Cartesian coordinates, cf. e.g. \cite[Subsection
4.2.3]{Hass-Katok_First}.
\begin{note}
\label{note:cable}
It is well known that a cable defined by the straight line $y=ax+b$ is dense in  $\T^2$ if and
only if $-\infty<a<+\infty$ and the slope $a=\frac{\alpha}{\beta}$ is irrational, see e.g.
\cite[Proposition 4.2.8]{Hass-Katok_First} or \cite[Section 5.4]{Misch-Fom}. 
\end{note}
Given a Cartesian coordinate system $XYZ$ of $\R^3$, a torus   can be obtained by rotation around $Z$-axis of a circle which lies in the plain $XZ$. If a radius  of the circle is $r$ and the circle is centered
at a point lying in $X$-axis at a distance $R$ from the origin, then in cylindrical coordinates 
$(r_0,\theta,z)$ of $\R^3$ (where $r_0$ is a radius-vector in Cartesian
coordinate system $XY$, $\theta$ is  an angle of the radius-vector
in coordinates $XY$, $z$ is a $Z$-coordinate in Cartesian coordinate system $XYZ$) the torus
is defined by the equation
$(r_0-R)^2 + z^2=r^2$   and a cable (with a rational slope $\frac{\alpha}{\beta}$
where $\alpha\in\Z$ and $\beta\in\N$) of the torus
is defined by the system of parametric equations (with parameter $t\in\R$)
of the form
\begin{equation}
\label{eq:note_cable}
\left[ \begin{array}{c} r_{0} \\ \theta  \\ z \end{array} \right]=\left[ \begin{array}{c} R+r\cos \left( \frac{\alpha}{\beta}t+\omega
\right) \\ t
\\ r\sin \left( \frac{\alpha}{\beta}t+\omega
\right) \end{array} \right],\; t\in\R.
\end{equation}
The cable defined by the above equations  winds $\beta$ times around $Z$-axis and $|\alpha|$ times around
a circle in the interior of the torus (the sign of $\alpha$ determines whether
the rotation is clockwise or counter-clockwise), see for an example of the
corresponding torus knot Figures \ref{fig:lin-R}
and  \ref{fig:lin-T} where $\alpha=5$ and $\beta=3$. Letting $\omega$ in
the above equations take a finite number of values we get an example of  torus
link, see e.g. Figures \ref{fig:link-square} and \ref{fig:link-torus} which
illustrate a link consisting of a pair of torus knots whose slopes 
are $\frac{3}{5}$. Note that Figures \ref{fig:2link-square} and \ref{fig:2link-torus}
illustrate  a union of two distinct torus links (of two and of three
knots respectively)  rather than a single torus
link of 5 knots. Finally, due to the above representation of a torus link
in the form of equations in cylindrical coordinates, we naturally associate
the torus link consisting of $N$ cables with slopes $\frac{\alpha}{\beta}$
to a family of complex-valued functions $\psi_k\:\R\>\C$ of real variable $t\in\R$
$$
\left\{\psi_j(t)=e^{i(\frac{\alpha}{\beta}t+\omega_j)}\:
j=0,1,2,\ldots,N-1\right\},
$$
where $i$ stands for imaginary unit $i\in\C$: $i^2=-1$.

\begin{figure}[ht]
\begin{minipage}[b]{.45\linewidth}
\begin{quote}\psset{unit=0.3cm}
 \begin{pspicture}(0,0)(15,15)
 \newrgbcolor{emerald}{.31 .78 .47}
 \newrgbcolor{el-blue}{.49 .98 1}
 \newrgbcolor{baby-blue}{.54 .81 .94}
 \newrgbcolor{periwinkle}{.8 .8 1}
  \psgrid[unit=0.359cm,gridwidth=1.2pt,subgridwidth=1.2pt,gridcolor=yellow,subgridcolor=yellow,griddots=0,
  subgriddiv=3,gridlabels=0pt](-4,0)(11,15)
  \psline[unit=0.359cm,linecolor=periwinkle,linewidth=2pt](-4,10)(-1,15)
  \psline[unit=0.359cm,linecolor=periwinkle,linewidth=2pt](-4,5)(2,15)
  \psline[unit=0.359cm,linecolor=periwinkle,linewidth=2pt](-4,0)(5,15)
  \psline[unit=0.359cm,linecolor=periwinkle,linewidth=2pt](-1,0)(8,15)
  \psline[unit=0.359cm,linecolor=periwinkle,linewidth=2pt](2,0)(11,15)
  \psline[unit=0.359cm,linecolor=periwinkle,linewidth=2pt](5,0)(11,10)
  \psline[unit=0.359cm,linecolor=periwinkle,linewidth=2pt](8,0)(11,5)
 \end{pspicture}
\end{quote}
\caption{%
{A limit plot of the function $f(z)= \frac{5}{3}z$, $z\in\Z_2$, in $\mathbb R^2$}}
\label{fig:lin-R}
\end{minipage}\hfill
\begin{minipage}[b]{.45\linewidth}
\includegraphics[width=0.92\textwidth,natwidth=610,natheight=642]{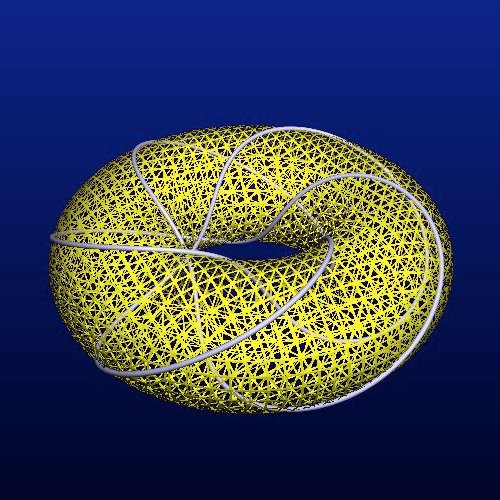}
\caption{A limit plot of the same function on the torus $\mathbb T^2$ \qquad \qquad \qquad}
\label{fig:lin-T}
\end{minipage}
\end{figure}

\section{Plots of  finite automaton functions: Constant and affine cases}
\label{sec:aff}
In this section we completely describe limit plots of finite automata maps of the
forms $z\mapsto c$ (constant maps), $z\mapsto az$ (linear maps) and $z\mapsto az+b$ (affine maps), where $a,b,c$ are
some (suitable) $p$-adic integers and the variable $z$ takes values in $\Z_p$.
\subsection{Limit plots of constants}
\label{ssec:const}
Recall that an automaton $\mathfrak A(s_0)=\langle\Cal I,\Cal S,\Cal O,S,O, s_0\rangle$ is called \emph{autonomous} once neither its state update function
$S$ nor its output function $O$ depend on input; i.e., when $s_{i+1}=S(s_i),
\xi_i=O(\chi_i,s_i)=O(s_i)$ ($i=0,1,2,\ldots$), cf. Fig. \ref{fig:Transd-sc}.

It is clear that an autonomous automaton function is a constant; however
a  limit plot of this function is not necessarily a straight line. For instance, the limit
plot of a constant $c\in\Z_p$  is the whole unit square $\mathbb I^2$ once $c=\sum_{i=0}^\infty\alpha_ip^i$
where 
the infinite word $u=\ldots\alpha_2\alpha_1\alpha_0$ over $\F_p$ is such that
every non-empty finite word $w=\gamma_{k-1}\gamma_{k-2}\ldots\gamma_{0}$ over $\F_p$ occurs
as a subword in $u$; that is, if there exist a finite word $v$ and an infinite
word $s$ over $\F_p$ such that $u$ is a concatenation of $v$, $w$ and $s$:
$u=swv$, cf. \cite{me:Discr_Syst}.

On the other hand,
once an autonomous automaton $\mathfrak A$ is finite, the corresponding infinite output
word must necessarily be eventually periodic. That is, $c=\alpha_0+\alpha_1p+\cdots+\alpha_{r-1}p^{r-1}+(\beta_0+\beta_1p+\cdots+\beta_{t-1}p^{t-1})\cdot
\sum_{j=0}^\infty p^{r+tj}$ for suitable $\alpha_i,\beta_j\in\F_p$; therefore a finite autonomous
automaton function is a rational constant, i.e., $c\in\Z_p\cap\Q$, cf. Propositions \ref{prop:p-adic-rat}
and \ref{prop:fin-auto}.

Furthermore, the numbers that correspond to  (sufficiently long) finite output words are then  all  the form 
$$
0.\beta_k\beta_{k-1}\ldots\beta_{0}\beta_{t-1}\beta_{t-2}\ldots\beta_{0}\beta_{t-1}\beta_{t-2}\ldots\beta_{0}
\ldots\beta_{t-1}\beta_{t-2}\ldots\beta_{0}\alpha_{r-1}\alpha_{r-2}\ldots\alpha_{0}
$$
for $k=0,1,\ldots,t-1$. Consequently, the limit plot of the automaton (in $\R^2$)
consists of $t$ pairwise
parallel straight lines which correspond to the numbers 
$$
0.\beta_k\beta_{k-1}\ldots\beta_{0}\beta_{t-1}\beta_{t-2}\ldots\beta_{0}\beta_{t-1}\beta_{t-2}\ldots\beta_{0}\ldots=
0.\beta_k\beta_{k-1}\ldots\beta_{0}(\beta_{t-1}\beta_{t-2}\ldots\beta_{0})^\infty
$$
where $k=0,1,\ldots,t-1$, cf. Subsection \ref{ssec:plots}; or (which is the same) to the numbers
$0.(\beta_k\beta_{k-1}\ldots\beta_{0}\beta_{t-1}\beta_{t-2}\ldots\beta_{k+1})^\infty$.
That is, all the lines from the limit plot are $y=p^\ell h\md 1$, $\ell\in\N_0$,
for any line $y=h$ belonging to the limit plot; thus the number of lines in the
limit plot does not exceed $t$. Respectively, being considered as a point set on the torus $\T^2$, the limit plot
consists of not more than $t$ parallels, cf., e.g., Figures \ref{fig:parallel-square} and \ref{fig:parallel-torus}.

Now we present a  more formal argument and derive a little bit more information about the number
of lines in the limit plot. Given $q\in\Z_p\cap\Q$, represent $q$ as an irreducible
fraction $q=a/b$ for suitable $a\in\Z$, $b\in\N$. Note that $p\nmid b$ since
$q\in\Z_p$. 
Denote
\begin{multline}
\label{eq:def_C}
\mathbf C(a/b)=\text{limit points of}
\left\{\left(p^\ell\cdot\left(1-\frac{a}{b}\right)\right)\md1\colon\ell=0,1,2,\ldots\right\}
=\\
\text{limit points of}\left\{\left(-p^\ell\cdot\frac{a}{b}\right)\md1\colon\ell=0,1,2,\ldots\right\}.
\end{multline}
Since $a/b\in\Z_p\cap\Q$, by Proposition \ref{prop:p-adic-rat} a $p$-adic canonical form of $a/b$ is
\begin{equation}
\label{eq:q-can}
a/b=\alpha_0+\alpha_1p+\cdots+\alpha_{r-1}p^{r-1}
+(\beta_0+\beta_1p+\cdots+\beta_{t-1}p^{t-1})\cdot
\sum_{j=0}^\infty p^{r+tj}
\end{equation}
for suitable $\alpha_i,\beta_m\in\{0,1,\ldots,p-1\}$, or, in other words,
the infinite word that corresponds to $a/b$ is 
$(\beta_{t-1}\ldots\beta_{0})^\infty\alpha_{r-1}\ldots\alpha_{0}$.
Then from Proposition \ref{prop:r-p-repr-qz} it follows that 
\begin{multline*}
(a/b)\md 1=(p^r\cdot0.(\hat\beta_{t-1}\ldots\hat\beta_0)^\infty)\md1=\\
0.(\hat\beta_{t-1-\bar r}\hat\beta_{t-2-\bar r}\ldots
\hat\beta_{0}\hat\beta_{t-1}\hat\beta_{t-2}\ldots\hat\beta_{t-\bar
r})^\infty\md1,
\end{multline*}
where $\hat\beta_i=p-1-\beta_i$ , $i=0,1,2,\ldots,t-1$, 
and $\bar r$ is the least non-negative residue of $r$ modulo $t$ if $t>1$
or $\bar r=0$ if otherwise. From here in view of \eqref{eq:-md1} we deduce that
$$
(-a/b)\md1=0.(\beta_{t-1-\bar r}\beta_{t-2-\bar r}\ldots
\beta_{0}\beta_{t-1}\beta_{t-2}\ldots\beta_{t-\bar
r})^\infty\md1
$$
and thus
\begin{multline*}
\mathbf C(a/b)=\\
\left\{0.(\beta_{t-1-\ell}\beta_{t-2-\ell}\ldots
\beta_{0}\beta_{t-1}\beta_{t-2}\ldots\beta_{t-\ell})^\infty\md1\colon\ell=0,1,2,\ldots,t-1 \right\}=\\
\left\{\frac{\nm(\upsilon)}{p^t-1}\:
\upsilon\in\{\hat\zeta_{t-1}\hat\zeta_{t-2}\ldots\hat\zeta_{0},
\hat\zeta_{t-2}\hat\zeta_{t-3}\ldots\hat\zeta_{0}\hat\zeta_{t-1},
\hat\zeta_{t-3}\hat\zeta_{t-4}\ldots\hat\beta_{0}\hat\zeta_{t-1}\hat\zeta_{t-2},
\ldots\}\right\},
\end{multline*}
where $(a/b)\md1=(\zeta_0+\zeta_1\cdot p+\cdots+\zeta_{t-1}\cdot p^{t-1})(p^t-1)^{-1}$
(cf. Proposition \ref{prop:p-repr} and Corollary \ref{cor:r-p-repr-qz}).
Now we can suppose that $t$ is a period length of the rational $p$-adic integer
$a/b\in\Z_p\cap\Q$ (cf. Subsection \ref{ssec:p-adic});  then  in view of Proposition \ref{prop:mult-per} we conclude that
\begin{multline}
\label{eq:C}
\mathbf C({a}/{b})=\left\{(-p^{\ell}\cdot({a}/{b}))\md1\colon \ell=0,1,\ldots,(\mlt_bp)-1\right\}=\\
\left\{0.(w)^\infty\md1\colon w \ \text{runs through all cyclic shifts of the word
}\ \beta_{(\mlt_bp)-1}\ldots\beta_0\right\}=\\
\left\{0.(v)^\infty\md1\colon v \ \text{runs through all cyclic shifts of the word
}\ \hat\zeta_{(\mlt_bp)-1}\ldots\hat\zeta_0\right\}=\\
\left\{\left(-p^{\ell}\cdot\frac{d}{p^{\mlt_bp}-1}\right)\md1\colon \ell=0,1,\ldots,(\mlt_bp)-1\right\}
\end{multline}
since 
\begin{align*}
1-\frac{\zeta_0+\zeta_1p+\cdots+\zeta_{t-1}p^{t-1}}{p^t-1}&=
\frac{\hat\zeta_{t-1}+\hat\zeta_0p+\hat\zeta_1p^2+\cdots+\hat\zeta_{t-2}p^{t-1}}{p^t-1}\
\text{and}\\
p\cdot\frac{\hat\zeta_0+\hat\zeta_1p+\cdots+\hat\zeta_{t-1}p^{t-1}}{p^t-1}&=\hat\zeta_{t-1}+
\frac{\hat\zeta_{t-1}+\hat\zeta_0p+\hat\zeta_1p^2+\cdots+\hat\zeta_{t-2}p^{t-1}}{p^t-1}.
\end{align*}Note that $0.(w)^\infty\md1=0.(w)^\infty$ except of the case
when $t=1$ and $w$ is a single-letter word that consists of the only letter
$p-1$ (in the latter case $0.(w)^\infty=1$ and thus $0.(w)^\infty\md1=0$).
Similarly, $0.(v)^\infty\md1=0.(v)^\infty$ except of the case when $a/b\in\Z$
and thus $\zeta_0=\ldots=\zeta_{t-1}=0$ 
(so $\hat\zeta_0=\ldots=\hat\zeta_{t-1}=p-1$ and  $0.(v)^\infty=1$). But this case happens if and only if 
$a/b
\in\Z$; i.e., when  $\mathbf C(a/b)=\{0\}$. 

We now summarize all these considerations in a proposition:
\begin{prop}
\label{prop:const}
Let $f_{\mathfrak A}\colon z\mapsto q$ be an automaton function of a finite
automaton $\mathfrak A$ \textup{(therefore $q\in\Z_p\cap\Q$ by Proposition
\ref{prop:fin-auto})}; 
then $\mathbf{LP}(\mathfrak A)\subset\T^2$ is a disjoint union
of 
$t$ parallels $\mathbf C(0,e)$, $e\in\mathbf C(q)$, and $t$ is a period length
of $q$ \textup{(cf. \eqref{eq:def_C} and
\eqref{eq:C})}.



\end{prop}

\begin{note}
\label{note:const}
In conditions of Proposition \ref{prop:const} the constant $q\in\Z_p\cap\Q$
can be represented as an irreducible fraction $q=a/b$ where $a\in\Z$, $b\in\N$,
$p\nmid b$ (we put $b=1$ and $a=0$ if $q=0$). Then
the limit plot $\mathbf{LP}(\mathfrak A)\subset\T^2$ is a torus link that consists of $t=\mlt_bp$
trivial torus cables (parallels) with slopes $0$; to the link
there corresponds a collection of $t$ complex constants (which are $b$-th roots of 1) 
$$
\left\{\psi_\ell=e^{-2\pi i p^\ell q
}\:\ell=0,1,\ldots,(\mlt_bp)-1\right\},
$$
where $i$ stands for imaginary unit $i\in \C$: $i^2=-1$ (cf. Subsection \ref{ssec:knot}).

Being considered in the unit real square $\mathbb I^2$, the limit plot $\mathbf{LP}(\mathfrak
A)$ is a collection of $t=\mlt_bp$ segments of straight lines $y=c(t,k,u)$
that cross $\mathbb I^2$,
where
\begin{multline}
\label{eq:const}
c(t,k,u)=\left(-p^k\cdot\frac{u}{p^t-1}\right)\md1=\\
0.(\hat\zeta_{t-1-k}\hat\zeta_{t-2-k}\ldots
\hat\zeta_{0}\hat\zeta_{t-1}\hat\zeta_{t-2}\ldots\hat\zeta_{t-k})^\infty\md1;\
k=0,1,\ldots,t-1.
\end{multline}
Here $q\md1=u(p^t-1)^{-1}$, $0\le u\le p^t-2$, and a base $p$-expansion
of $u$ is $u=\zeta_0+\zeta_1\cdot p+\cdots+\zeta_{t-1}\cdot p^{t-1}$
(cf. Proposition \ref{prop:p-repr}); $\hat\zeta=p-1-\zeta$ for $\zeta\in\{0,1,\ldots,p-1\}$.
In other words, all the constants $c(t,k,u)$ are of the form
\begin{equation}
\label{eq:const-frac} 
c(t,k,u)=0.\upsilon^\infty\md1=\frac{\nm(v)}{p^t-1}\md1,
\end{equation}
where $v$ runs
trough all
cyclic shifts of the word $\hat\zeta_{t-1}\hat\zeta_{t-2}\ldots\hat\zeta_{0}$; that is,
$v\in\{\hat\zeta_{t-1}\hat\zeta_{t-2}\ldots\hat\zeta_{0}, 
\hat\zeta_{t-2}\hat\zeta_{t-3}\ldots\hat\zeta_{0}\hat\zeta_{t-1},\ldots\}$.

If $q$ is represented in a $p$-adic canonical form \eqref{eq:q-can} rather
than in a form of Proposition \ref{prop:p-repr},
then all the lines of the limit plot can be represented as 
\begin{equation}
\label{eq:const-r}
y=0.(\beta_{t-1-\ell}\beta_{t-2-\ell}\ldots
\beta_{0}\beta_{t-1}\beta_{t-2}\ldots\beta_{t-\ell})^\infty\md1;\ \ell=0,1,2,\ldots,t-1.
\end{equation}
Note that we may omit $\md 1$ in \eqref{eq:const-frac} and in \eqref{eq:const-r} in all cases but the
case when simultaneously the length
$t$ of the period is 1 and $\hat\zeta_0=p-1$ (respectively, $\beta_0=p-1$); but $q\in\Z$ in that case and therefore $\mathbf C(q)=\{0\}$.

\end{note}
The following property of the set $\mathbf C(q)$ will be used in further
proofs:
\begin{cor}
\label{cor:const}
Given $q_1,q_2\in\Z_p\cap\Q\cap[0,1)$, the following alternative holds: Either
$\mathbf C(q_1)=\mathbf C(q_2)$ or $\mathbf C(q_1)\cap\mathbf C(q_2)=\emptyset$.
\end{cor}
\begin{proof}[Proof of Corollary \ref{cor:const}]
The result is clear enough since the numbers that constitute $\mathbf C(q)$
are exactly all  numbers whose base-$p$ expansions are of the form $0.(u)^\infty$
where $u$ runs through all cyclic shifts of the finite word $w$ which is the (shortest)
period of $q\md1$, cf. Note \ref{note:const}; nonetheless we give a formal
proof which follows.

Given $q_i\in\Z_p\cap\Q\cap[0,1)$, $i=1,2$, represented as irreducible
fractions $q_i=a_i/b_i$ whose denominators $b_i$ are co-prime to $p$, let
$\mathbf C(q_1)\cap\mathbf C(q_2)\ne\emptyset$; then 
$p^{\ell_1}(a_1/b_1)=p^{\ell_2}(a_2/b_2)$ for suitable $\ell_1,ell_2\in\N_0$.
If $\ell_1=\ell_2$ then $a_1/b_1=a_2/b_2$ and thus $\mathbf C(q_1)=\mathbf C(q_2)$. Let $\ell_1>\ell_2$, then  $p^{\ell_1-\ell_2}a_1b_2=a_2b_1$; so
since $\gcd(b_1,p)=1$ we conclude that 
$a_2=p^{\ell_1-\ell_2+s}a_2^\prime$ for a suitable $s\in\N_0$ and $a^\prime_2\in\Z$
such that $\gcd(a_2^\prime,p)=1$. Therefore necessarily $a_1=p^sa^\prime_1$
where $\gcd(a^\prime_1,p)=1$ since $\gcd(b_1,p)=\gcd(b_2,p)=1$. But then
we conclude that 
$p^{\ell_1-\ell_2+s}a_1^\prime b_2= p^{\ell_1-\ell_2}a_1b_2=a_2b_1=p^{\ell_1-\ell_2+s}a_2^\prime
b_1$ and therefore $q_1=p^sq$, $q_2=p^{\ell_1-\ell_2+s}q$ where $q=a_1^\prime/b_1=a_2^\prime/b_2$.
Hence $\mathbf C(q_1),\mathbf C(q_2)\subset \mathbf C(q)$; the inverse inclusion
also holds since $\mathbf C(p^\ell q)=\mathbf C(q)$ for any $q\in\Z_p\cap\Q$
by. e.g., \eqref{eq:const-r}.
\end{proof}
\begin{exmp}
\label{ex:const}
Let $p=2$ and $q=2/7$. Then $\mlt_72=3$ and the limit plot consists of 3 lines. The binary infinite word that corresponds to the 2-adic canonical representation
of $2/7$ is $(011)^\infty10$, so the period of $2/7$ is $011$, the pre-period is
$01$, and  $u=2=0+1\cdot 2+0\cdot 2^2$.  Therefore the  tree lines of the
limit plot are: $y=0.(101)^\infty=5/7=(-2/7)\md1=c(3,0,2)$, $y=0.(011)^\infty=6/7=(-1/7)\md1=c(3,2,2)$, $y=0.(110)^\infty=3/7=(-4/7)\md1=c(3,1,2)$. The limit plot (on the unit square
and on the torus) is illustrated by Figures \ref{fig:parallel-square} and
\ref{fig:parallel-torus} accordingly.
\end{exmp}

\begin{figure}[ht]
\begin{minipage}[b]{.45\linewidth}
\begin{quote}\psset{unit=0.3cm}
 \begin{pspicture}(0,0)(15,15)
 \newrgbcolor{emerald}{.31 .78 .47}
 \newrgbcolor{el-blue}{.49 .98 1}
 \newrgbcolor{baby-blue}{.54 .81 .94}
 \newrgbcolor{erin}{0.498 1 .247}
 \newrgbcolor{ch-green}{0 1 .247}
 \newrgbcolor{spring-bud}{0.655 .988 0}
 \newrgbcolor{lemon}{1 .968 0}
   \psframe[
  linewidth=0.5pt,linecolor=green](-4.8,0)(13.2,18)
  \pscircle[fillstyle=solid,fillcolor=yellow,linewidth=.4pt](-4.8,0.1){.3}
  \psline[unit=0.359cm,linecolor=baby-blue,linewidth=2pt](-4,12.857)(11,12.857)
\psline[unit=0.359cm,linecolor=spring-bud,linewidth=2pt](-4,12.143)(11,12.143)
\psline[unit=0.359cm,linecolor=lemon,linewidth=2pt](-4,6.43)(11,6.43)
 \end{pspicture}
\end{quote}
\caption{%
{A limit plot of the constant function $f(z)= \frac{2}{7}$ ($z\in\Z_2$), in $\mathbb R^2$}}
\label{fig:parallel-square}
\end{minipage}\hfill
\begin{minipage}[b]{.45\linewidth}
\includegraphics[width=0.92\textwidth,natwidth=610,natheight=642]{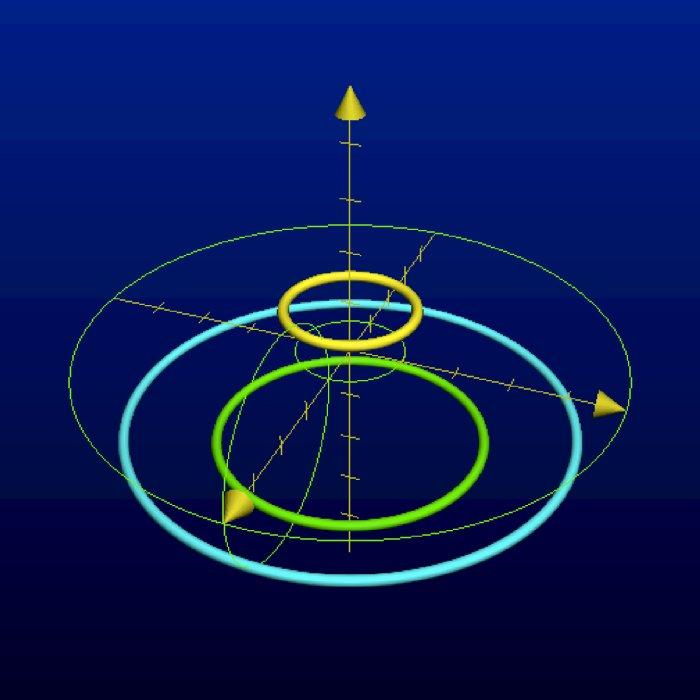}
\caption{A limit plot of the same function on the torus $\mathbb T^2$\quad \qquad \qquad \qquad \qquad \qquad}
\label{fig:parallel-torus}
\end{minipage}
\end{figure}

\subsection{Limit plots of linear maps} 
\label{ssec:linear}
In this subsection we consider limit plots of
linear maps $z\mapsto cz$ ($z\in\Z_p$) which are finite automaton functions.
By Proposition \ref{prop:fin-auto}, the latter takes place if and and only
if $c\in\Z_p\cap\Q$.
\begin{prop}
\label{prop:mult-p}
Given  $c\in\Z_p\cap\Q$, represent $c=a/b$, where $a\in\Z$, $b\in\N$, $a,b$
are coprime, $p\nmid b$. If  $\mathfrak A$ is an automaton such that 
$f_{\mathfrak A}(z)=cz$ $(z\in\Z_p)$ then 
$\mathbf{LP}(\mathfrak A)=\{(x\md1;(cx)\md1)\:x\in\R\}
=\mathbf C(c,0)$ is a cable \textup(with a slope $c$\textup) of the unit 2-dimensional real torus $\mathbb T^2$.
For every $c\in\Z_p\cap\Q$ the automaton $\mathfrak A$ may
be taken a finite.
\end{prop}
\begin{proof}[Proof of Proposition \ref{prop:mult-p}]
By Proposition \ref{prop:fin-auto},
the map $
z\mapsto cz$ on $\Z_p$ is an automaton function of a finite  
automaton if and only if $c\in\Z_p\cap\Q$.



Given $x\in[0,1)$, take $z\in\Z_p$ such that
$\lim_{i\to\infty}z\md p^{k_i}/p^{k_i}=x$ 
for a suitable strictly increasing sequence $k_1,k_2,\ldots\in\N$.
As $c\in\Z_p\cap\Q$, 
then 
$c=u+v/(p^t-1)$ for suitable $u\in\Z$, $t\in\N$, $v\in\{0,1,2,\ldots,p^t-2\}$,
by Proposition \ref{prop:p-repr}. If $t>1$, then considering residues of $k_i$
modulo $t$ we see that at least one residue (say, $\ell\in\{0,1,\ldots,t-1\}$)
occurs in the sequence $k_1,k_2,\ldots$ infinitely many times. Therefore
$\lim_{j\to\infty}z\md p^{r_jt+\ell}/p^{r_jt+\ell}=x$ for a respective strictly
increasing sequence $r_1,r_2,\ldots\in\N$. The latter equality trivially holds when $t=1$: one just takes $r_j=k_j$ and $\ell=0$. So further we assume
that $k_i=r_it+\ell$, $i=1,2,\ldots$.

For $i=1,2,\ldots$ we have that
\begin{multline}
\label{eq:fr_mod}
\frac{(cz)\md p^{k_i}}{p^{k_i}}=\frac{1}{p^{k_i}}((c\md p^{k_i})(z\md p^{k_i}))\md{p^{k_i}}=\\
\left(c\md p^{k_i}\cdot\frac{z\md p^{k_i}}{p^{k_i}}\right)\md 1
\end{multline}
As $0\le\ell<t$ and $k_i=r_it+\ell$, we have that
\begin{equation}
\label{eq:fr_mod1}
c\md p^{k_i}=\left(u+\frac{v}{p^t-1}\right)\md p^{k_i}=
\left(u-v\cdot\frac{p^{(r_i+1)t}-1}{p^t-1}\right)\md p^{k_i}
\end{equation}
Note that the argument of $\md$ in the right-hand side of \eqref{eq:fr_mod1}
is negative once $i$ is sufficiently large; therefore  once $i$ is large enough then
$$
\left(u-v\cdot\frac{p^{(r_i+1)t}-1}{p^t-1}\right)\md p^{k_i}=
Lp^{k_i}+u-v\cdot\frac{p^{(r_i+1)t}-1}{p^t-1}
$$ 
for a suitable $L\in\N$
which does not depend on $i$ (actually it is not
difficult to see that $L=\lceil
vp^{t-\ell}(p^t-1)^{-1}\rceil$). Thus,
\begin{multline}
\label{eq:fr_mod2}
\left(c\md p^{k_i}\cdot\frac{z\md p^{k_i}}{p^{k_i}}\right)\md 1=\\
\left(\left(Lp^{k_i}+u-v\cdot\frac{p^{(r_i+1)t}-1}{p^t-1}\right)\cdot\frac{z\md p^{k_i}}{p^{k_i}}\right)\md
1=\\
\left(L\cdot z\md p^{k_i}+u\cdot\frac{z\md p^{k_i}}{p^{k_i}}+
\frac{v}{p^t-1}\cdot\frac{z\md p^{k_i}}{p^{k_i}}
-\frac{vp^{t-\ell}}{p^t-1}\cdot z\md p^{k_i}
\right)\md 1=\\
\left(c\cdot\frac{z\md p^{k_i}}{p^{k_i}}
-\frac{vp^{t-\ell}}{p^t-1}\cdot z\md p^{k_i}
\right)\md 1
\end{multline}


Firstly we note that given $w\in\N_0$, $r\in\N$
\begin{equation}
\label{eq:recurr}
\frac{wp^{rt}}{p^t-1}\md1=\left(w\cdot\frac{p^{rt}-1}{p^t-1}+\frac{w}{p^t-1}\right)\md1=
\left(\frac{w}{p^t-1}\right)\md1
\end{equation} 
as $p^t-1$ is a factor of $p^{rt}-1$. 

Secondly, put $\bar z=p^{t-\ell}z$, then $p^{t-\ell}(z\md p^{r_it+\ell})=\bar
z\md{p^{(r_i+1)t}}$ and
$$
x=\lim_{i\to\infty}\frac{z\md p^{k_i}}{p^{k_i}}=\lim_{i\to\infty}\frac{\bar z\md p^{r_it}}{p^{r_it}},
$$
so in \eqref{eq:fr_mod2}
$$
\frac{vp^{t-\ell}}{p^t-1}\cdot z\md p^{r_it+\ell}=\frac{v}{p^t-1}\cdot \bar z\md p^{(r_i+1)t}
$$
and $\bar z=zp^{t-\ell}\in\Z_p$ 
(recall that $k_i=r_it+\ell$ where $\ell\in\{0,1,\ldots,t-1\}$).

Let $\bar z=\zeta_0+\zeta_1p^t+\zeta_2p^{2t}+\cdots$ be a base-$p^t$ representation
of $\bar z$ (that is, $\zeta_j\in\{0,1,\ldots,p^t-1\}$); then by combining
\eqref{eq:fr_mod} and \eqref{eq:fr_mod2} with \eqref{eq:recurr} we get
\begin{multline}
\label{eq:fr_mod_wt}
\frac{(cz)\md p^{k_i}}{p^{k_i}}=\left(c\cdot\frac{z\md p^{k_i}}{p^{k_i}}
-\frac{v}{p^t-1}\cdot 
\wt_{p^t}(\bar z\md p^{(r_i+1)t})
\right)\md1=\\
\left(c\cdot\frac{z\md p^{k_i}}{p^{k_i}}
-\frac{v}{p^t-1}\cdot (\wt_{p^t}(\bar z\md p^{(r_i+1)t}))\md(p^t-1)
\right)\md1
\end{multline}
where $\wt_{p^t}$ stands for a $p^t$-weight of a natural number, that is, the sum
of digits of the number in its base-$p^t$ representation; i.e., $\wt_{p^t}(\bar
z\md{p^{(r_i+1)t}})= \zeta_0+\zeta_1+\cdots+\zeta_{i}\in\N_0$.
Therefore every limit point of the sequence $((cz)\md p^{k_i}/p^{k_i})_{i=1}^\infty$
is of the form
\begin{equation}
\label{eq:lim}
\left(cx
+\frac{vw}{p^t-1}
\right)\md1
\end{equation}
for a suitable $w\in\{0,1,\ldots,p^t-2\}$.

We claim that, on the other hand, given $x\in[0,1)$ 
and $h\in\{0,v, 2v\md{(p^t-1)},\ldots,((p^t-2)v)\md{(p^t-1)}\}$ (that is,
$h$ lying in the ideal $\langle v\rangle$ of the residue ring $\Z/(p^t-1)\Z$ generated by $v$) there exists $z\in\Z_p$ and a strictly increasing sequence $k_1,k_2,\ldots,\in\N$ such that 
\begin{align*}
x&=\lim_{i\to\infty}\frac{z\md p^{k_i}}{p^{k_i}}, \ \text{and}\\
\frac{(cz)\md p^{k_i}}{p^{k_i}}&=\left(c\cdot\frac{z\md p^{k_i}}{p^{k_i}}
+\frac{h}{p^t-1}
\right)\md1.
\end{align*}
Indeed, take $z\in\Z_p$ and $k_i=r_it+\ell$ as above; then  all limit points of the sequence
$((cz)\md p^{k_i}/p^{k_i})_{i=1}^\infty$ are of the form \eqref{eq:lim} for, say, 
$w=w_1,\ldots,w_s\in\{0,1,\ldots,p^t-2\}$. If $h\equiv vw_j\pmod{(p^t-1)}$
for some $j=1,2,\ldots, s$, 
then there is nothing to prove; if $h\not\equiv vw_j\pmod{(p^t-1)}$
for all $j=1,2,\ldots, s$ then we tweak $z$ as follows. As the point of the
form \eqref{eq:lim} for $w=w_1$ is a limit point of the sequence  $((cz)\md p^{k_i}/p^{k_i})_{i=1}^\infty$
then $(-vw_1)\md{(p^t-1)}$ occurs in the sequence  $((-v((\wt_{p^t}\bar z)\md p^{(r_i+1)t}))\md{(p^t-1)})_{i=1}^\infty$
infinitely many times (cf. \eqref{eq:fr_mod_wt}); so some $\bar w\in\{0,1\ldots,p^t-2\}$ such that $v\bar w\equiv vw_1\pmod{(p^t-1)}$ occurs in the
sequence $((\wt_{p^t}\bar z)\md p^{(r_i+1)t})\md{(p^t-1)}$ infinitely many times:
$$
\bar w=((\wt_{p^t}\bar z)\md p^{(r_i+1)t})\md{(p^t-1)}=(\zeta_0+\zeta_1+\cdots+\zeta_{i})\md
(p^t-1)
$$
for $i=i_1, i_2,\ldots$ ($1<i_1<i_2<\ldots$).

As $h\in\langle v\rangle$, then $h\equiv
-v\tilde
w\pmod{(p^t-1)}$ for a suitable $\tilde w\in\{0,1,\ldots,p^t-2\}$.
Now put $\tilde z=\zeta_0+\tilde\zeta_1p^t+\zeta_2 p^{2t}+\zeta_3 p^{3t}+\cdots$,
where $\tilde\zeta_1\equiv\zeta_1-\bar w+\tilde w\pmod{(p^t-1)}$; then 
$\tilde w=(\wt_{p^t}\tilde z)\md p^{(r_i+1)t})\md{(p^t-1)}$. But 
$\lim_{i\to\infty}(z\md p^{k_i}/p^{k_i})=\lim_{i\to\infty}(\tilde z\md p^{k_i}/p^{k_i})=x$; so
finally we conclude by \eqref{eq:fr_mod_wt} that 
$$
\lim_{j\to\infty}\frac{(c\tilde z)\md p^{r_{i_j}t+\ell}}{p^{r_{i_j}t+\ell}}=
\left(cx
+\frac{h}{p^t-1}
\right)\md1.
$$
Thus we have shown that
\begin{equation}
\label{eq:au_cable}
\mathbf{LP}(\mathfrak A)=\left\{\left(x;\left(cx
+\frac{e}{p^t-1}
\right)\md1\right)\colon x\in[0,1), e\in\langle v\rangle \right\}.
\end{equation}
But the right-hand side in \eqref{eq:au_cable} is a cable of torus with
slope $c$ since
\begin{equation}
\label{eq:cable}
\left\{\left(x;\left(cx
+\frac{h}{p^t-1}
\right)\md1\right)\colon x\in[0,1), h\in\langle v\rangle \right\}=
\left\{\left(y\md 1;\left(cy
\right)\md1\right)\colon y\in\R\right\}.
\end{equation}
Indeed, if $y_1=y+n$ for some $n\in\Z$ then $y_1\md 1=y\md 1$ and
$(cy_1)\md1=(c(y+n))\md1=(((u+v(p^t-1)^{-1})(y+n))\md1=(cy+vn(p^t-1)^{-1})\md1
=(cy+((vn)\md(p^t-1))\cdot(p^t-1)^{-1})\md1$,
and \eqref{eq:cable} follows. 
This concludes
the proof. 
%
\end{proof}
\begin{exmp}
\label{exmp:lin}
Take $p=2$ and $c=5/3$. Figures \ref{fig:lin-R} and \ref{fig:lin-T} illustrate
the limit plot of the  function $f(z)= (5/3)\cdot z$ 
in $\mathbb I^2$ and in $\T^2$  respectively.
\end{exmp}

\subsection{Limit plots of affine maps}
\label{ssec:plot-affine}
In this subsection  we combine  the  above two cases (constant maps and  linear maps)
into a single one to describe limit plots  of finite automata whose functions
are affine, i.e., of the form $z\mapsto c\cdot z+q$ ($z\in\Z_p$). It is evident that the limit plot  should be a torus link consisting of several disjoint
cables with slopes $c$ since the limit plot of the constant $q$ is a collection
of parallels, cf. Propositions \ref{prop:mult-p} and \ref{prop:const}. 
We will give a formal proof of this claim and find the number of knots in the
link. 

Recall that by Proposition \ref{prop:fin-auto}
the map $z\mapsto c\cdot z+q$ of $\Z_p$ into itself is an automaton function
of some  finite
automaton if and only if $c,q\in\Z_p\cap\Q$. The following proposition shows
that we do not alter the limit plot of the map  once we replace $q$ by $q+n$
for arbitrary $n\in\Z$.
\begin{prop}
\label{prop:add-p-md1}
Given $f\:z\mapsto cz+q$ $(z\in\Z_p)$ where $c,q\in\Z_p\cap\Q$, denote $\bar q=q\md1$, $\bar f\:z\mapsto cz+\bar q$.  Then $\mathbf{LP}(f)=\mathbf{LP}(\bar
f)$.
\end{prop}
\begin{proof}[Proof of Proposition \ref{prop:add-p-md1}]
Indeed,
once $n\in\Z$ then $\lim_{k\to\infty}n \md p^k/p^k\in\{0,1\}$; the limit is
$1$ if and only if $n$ is negative since given a canonical $p$-adic representation
$n=\alpha_0+\alpha_1p+\cdots$ of a negative $n\in\Z$, all $\alpha_i=p-1$  if $i$ is large enough, cf. Subsection \ref{ssec:p-adic}.
Therefore ($\lim_{k\to\infty}(z+n)\md p^k/p^k)\md1=(\lim_{k\to\infty}(z\md p^k+n\md
p^k)\md p^k/p^k)\md1=\lim_{k\to\infty}(z\md p^k/p^k+n\md
p^k/p^k)\md 1=
(\lim_{k\to\infty}z\md p^k/p^k)\md1$ for all $z\in\Z$. 
\end{proof}
Note that  the map $z\mapsto
cz+\bar q$ from the statement of Proposition \ref{prop:add-p-md1} is an automaton function for a suitable finite automaton $\mathfrak
B$ and $\mathbf{LP}(\mathfrak A)=\mathbf{LP}(\mathfrak B)$, where $\mathfrak
A$ is a finite automaton whose automaton function is $f$.

Now we  describe  limit plot of  a special  affine map with $c=1, q\ne 0$.
\begin{lem}
\label{lem:add-p}
Given a finite automaton
$\mathfrak A$ whose automaton function is $f(z)=z+q$ \textup{($q\in\Z_p\cap\Q$
then
)}, the limit plot $\mathbf{LP}(\mathfrak A)\subset\mathbb T^2$ is a link of a finite number of torus knots which
are cables $\mathbf C(1,e)$ where  $e$ is running over $\mathbf C(q)$.
\end{lem}
\begin{proof}[Proof of Lemma \ref{lem:add-p}]
We will prove  that once $\mathfrak A$ is a finite automaton such that
$f_{\mathfrak A}(z)=f(z)=z+q$ then 
\begin{equation}
\label{eq:add-p-main}
\mathbf{LP}(\mathfrak A)=\bigcup_{e\in\mathbf C(q)}\mathbf C(1,e).
\end{equation}
Note  that if $e=0$ and $e\in\mathbf C(q)$ then $\mathbf C(q)=\{0\}$ by Proposition
\ref{prop:const} and there
is nothing to prove. So further we assume that $e\in\mathbf C(q)$ and $e\ne
0$. 

By Proposition \ref{prop:add-p-md1} we may assume that $q\in\Z_p\cap\Q\cap[0,-1)$ then $q=d\cdot(p^t-1)^{-1}-1$ for  suitable $d\in\{0,1\ldots,p^t-2\}$, cf. Proposition \ref{prop:p-repr};
that is, 
$d=\zeta_{t-1}+\zeta_{t-2}p+\cdots+\zeta_{0}p^{t-1}$, where $\zeta_0,\ldots,\zeta_{t-1}\in\{0,1,\ldots, p-1\}$ and therefore
\begin{multline*}
q=
-(\zeta_{t-1}+\zeta_{t-2}p+\cdots+\zeta_{0}p^{t-1})(1+p^t+p^{2t}+\cdots)-1=\\
((p-1-\zeta_{t-1})+(p-1-\zeta_{t-2})p+\cdots+(p-1-\zeta_{0})p^{t-1})(1+p^t+p^{2t}+\cdots)
\end{multline*}
as $(p^t-1)^{-1}=-(1+p^t+p^{2t}+\cdots)$ in $\Z_p$ by Note \ref{note:inverse}. 
Therefore, in $\Z_p$ the rational number $q$ can be represented as  
\begin{multline}
\label{eq:add-p-a}
q=
(\eta_0+\eta_1p+\cdots+\eta_{t-1}p^{t-1})\cdot(1+p^t+p^{2t}+\cdots),
\end{multline}
where 
$\eta_j=p-1-\zeta_{t-1-j}$, $j=0,1,\ldots,t-1$.

 Given $x\in[0,1)$ take a sequence $n_i\in\N_0$, $i=1,2,\ldots$, and a strictly increasing sequence
$k_i\in\N$, $i=1,2,\ldots$, such that $k_i\ge\lfloor\log_pn_i\rfloor+1$, $\lim_{i\to\infty}n_i/p^{k_i}=x$,
and $k_i\md t= s\in\{0,1,\ldots,t-1\}$ for all $i=1,2,\ldots$. This
is always possible since if, e.g., $x=\xi_1p^{-1}+\xi_{2}p^{-2}+\cdots$ for
suitable $\xi_1,\xi_2,\ldots\in\{0,1,\ldots,p-1\}$ then one takes 
$n_i=\xi_1p^{j-1}+\xi_2p^{j-2}+\cdots+
\xi_{j-1}$ where  $j=it+s$ and  put $k_i=it+s$ for $i=1,2,\ldots$.

Considering a sequence $(n_i+q)_{i=0}^\infty$ in $\Z_p$, we see that 
\begin{equation}
\label{eq:add-p}
\frac{(n_i+q)\md p^{k_i}}{p^{k_i}}=\left(\frac{n_i}{p^{k_i}}+\frac{q\md p^{k_i}}{p^{k_i}}\right)\md1
\end{equation}
But $k_i=it+s$, $s\in\{0,1\ldots,t-1\}$; thus 
\begin{multline}
\label{eq:add-p-am}
\left(\lim_{i\to\infty}\left(\frac{q\md p^{k_i}}{p^{k_i}}\right)\right)\md1=
(\lim_{i\to\infty}((\eta_0+\eta_1p+
\cdots+\eta_{s-1}p^{s-1})p^{-s}+\\
(\eta_0+\eta_1p+
\cdots+\eta_{t-1}p^{t-1})\cdot
(p^{-it-s}+p^{(-i+1)t-s}+\cdots+p^{-t-s}))
)\md1=\\
0.(\eta_{s-1}\eta_{s-2}\ldots\eta_0\eta_{t-1}\eta_{t-2}\ldots\eta_{s})^\infty\md1
\end{multline}
if $s\ne 0$, or
\begin{multline}
\label{eq:add-p-am1}
\left(\lim_{i\to\infty}\left(\frac{q\md p^{k_i}}{p^{k_i}}\right)\right)\md1=\\
(\lim_{i\to\infty}(
(\eta_0+\eta_1p+
\cdots+\eta_{t-1}p^{t-1})\cdot
(p^{-it}+p^{(-i+1)t}+\cdots+p^{-t}))
)\md1=\\
0.(\eta_{t-1}\eta_{t-2}\ldots\eta_0)^\infty\md1
\end{multline}
if $s=0$. From \eqref{eq:add-p-am1} and \eqref{eq:add-p-am} it follows that 
$\lim_{i\to\infty}
(q\md p^{k_i}/p^{k_i})\md1
\in\mathbf C(q)$ by \eqref{eq:const} of Note \ref{note:const}.
Thus we have proved that given $x\in[0,1)$ and $e\in\mathbf
C(q)$, necessarily $(x,(x+e)\md1)\in\mathbf{LP}(\mathfrak A)$; so $\mathbf{LP}(\mathfrak A)\supset\mathbf C(1,e)$ for every $e\in\mathbf C(b)$. 

On the other hand, given $z\in\Z_p$ and a strictly increasing sequence $k_1,k_2,\ldots\in\N$,
by combining \eqref{eq:add-p-am1} and \eqref{eq:add-p-am} with \eqref{eq:const} of Note \ref{note:const} we conclude
that all limit points of the sequence $q\md p^{k_i}/p^{k_i}$, $i=1,2,\ldots$, are in $\mathbf
C(q)$ by an argument similar
to the above one. Therefore, limit points of the sequence $(z\md
p^{k_i}/p^{k_i}+q\md p^{k_i}/p^{k_i})\md1$, $i=1,2,\ldots$, are all of the form $(x+e)\md1$, where
$x$ is an limit point of the sequence $z\md p^{k_i}/p^{k_i}$ and $e\in\mathbf
C(q)$. This proves that $\mathbf{LP}(\mathfrak A)\subset\cup_{e\in\mathbf
C(q)}\mathbf C(1,e)$ and that \eqref{eq:add-p-main} is true.

\end{proof}
Now we are ready to prove the main claim of the Section.
\begin{thm}
\label{thm:mult-add-p}
Given $c,q\in\Z_p$, a map $z\mapsto cz+q$ of $\Z_p$ into itself is an automaton function of a finite automaton
if and only if $c,q\in\Z_p\cap\Q$.
Given a finite automaton
$\mathfrak A$ whose automaton function is $f(z)=cz+q$ for 
$c,q\in\Z_p\cap\Q$, represent $c,q$ as irreducible fractions 
$c=a/b$, $q=a^\prime/b^\prime$, where $a,a^\prime\in\Z$,
$b,b^\prime\in\N$ and $\gcd(a,b)=\gcd(a^\prime,b^\prime)=\gcd(b,p)=\gcd(b^\prime,p)=1$; then
the limit plot $\mathbf{LP}(\mathfrak A)\subset\mathbb T^2$ is a link of $\mlt_mp$ 
torus knots,  where $m=b^\prime/\gcd(b,
b^\prime)$, and every knot of the link is 
a cable $\mathbf C(c,e)$ for $e\in\mathbf C(q)$:
\begin{equation}
\label{eq:thm:mult-add-p}
\mathbf{LP}(\mathfrak A)
=
\left\{\left(y\md 1;\left(cy+e
\right)\md1\right)\colon y\in\R, e\in\mathbf C(q)\right\}.
\end{equation}
Moreover, $\mathbf
C(c,e_1)=\mathbf C(c,e_2)$ for 
$e_1,e_2\in\mathbf C(q)$ 
if and only if $r_1\equiv r_2\pmod m$ where $e_i=(-p^{r_i}q)\md
1$, $i=1,2$, cf. \eqref{eq:const}.  
\begin{note}
\label{note:mult-add-p}
Once $m=1$, i.e., once $b^\prime\mid b$, the congruence $r_1\equiv r_2\pmod
m$ holds trivially, $\mlt_1p=1$ and the link consists of a single knot; so
in that case $\mathbf C(c,e_1)=\mathbf C(c,e_2)$ for all $e_1,e_2\in\mathbf C(q)$.
\end{note}
\end{thm}
\begin{proof}[Proof of Theorem \ref{thm:mult-add-p}]
The first statement of the theorem is already proved, see Proposition \ref{prop:fin-auto}.

Given $q,c\in\Z_p\cap\Q$, we have that 
\begin{align}
\label{eq:c-expr}
c&=u+\frac{v}{p^t-1},\\ 
\label{eq:b-expr}
q&=\frac{w}{p^T-1} 
\end{align}
for suitable $u\in\Z$, $t,T\in\N$, $v\in\{0,1,2,\ldots,p^t-2\}$,
$w\in\{0,1\ldots,p^T-2\}$,
by Proposition \ref{prop:p-repr}. Note that we may assume that $0<q<1$ since  the set of  all limit points  of the
sequence $((z+q)\md p^k/p^k)_{k=1}^\infty$ is the same as that of the sequence $((z+q\md1)\md p^k/p^k)_{k=1}^\infty$ by Proposition \ref{prop:add-p-md1}
and the case $q=0$ is already considered, cf. Proposition \ref{prop:mult-p}.

Now we will prove that $\mathbf{LP}(\mathfrak A)\supset \mathbf C(c,e)$
for $e\in\mathbf C(q)$.  As $q\in\Z_p\cap\Q\cap(0,1)$, the canonical $p$-adic representation of $q$ is
eventually periodic and the period length of $q$ is $T$, cf. Subsection \ref{ssec:p-adic}. Now fix $e\in\mathbf C(q)$, take corresponding $d\in\{0,1,\ldots,T-1\}$ and consider a
sequence $n_j=d+iT\in\N$ $(j=1,2,\ldots)$; then $\lim_{j\to\infty}q\md p^{n_j}/p^{n_j}=e$, cf. the proof of Lemma
\ref{lem:add-p}. Given $x\in[0,1)$ take $z\in\Z_p$ and a sequence 
$\Cal K=(k_i=\ell+r_it)_{i=1}^\infty$ as in the proof
of Proposition \ref{prop:mult-p}; so $x=\lim_{i\to\infty}z\md p^{k_i}/p^{k_i}$.
Note that if $\bar z=p^mz$ for some $m\in\N_0$ then $x=\lim_{i\to\infty}\bar
z\md p^{k_i+m}/p^{k_i+m}$; so  the proof of Proposition \ref{prop:mult-p}
remains valid if one substitutes
$\bar z$  for $z$ and any strictly increasing subsequence $(\check k_i)$ of the sequence
$\bar{\Cal K}=(\bar
k_i=m+\ell+r_it)$ for the sequence $\Cal K$.  

We claim that for some $m\in\N_0$ there exist an increasing sequence $j_s\in\N$ and a subsequence $(\bar r_s)_{s=1}^\infty$ of the sequence
$(r_i)_{i=1}^\infty$ such that 
\begin{equation}
\label{eq:T-t}
m+\ell+\bar r_st=d+j_sT\  \text{for all}\ s=1,2,3,\ldots.
\end{equation}
Indeed, let $D=\gcd(T,t)$ be the greatest common divisor of $T$ and $t$; then $T=\check
TD$, $t=\check tD$, $\check t$ and $\check T$ are co-prime.  As the infinite sequence $(r_i)_{i=1}^\infty$
is strictly increasing, there exists  $\check n\in\{0,1,\ldots,\check
t-1\}$ such that $r_i+\check n\equiv 0\pmod{\check T}$ for infinitely many $i\in\N$,
say, for $i=i_1,i_2,\ldots$. Put $\bar r_s=r_{i_s}$; $s=1,2,3,\ldots$.
 
Take the smallest $\bar n=\check n+n\check T$, $n\in\N_0$, such that $d-\ell+\bar n\check tD\ge 0$, then put $m=d-\ell+\bar n\check tD$
and find $j_s$ from the equation \eqref{eq:T-t} which now is equivalent to
the equation $(\bar n+\bar r_s)\check t=j_s\check T$: As $\bar n+\bar r_s=h_s\check
T$ for a suitable $s\in\N$ by the definition of $\bar n$, one sees that 
$j_s=\check th_s$ for  $s=1,2,3,\ldots$ thus proving our claim.

We conclude now that given arbitrary $y\in\R$ 
 and $e\in\mathbf
C(q)$ there exist $\bar z\in\Z_p$ and a sequence $\check{\Cal K}=(\check k_s=\bar r_st+\ell+m=d+j_sT)$
such that
\begin{align}
\label{eq:lim-1.1}
y\md1=x&=\lim_{s\to\infty}\frac{\bar z\md p^{\check k_s}}{p^{\check k_s}},\\
\label{eq:lim-1.2}
(cy)\md1&=\lim_{s\to\infty}\frac{(c\bar z)\md p^{\check k_s}}{p^{\check k_s}},\\
\label{eq:lim-1.3}
e&=\lim_{s\to\infty}\frac{q\md p^{\check k_s}}{p^{\check k_s}}; 
\end{align}
cf. \eqref{eq:au_cable}, \eqref{eq:cable} and Proposition \ref{prop:const}.
Therefore, $\lim_{s\to\infty}(c\bar z+q)\md p^{\check k_s}/p^{\check k_s}=
\lim_{s\to\infty}((c\bar z)\md p^{\check k_s}/p^{\check k_s}+q\md p^{\check k_s}/p^{\check k_s})\md1=(cy+e)\md1$ and so the point $(y\md1,(cy+e)\md1)\in\mathbf C(c,e)$
is in $\mathbf{LP}(\mathfrak A)$. Thus we have
proved
that $\mathbf{LP}(\mathfrak A)\supset\mathbf C(c,e)$ for every $e\in\mathbf
C(q)$.

On the other hand, given arbitrary $z\in\Z_p$ and arbitrary  strictly increasing
sequence $k_1,k_2,\ldots\in\N$, limit points of the point sequence $(z\md p^{k_i}/p^{k_i};
(cz)\md
p^{k_i}/p^{k_i})$ are all in $\mathbf C(c,0)$ by Proposition \ref{prop:mult-p}
whereas limit points of the sequence $q\md p^{k_i}/p^{k_i}$ are all
in $\mathbf C(q)$ by Proposition \ref{prop:const}. Therefore limit
points of the point sequence $((z\md p^{k_i}/p^{k_i};
(cz+q)\md
p^{k_i}/p^{k_i}))_{i=1}^\infty=((z\md p^{k_i}/p^{k_i},
((cz)\md
p^{k_i}/p^{k_i}+q\md p^{k_i}/p^{k_i})\md1))_{i=1}^\infty$ are all in $\cup_{e\in\mathbf
C(q)}\mathbf C(c,e)$. Finally we conclude that $\mathbf{LP}(\mathfrak A)=\cup_{e\in\mathbf
C(q)}\mathbf C(c,e)$; or (which is the same) that
\begin{equation}
\label{eq:mult-add-p}
\mathbf{LP}(\mathfrak A)=
\left\{\left(y\md 1;\left(cy+e
\right)\md1\right)\colon y\in\R, e\in\mathbf C(q)\right\}
\end{equation}
Note that it may happen that $\mathbf C(c,q)=\mathbf C(c,q_1)$ even if $q\ne q_1$ (and even $q\notin
\mathbf C(q_1)$): For instance, \eqref{eq:au_cable} shows that $\mathbf C(c,q)=\mathbf C(c,0)$ for some $q\ne 0$. Therefore to finish the proof we must now calculate the number of
pairwise distinct cables $\mathbf C(c,e)$ when $e\in\mathbf C(q)$.

During the proof of Proposition \ref{prop:mult-p} we have shown that (in
the notation of the proposition under the proof)
$$
\left\{\left(y\md 1;\left(cy
\right)\md1\right)\colon y\in\R\right\}=
\left\{\left(y\md 1;\left(cy+\frac{j}{b}
\right)\md1\right)\colon y\in\R\right\}
$$
for every $j\in\Z$, cf. equation \eqref{eq:cable} and the text which follows it. Therefore $\mathbf C(c,e_1)=\mathbf C(c,e_2)$  if $e_1-e_2\equiv
(j/b)\md 1$ for some $j\in\Z$. The converse statement is also true:
if $\mathbf C(c,e_1)=\mathbf C(c,e_2)$ then $e_1-e_2\equiv
(j/b)\md 1$ for some $j\in\Z$. 

To prove this, for $h\in\mathbf C(q)$ let $A(c,h )$ be
a set of all points where the cable $\mathbf C(c,h)$ crosses zero meridian of
the torus $\T^2$; that is, 
$$
A(c,h)=\mathbf{AP}\left(\left\{\left(0;\left(\frac{(cz)\md
p^{s_r}}{p^{s_r}}+h\right)\md1\right)\:z\in\Z_p, \lim_{r\to\infty}\frac{z\md
p^{s_r}}{p^{s_r}}=0\right\}\right),
$$
where  
$s_1, s_2,\ldots\in\N$, $s_1<s_2<\ldots$; 
therefore by \eqref{eq:lim} 
\begin{multline}
\label{eq:acc-point}
\mathbf{AP}\left(\left(\left(\frac{(cz)\md
p^{s_r}}{p^{s_r}}+h\right)\md1\right)_{r=0}^\infty\right)=\\
\left\{\left(\frac{j}{b}+h\right)\md 1\: j=0,1,2,\ldots
\right\}.
\end{multline}
%

Finally, as $\mathbf C(c,e_1)=\mathbf C(c,e_2)$ if and only if $A(c,e_1)=A(c,e_2)$
since the both cables cross zero meridian at a same angle (which is equal
to $\arctan c$), this means that $\mathbf C(c,e_1)=\mathbf C(c,e_2)$ if and only if
$e_1- e_2\equiv jb^{-1}\pmod 1$ for some $j\in\N_0$, as claimed.

Now we are able to calculate the number of torus knots (cables) which constitutes
the link $\mathbf{LP}(\mathfrak A)$.  Let for some $j_1,j_2\in\{0,1,\ldots,b-1\}$, ($j_1\ne j_2$) and $e_1,e_2\in\mathbf
C(q)$ the following equality holds:
\begin{equation}
\label{eq:congr-md1}
\left(\frac{j_1}{b}+e_1\right)\md 1=\left(\frac{j_2}{b}+ e_2\right)\md 1.
\end{equation}
We see that $e_i=-p^{r_i}(\frac{a^\prime}{b^\prime})\md 1$ for suitable $r_i\in\{0,1,\ldots,(\mlt_{b^\prime}p)-1\}$   by Note \ref{note:const}
($i=1,2$). Therefore \eqref{eq:congr-md1}
is equivalent to the 
congruence
$$
p^{r_1}\frac{a^\prime}{b^\prime}-p^{r_2}\frac{a^\prime}{b^\prime}\equiv
\frac{j}{b}\pmod 1\ 
$$
for a suitable $j\in\{0,1,\ldots, b-1\}$; but the latter congruence
in turn is equivalent to the congruence
\begin{equation}
\label{eq:congr}
p^{r_2}\left(p^{r_1-r_2}-1\right)a^\prime n\equiv jm\pmod{nmd},
\end{equation}
where $d=\gcd(b^\prime,b)$, $m=b^\prime/d$, $n=b/d$ (we assume
that $r_1>r_2$ since the case $r_1=r_2$ is trivial). From here it follows that
$p^{r_2}\left(p^{r_1-r_2}-1\right)a^\prime n\equiv 0\pmod{m}$ once $m\ne
1$; therefore necessarily $r_1\equiv r_2\pmod{\mlt_mp}$ since $\gcd(b^\prime,b)=\gcd(p,b)=\gcd(p,b^\prime)=1$. So $\left(p^{r_1-r_2}-1\right)=mh$ for a suitable
$h\in\N$ and thus \eqref{eq:congr} is equivalent to the congruence $p^{r_2}h
a^\prime n\equiv j\pmod{nd}$, and the latter congruence gives the value of $j$ (modulo $b=nd$) so that \eqref{eq:congr-md1} is satisfied.
This means that when $m\ne 1$, \ref{eq:congr-md1}
holds if and only if $r_1\equiv r_2\pmod{\mlt_mp}$
Thus, if $m\ne 1$ (that is, if $b^\prime$ is not
a factor of $b$) then the number of pairwise distinct torus knots
in the link is $\mlt_m p$.

In the remaining case when $m=1$ (i.e., when $b^\prime$ divides $b$) \eqref{eq:congr}
always
holds: If $p^{r_1-r_2}\equiv 1\pmod{d}$ then we can take $j=0$ to satisfy
\eqref{eq:congr}; otherwise  the left-hand side of \eqref{eq:congr}  just gives an expression for a unique residue $j$ modulo $b=nd$ (which thus satisfies \eqref{eq:congr}). Therefore the link consist of a unique cable; so the number of pairwise distinct cables in the link is $1=\mlt_1p$
in this case as well.
This concludes the proof. 
\end{proof}
\begin{note}
\label{note:mult-add-p-1}
In conditions of Theorem \ref{thm:mult-add-p} note that $b^\prime\vert b$ is the only case when  the link $\mathbf{LP}(\mathfrak A)$
consists  of a single cable . 
Note also that from  the proof of Theorem \ref{thm:mult-add-p} it is clear that if the number $\#\mathbf C(q)$ of points in $\mathbf C(q)$ is 1 then the link necessarily consists of a single
cable.  By note \ref{note:const},  $\#\mathbf C(q)=1$ if and only if the
period
length of $q$ is 1 and therefore
$q\md1=0.(\xi)^\infty\md1$ for some  $\xi\in\{0,1,\ldots,p-1\}$).
\end{note}
\begin{exmp}
\label{exmp:mult-add-p}
Let $p=2$ and $f(z)=(3/5)\cdot z+(1/3)$. Then in conditions of Theorem \ref{thm:mult-add-p}
we have that $m=3$ and therefore the link consists of $\mlt_32=2$ cables
with slopes $3/5$, cf. Figures \ref{fig:link-square} and  \ref{fig:link-torus}.
\end{exmp} 
\begin{figure}[ht]
\begin{minipage}[b]{.45\linewidth}
\begin{quote}\psset{unit=0.3cm}
 \begin{pspicture}(0,0)(15,15)
 \newrgbcolor{emerald}{.31 .78 .47}
 \newrgbcolor{el-blue}{.49 .98 1}
 \newrgbcolor{baby-blue}{.54 .81 .94}
   \psframe[
  linewidth=0.5pt,linecolor=green](-4.8,0)(13.2,18)
  \pscircle[fillstyle=solid,fillcolor=yellow,linewidth=.4pt](-4.8,0.1){.3}
  \psline[unit=0.359cm,linecolor=baby-blue,linewidth=2pt](-4,14)(-2.4,15)
  \psline[unit=0.359cm,linecolor=baby-blue,linewidth=2pt](-4,11)(2.6,15)
  \psline[unit=0.359cm,linecolor=emerald,linewidth=2pt](-4,13)(-0.7,15)
  \psline[unit=0.359cm,linecolor=baby-blue,linewidth=2pt](-4,5)(11,14.2)
  \psline[unit=0.359cm,linecolor=emerald,linewidth=2pt](-4,4)(11,13.2)
  \psline[unit=0.359cm,linecolor=emerald,linewidth=2pt](-4,7)(9,15)
  \psline[unit=0.359cm,linecolor=baby-blue,linewidth=2pt](-4,8)(7.5,15)
  \psline[unit=0.359cm,linecolor=emerald,linewidth=2pt](-4,10)(4.2,15)
  \psline[unit=0.359cm,linecolor=baby-blue,linewidth=2pt](-4,2)(11,11.2)
  \psline[unit=0.359cm,linecolor=emerald,linewidth=2pt](-4,1)(11,10.2)
  \psline[unit=0.359cm,linecolor=baby-blue,linewidth=2pt](2.5,0)(11,5.2)
  \psline[unit=0.359cm,linecolor=baby-blue,linewidth=2pt](-2.4,0)(11,8.2)
  \psline[unit=0.359cm,linecolor=emerald,linewidth=2pt](9,0)(11,1.2)
  \psline[unit=0.359cm,linecolor=emerald,linewidth=2pt](4.1,0)(11,4.2)
  \psline[unit=0.359cm,linecolor=emerald,linewidth=2pt](-0.75,0)(11,7.2)
  \psline[unit=0.359cm,linecolor=baby-blue,linewidth=2pt](7.4,0)(11,2.2)
 \end{pspicture}
\end{quote}
\caption{%
{Limit plot of the function $f(z)= \frac{3}{5}z+\frac{1}{3}$, $z\in\Z_2$, in $\mathbb R^2$}}
\label{fig:link-square}
\end{minipage}\hfill
\begin{minipage}[b]{.45\linewidth}
\includegraphics[width=0.92\textwidth,natwidth=610,natheight=642]{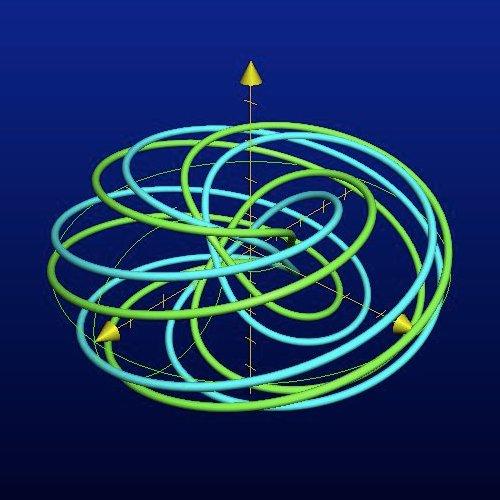}
\caption{Limit plot of the same function on the torus $\mathbb T^2$\qquad \qquad}
\label{fig:link-torus}
\end{minipage}
\end{figure}
\begin{cor}
\label{cor:mult-add-compl}
There is a one-to-one correspondence between maps of the
form $f\:z\mapsto\frac{a}{b}z+\frac{a^\prime}{b^\prime}$ on $\Z_p$ \textup(where
$\frac{a}{b},\frac{a^\prime}{b^\prime}\in\Z_p\cap\Q$; $a,a^\prime\in\Z$; $b,b^\prime\in\N$\textup) and collections of $\mlt_{m}p$
complex-valued exponential
functions $\psi_k\:\R\>\C$ of real variable $y\in\R$ 
$$
\left\{\psi_k(y)=e^{i(\frac{a}{b}y-2\pi p^k\frac{a^\prime}{b^\prime})}\: 
k=0,1,2,\ldots,(\mlt_{m}p)-1\right\}.
$$
Here $i\in\C$ is imaginary unit and $m=b^\prime/\gcd(b,b^\prime)$.
\end{cor}
\begin{proof}[Proof of Corollary \ref{cor:mult-add-compl}]
Indeed, embedding the unit torus $\T^2$ into a 3-dimensional Euclidean space $\R^3$ and using cylindrical coordinates  as in Note
\ref{note:cable}, in view of Theorem \ref{thm:mult-add-p} every knot from the link can be expressed 
in the form \eqref{eq:note_cable} with $\omega=2\pi e$ for $e\in\mathbf
C(q)$ since $\cos\omega$ and $\sin\omega$ specifies  position of the point
where the knot crosses zero meridian of the torus (i.e., when $\theta\equiv0\pmod
{2\pi}$ in \eqref{eq:note_cable}). But $q=a^\prime/b^\prime$ and thus $\mathbf
C(q)=\left\{(-p^{\ell}\cdot({a^\prime}/{b^\prime}))\md1\colon \ell=0,1,\ldots,(\mlt_{b^\prime}p)-1\right\}$
by \eqref{eq:C}. As two such knots (with accordingly $\omega_i=2\pi e_i$, $i=1,2$) coincide if an only if $\omega_1\equiv\omega_2\pmod{2\pi\cdot(a/b)}$ by \eqref{eq:note_cable},
i.e., if and only if $e_1\equiv e_2\pmod{a/b}$. But the latter congruence
is equivalent to \eqref{eq:congr-md1}; so finally the assertion follows from
Theorem \ref{thm:mult-add-p}.
\end{proof}



%
\section{Finite computability}
\label{sec:fin-comp}
In this section we introduce central notion of the paper, the \emph{finite computability}, and prove some technical results which will be needed further
during the proof of main result of the paper, the affinity of finitely
computable smooth functions, cf. further Section \ref{sec:main}.
\begin{defn}
\label{def:eval-s}
A non-empty point set $S\subset\mathbb I^2$ ($S\subset\mathbb T^2$, $S\subset\mathbb I\times \mathbb S$, $S\subset\mathbb S\times\mathbb I$) is  called  \emph{\textup(ultimately\textup)
finitely computable} (or, (ultimately) computable by a finite automaton) if
there exists a finite automaton $\mathfrak A$ such that $S$ is a subset of
$\mathbf
{P}(\mathfrak A)$
(of $\mathbf
{LP}(\mathfrak A)$). 
We say that the automaton $\mathfrak A$  \emph{\textup(ultimately\textup)
computes}
the set $S$; and $\mathfrak A$  is called
an \emph{\textup(ultimate\textup) computing automaton} of the set $S$.
\end{defn}
In most further cases given a real function $g\colon D\>\R$ with the domain $D\subset \R$
by the
\emph{graph 
of the function} (on the torus $\mathbb T^2$) we mean
the point subset $\mathbf G_D(g)=\{(x\md1;g(x)\md1)\colon x\in D\}\subset\mathbb
T^2$.
However, given  a function $g\:D\>T$ where either
$D\subset[0,1]$ or
$D\subset\mathbb S$  and $T$ is either $[0,1]$ or  $\mathbb
S$, we call a graph $\mathbf G_D$ of the function $g$ the set 
$\{(\bar x;\overline{g(x)})\:x\in D\}$ where either $\bar x=x$ if $D\subset[0,1]$
or $\bar x=x\md1$ if $D\subset\mathbb S$ and accordingly either $\overline{g(x)}=g(x)$
if $T=[0,1]$ or $\overline{g(x)}=(g(x))\md1$ if $T=\mathbb S$. In the sequel
we  always explain what is meant by $\mathbf G_D(g)$ if this is not clear from
the context. Also, we may omit the subscript $D$ when it is clear what is the domain.

\begin{defn}
\label{def:eval}
%

Given a
real function $g\:D\>\R$ with domain $D\subset\R$ and an automaton $\mathfrak
A$, the function $g$ is called
\emph{\textup(ultimately\textup) computable by $\mathfrak A$ at the point $x\in D$} if 
$(x\md1;g(x)\md1)\in\mathbf{P}(\mathfrak A)\subset\mathbb T^2$ ($(x\md1;g(x)\md1)\in\mathbf{LP}(\mathfrak A)\subset\mathbb T^2$). 
Also, if either
$D\subset[0,1]$ or
$D\subset\mathbb S$ and $g\:D\>T$ where  either $T=[0,1]$ or  $T=\mathbb
S$ 
we will say that $\mathfrak A$ \emph{\textup(ultimately\textup) computes $g$ at the point $x\in D$}
if $(\bar x;\overline{g(x)})\in\mathbf{LP}(\mathfrak A)$ where either $\bar x=x$ if $D\subset[0,1]$
or $\bar x=x\md1$ if $D\subset\mathbb S$ and accordingly either $\overline{g(x)}=g(x)$
if $T=[0,1]$ or $\overline{g(x)}=(g(x))\md1$ if $T=\mathbb S$ (cf. Note \ref{note:plot-auto}) 

Given a
real function $g\:D\>\R$ with domain $D\subset\R$, the function $g$ is called
\emph{\textup(ultimately\textup) finitely computable} (or, \emph{\textup(ultimately\textup)  computable by
a finite automaton}) if there exists a finite automaton $\mathfrak A$ such
that $\mathbf G(g)\subset\mathbf{P}(\mathfrak A)\subset\mathbb T^2$ ($\mathbf G(g)\subset\mathbf{LP}(\mathfrak A)\subset\mathbb T^2$).
The automaton $\mathfrak A$ which \emph{\textup(ultimately\textup) computes}
the function $g$ is called
the \emph{\textup(ultimate\textup) computing automaton} of the function $g$. In a similar manner we
define these notions for the cases when $g\:D\>T$ and $D$, $T$ are as above.
\end{defn}

\subsection{The mark-ups} 
\label{ssec:mark-up}
In loose terms, when assigning a real-valued function $f^\mathfrak A\: [0,1]\>[0,1]$  to automaton $\mathfrak
A$ via Monna map  $\mon:\Z_p\to\mathbb R$ (cf. subsection \ref{ssec:plots}) one feeds the automaton  by  a base-$p$-expansion
of argument $x\in[0,1]$  and considers the output as a base-$p$ expansion of $f^\mathfrak A(x)$:
A base-$p$ expansion specifies a unique right-infinite word in the alphabet
$\F_p$ and the automaton `reads the word from head to tail', i.e., 
is feeded by  digits of the base-$p$ expansion from left to right (i.e., digits
on more
significant positions are feeded prior to digits on less significant positions);
and the output word specifies a base-$p$ expansion of a unique real
number from $[0,1]$.  

To examine functions computed
by automata in the meaning of Definition \ref{def:eval} it would also be convenient to work with base-$p$ expansions of real numbers; 
but the problem is that we need feed the automaton 
by
a right-infinite word in the inverse order `from tail (which is at infinity)
to head': Digits
on less 
significant positions (the rightmost ones) should be feeded prior to digits on more significant positions
(the leftmost ones). So straightforward inversion is impossible since it is unclear which letter should be the first when feeding
the automaton this way; thus  output word  is undefined and so is the real number whose base-$p$ expansion is the output word. In this subsection we  rigorously
specify this inversion and develop some techniques needed  in further proofs.

Let  a function $g\:D\>\mathbb S$ (or $g\:D\>[0,1]$) whose domain $D$ is  either
a subset of a real
unit
circle $\mathbb S$ or a subset of a unit segment $[0,1]$ be ultimately computable by a finite automaton $\mathfrak A=\mathfrak A(s_0)$; that
is, for any $x\in D$ there exists $x\in\Z_p$ such that $x$ is a limit
point of the sequence $(z\md p^k/p^k)_{k=1}^\infty$ and $g(x)$ is a limit
point of the sequence $((f_{\mathfrak A}(z))\md p^k/p^k)_{k=1}^\infty$, where
$f_{\mathfrak A}\:\Z_p\>\Z_p$ is automaton function of the automaton $\mathfrak
A$, cf.  Definition \ref{def:eval} and Definition \ref{def:plot-auto}.
As said,  further  to examine  finitely computable real functions  it is however more convenient to work with automata
maps as maps  of reals into reals
rather than to consider automata functions on $p$-adic
integers and then represent $x\in\R$ and $g(x)\in\R$ as limit
points of the sequences $(z\md p^k/p^k)_{k=1}^\infty$ and $((f_{\mathfrak A}(z))\md p^k/p^k)_{k=1}^\infty$, respectively. 

Further in this subsection we are
going to show that  once $x\in D$ and
once $x=0.\chi_1\chi_2\ldots$ is a base-$p$ expansion of $x$, we can find
a state $s=s(x)\in\Cal S$ of the automaton $\mathfrak A$ and  a strictly increasing infinite sequence of indices $1\le k_1<k_2<\ldots$ such that the
sequence $(0.\mathfrak a_s(\chi_1\chi_2\ldots\chi_{k_j}))_{j=1}^\infty$ tends
to $(g(x))\md1$ (recall that $\mathfrak a_s(\zeta_1\zeta_2\ldots\zeta_\ell)$
is an $\ell$-letter output word of the automaton $\mathfrak A(s)$ whose initial
state is $s$ once the automaton has been feeded
by the $\ell$-letter input word $\zeta_1\zeta_2\ldots\zeta_\ell$, cf. Subsection
\ref{ssec:auto}). This means, loosely speaking, that once we feed the automaton
$\mathfrak A(s)$ with  approximations  $0.\chi_1\chi_2\ldots\chi_{k_j}$ of
$x$,
the automaton outputs the sequence of approximations 
$0.\mathfrak a_s(\chi_1\chi_2\ldots\chi_{k_j})$ of $g(x)$, and 
these
sequences tend to $x$ and to $g(x)$ accordingly while  $j\to\infty$. Moreover,
we will show that if the function $g$ is continuous then there exists a state $s\in\Cal S$ such that all $x\in D$ for which $s(x)=s$ constitute a dense
subset in $D$. 

Recall that given $x\in(0,1)$, there exists a (right-)infinite word $w=\gamma_0\gamma_1\ldots$
over $\{0,1,\ldots,p-1\}$ such that 
\begin{equation}
\label{eq:base-p}
x=0.\gamma_0\gamma_1\ldots=0.w
=\sum_{i=0}^\infty\gamma_ip^{-i-1},
\end{equation}
the \emph{base-$p$ expansion} of $x$. If $x$  is not of the form $x=n/p^k$
for some $n=\alpha_0+\alpha_1p+\cdots+\alpha_{\ell} p^{\ell}\in\{0,1,\ldots,p^k-1\}$,
where $\ell=\mathbf{le}(n)=\lfloor\log_pn\rfloor+1$ is the length of the
base-$p$
expansion of $n\in\N_0$ (recall that we put $\lfloor\log_p0\rfloor=0$, cf.
Subsection \ref{ssec:a-map}),
$\alpha_0,\alpha_1,\ldots\alpha_\ell\in\{0,1,\ldots,p-1\}$, then the right-infinite
word 
$\wrd(x)=\gamma_0\gamma_1\ldots$ over $\{0,1,\ldots,p-1\}$
is uniquely defined (and the corresponding $x$ is said to have a \emph{unique}
base-$p$ expansion); else there are exactly two infinite words,
\begin{align}
\label{eq:base-p_r}
\wrd^r(x)&=\alpha_0\alpha_1\ldots\alpha_{\ell-1}\alpha_{\ell}00\ldots=\alpha_0\alpha_1\ldots\alpha_{\ell-1}\alpha_{\ell}(0)^\infty\\
\label{eq:base-p_l}
\wrd^l(x)&=\alpha_0\alpha_1\ldots\alpha_{\ell-1}(\alpha_{\ell}-1)(p-1)(p-1)\ldots=\alpha_0\alpha_1\ldots\alpha_{\ell-1}(\alpha_{\ell}-1)(p-1)^\infty,
\end{align}
where $\alpha_\ell\ne 0$, such that $x=0.\wrd_r(x)=0.\wrd_l(x)$. In that case  $x$ is said to have a \emph{non-unique}
base-$p$ expansion; the corresponding base-$p$ expansions are called \emph{right}
and \emph{left} respectively.
Both 0 and 1 are assumed to have  unique base-$p$ expansions since $0=0.00\ldots$, $1=0.(p-1)(p-1)\ldots$; so $
\wrd(0)=00\ldots$,
$
\wrd(1)=(p-1)(p-1)\ldots$. This way we define $
\wrd(x)$ for all $x\in[0,1]$; and to $x=n/p^k$ we will usually put into the correspondence both infinite
words $\wrd^l(x)$ and $\wrd^r(x)$ if converse is not stated explicitly.
The only difference in considering a unit circle $\mathbb S$ rather than
the unit segment $\mathbb I=[0,1]$  is that we identify 0 and
1 and thus have two representations for 0, $0.(0)^\infty$ and $0=1\md1=0.(p-1)^\infty$.

Given a finite word  $w=\alpha_{m-1}\alpha_{m-2}\cdots\alpha _0$, we denote
via  $\overrightarrow{w}$ the (right-)infinite word $
\overrightarrow{w}= \alpha_{m-1}\alpha_{m-2}\cdots\alpha _0(0)^\infty$ and we put
$0.\overrightarrow{w}=0.\alpha_{m-1}\alpha_{m-2}\cdots\alpha _0(0)^\infty\ldots$
(note that  then $0.\overrightarrow{w}=\rho(w)$). Of course,  
$0.\overrightarrow{w}=0.w=\sum_{i=0}^{m-1}\alpha_ip^{-m+i}$; but we use notation
$0.\overrightarrow{w}$ if we want to stress that we deal with infinite base-$p$
expansion. To unify our notation, we
also may write $\overrightarrow{w}=\zeta_1\zeta_2\ldots$ for a (right-)infinite
word $w=\zeta_1\zeta_2\ldots$; then $0.\overrightarrow{w}=0.w=0.\zeta_1\zeta_2\ldots$.

Let $\overrightarrow{w}=\gamma_0\gamma_1\ldots$ be a (right-)infinite word
over $\F_p=\{0,1,\ldots,p-1\}$. 
Given
an automaton $\mathfrak A$ with the initial state $s$, we further denote via $\mathfrak a_s(\overrightarrow{w})$
the set of all limit points of the sequence $(\rho(\mathfrak a_s(\gamma_0\gamma_1\ldots\gamma_{k})))_{k=0}^\infty$.
We may omit the subscript $s$ if it is clear from the context what is the
initial state of the automaton. 

Given $x\in[0,1]$  further $\mathfrak a_s(x)$ stands for 
$\mathfrak a_s(\overrightarrow{w}(x))$ if $x$ admits a unique base-$p$
expansion, 
and  $\mathfrak a_s(x)=\mathfrak a_s(\overrightarrow{w}(x)_l)\cup\mathfrak a_s(\overrightarrow{w}(x)_r)$
if the expansion is non-unique (thus, if $x$ admits both left and right base-$p$ expansions). We also consider $\mathfrak a_s(x)$ for  $x\in\mathbb S$ rather
that  for $x\in[0,1]$; in
that case we take for $0$ both base-$p$ expansions  $0.(0)^\infty$ and
$0.(p-1)^\infty$ (since $0=(0.(p-1)^\infty)\md1$) and reduce modulo 1 all limit
points of all sequences $(\rho(\mathfrak a_s(\gamma_0\gamma_1\ldots\gamma_{k})))_{k=0}^\infty$.
We further use the same symbol $\mathfrak a_s(x)$ independently of whether we consider
$x\in[0,1]$ or $x\in\mathbb S$; we make special remarks when this may cause
a confusion.

We stress that   $\mathfrak a_s(w)$ is
a uniquely defined finite word whenever $w\in\Cal W$ is a finite word (and
therefore $\rho(\mathfrak
a_s(w))$ consists of a single number), but in the case when $w$ is an infinite
word or $w$ is a real number from $[0,1]$ (or $w\in\mathbb S$), the set $\mathfrak a_s(w)$ may
contain more than one element.

Given $x\in\Q\cap[0,1]$, in view of Lemma \ref{le:per} it is clear that if
the automaton $\mathfrak A$ is finite then $\mathfrak
a(x)\in Q\cap[0,1]$ since a real number is rational if and only if its base-$p$ expansion  is eventually periodic. The following propositions reveals some
more details about $\mathfrak a(x)$ for a rational $x$; and especially for $x=0$.
\begin{prop}
\label{prop:auto-Q}
If $\mathfrak A$ is finite, $x\in\Q\cap[0,1]$ then $\mathfrak a(x)\subset\Q\cap[0,1]$
and $\mathfrak a(x)$ is a finite set.
Moreover, if $x\in\Z_p\cap\Q\cap[0,1]$ 
then
$\mathfrak a(x)\subset\Z_p\cap\Q\cap[0,1]$.  In particular, if $x=0\in\mathbb
S$ then $\mathfrak a(x)=\mathbf C(q_1)\cup\mathbf C(q_2)$ for suitable $q_1,q_2\in\Z_p\cap\Q\cap[0,1)$
\textup{(cf. Subsection \ref{ssec:const})}. Let a $\mathfrak A$-computable
function $g\:D\>\mathbb S$ be defined on the  domain $D\subset\mathbb S$ and  continuous at $0\in D$.  If the domain $D$ is open then
there exists $q\in\Z_p\cap\Q\cap[0,1)$ such taht  $\mathfrak a(0.(0)^\infty)=\mathfrak a(0.(p-1)^\infty)\in\mathbf C(q)$; and    either $\mathfrak a(0.(0)^\infty)\in\mathbf C(q)$ or $\mathfrak a(0.(p-1)^\infty)\in\mathbf C(q)$ if the domain $D$ is
half-open and $x$ is a boundary of $D$.  
\end{prop}
\begin{proof}[Proof of Proposition \ref{prop:auto-Q}]
Follows from Lemma \ref{le:per} and Proposition \ref{prop:const} (see the
proof of the latter). 

Given $x\in\Q\cap[0,1]$,  a base-$p$ expansion of $x$ is eventually periodic,
cf. \eqref{eq:r-repr}:
$x=0.\chi_0\ldots\chi_{k-1}(\xi_0\ldots\xi_{n-1})^\infty$.
Given $x$, take $n,k$ the smallest possible (note that then the word  $v=\chi_0\ldots\chi_{k-1}$
may be empty). From the definition of $\mathfrak a(x)$ it follows that $\mathfrak a(x)$ consists of all limit points of the sequences $\Cal K(r)=(\rho(\mathfrak(v(u)^kr))_{k=K}^\infty$,
where $K$ is large enough and $r\in\{\xi_{0},\xi_{0}\xi_{1},\ldots,\xi_0\xi_1\ldots\xi_{n-1}\}$
are  all suffixes of the word $u=\xi_0\ldots\xi_{n-1}$, for every right-infinite
word $w$ which corresponds to a base-$p$ expansion of $x=0.w$. As the automaton $\mathfrak A=\mathfrak A(s_0)$ is finite, the number
of states it reaches after being feeded by   either of words $r$, where 
is finite
; say, these states are $s_1,\ldots,s_N\in\Cal S$. By the same reason, being feeded by the words $(u)^K$ where $K$ is large enough, the automata $\mathfrak A_1=\mathfrak A(s_1), \ldots,\mathfrak A_N=\mathfrak A(s_N)$ output respectively words $v_1(u_1)^{K_1}t_1, \ldots, v_N(u_N)^{K_N}t_N$,
where the words $u_1,\ldots,u_N, t_1,\ldots,t_N$ do not depend on $K$, each
of the words $v_1,\ldots,v_N$ is
 either empty or a prefix  of the respective word $w_1,\ldots,w_N$, and all the output words  $v_1(u_1)^{K_1}t_1, \ldots, v_N(u_N)^{K_N}t_N$ have the same length as the one of the input word $(u)^K$, cf. Lemma \ref{le:per}.
That is, the output words of automata $\mathfrak A_i$
are all of the form $(\bar u)^Lt$, where $\bar u$ stands for a cyclically shifted word $u$, and after outputting the word $(\bar u)^Lt$, each automaton
$\mathfrak A_i$, $i=1,2,\ldots, N$, reaches some of finitely many its states,
say, $s^\prime_{i1},\ldots,s^\prime_{iM(i)}$. After reaching  respective state,
the
automaton is feeded by the word $v$ and outputs the corresponding output word
$v^\prime_{i,j}$, $j=1,2,\ldots,M(i)$. Therefore, all limit points of  the sequences $\Cal K(r)$ are of the
form $0.
{v^\prime_{i,j}(\bar u_i)^\infty}$ where  $\bar u_i$ runs through a (sub)set of all cyclic shifts
of the word $u_i$, $i=1,2,\ldots,N$,
$j=1,2,\ldots,M(i)$. But there are only finite number of points of that form;
therefore, given a base-$p$ expansion of $x=0.\overrightarrow w$, the set
$\mathfrak a(0.\overrightarrow w)$ is a union of  a finite number points of the said form. As
every $x\in\Q$ has at most two base-$p$ expansions, this proves the first
claim of the proposition.

If $x\in\Z_p\cap\Q\cap[0,1]$, then  base-$p$ expansion of $x$ is purely periodic
by Proposition \ref{prop:r-p-repr-qz}: $x=0.(\chi_0\ldots\chi_{n-1})^\infty$. Therefore once the automaton $\mathfrak A$ is being feeded by finite words of the form $w^Lt$, where $w=\chi_0\ldots\chi_{n-1}$,
$t$ is a suffix of $w$ or empty and $L$ is large enough, by Lemma \ref{le:per}
the corresponding
output word will be of the form $(u_t)^{N(t)}v_t$, and the number of different
words $u_t$ is finite since the number of different words $t$ is finite.
Applying the same argument as above we conclude that all limit points
of corresponding sequences are of the form $0.
{(u)^\infty}$
where $u$ runs through a finite number of finite words.  But by Corollary
\ref{cor:r-repr}, all these points are in $\Z_p\cap\Q\cap[0,1]$. This proves
the second claim of the proposition.

To prove the final claim, we must consider both base-$p$ representations
of zero point of the unit circle $\mathbb S$: $0=0.(0)^\infty$ and $0=0.(p-1)^\infty$.
Sending a left-infinite zero sequence to the automaton $A$, the output sequence
will be of the form $w^\infty t$ by suitable finite words $w,t$ by Lemma
\ref{le:per}; so $\mathfrak a(0^\infty)$ consists of all points of the form
$0.(u)^\infty$, where $u$ runs through all cyclic shifts of the word 
$w=\chi_{n-1}\ldots\chi_0$; therefore $\mathfrak a(0^\infty)=\mathbf C(q_1)$
for a suitable $q_1\in\Z_p\cap\Q\cap[0,1)$, cf. Note \ref{note:const}.  By
the same reason, $\mathfrak a((p-1)^\infty)=\mathbf C(q_2)$
for a suitable $q_2\in\Z_p\cap\Q\cap[0,1)$.  In the case when the $\mathfrak
A$-computable function $g\:D\>\mathbb S$ is continuous at $0$ and there
exists an open neighborhood $U$ of $0$ such that $U\subset D$
then  necessarily
$g(0)\in\mathfrak a(0.0^\infty)$ and $g(0)\in\mathfrak a(0.(p-1)^\infty)$;
so $\mathbf C(q_1)\cap\mathbf C(q_2)\ne\emptyset$ and therefore $\mathbf C(q_1)=\mathbf C(q_2)$ by Corollary \ref{cor:const}. If no such neighborhood
$U$ exists then the domain is half-open and $0$ is a boundary point; thus due
to the continuity of $g$ at $0$ we see that either $g(0)\in\mathfrak a(0.0^\infty)$ or $g(0)\in\mathfrak a(0.(p-1)^\infty)$ and the conclusion follows. 

%
\end{proof}
\begin{cor}
\label{cor:auto-Q}
Let $\mathfrak A$ be a finite automaton, let $(x;y)\in\mathbf P(\mathfrak A)\subset\T^2$, and let $x\in\Z_p\cap
\Q\setminus\{0\}$; then $y\in\Z_p\cap\Q$. If $x=0$ then $y\in[0,1)\cap\Q$;
moreover, there exists $y\in\Z_p\cap\Q$ such that $(0;y)\in\mathbf P(\mathfrak
A)$.
\end{cor}
\begin{proof}[Proof of Corollary \ref{cor:auto-Q}]
As $x\in\Z_p\cap\Q$ and $x\in[0,1)$ then by Corollary \ref{cor:r-p-repr-qz}
the base-$p$ expansion is purely periodic; that is, $x=0.w$, where $w\in\Cal
W^\infty$ is a right-infinite periodic word: $w=(v)^\infty$ for a suitable
finite non-empty word
$v\in\Cal W$. As $(x,y)\in\mathbf P(\mathfrak A)$ then by Note \ref{note:plot-auto-in}
there exists a sequence $(u_i)_{i=0}^\infty$ of finite non-empty words such
that $\lim_{i\to\infty}0.u_i=x$ and simultaneously $\lim_{i\to\infty}0.\mathfrak
a(u_i)=y$. Let $x\ne 0$; then $v$ is not a 1-letter zero word: $v\ne0$. Therefore
since $\lim_{i\to\infty}0.u_i=x$,  for all sufficiently large $i$ the words $u_i$ must be of the form $u_i=(v)^{L_i}\bar u_i$ where $L_i$ increases unboundedly
while $i\to\infty$. Therefore we may assume that the sequence $L_i$ is strictly
increasing (we consider a strictly increasing subsequence of the sequence $(L_i)$ if otherwise) and that all $u_i$ are of the form $u_i=(v)^{L_i}\bar
u_i$. Let $s(i)$ be a state the automaton $\mathfrak A$ reaches after being
feeded by the word $\bar u_i$ ($s(i)=s_0$ if $\bar u_i=\phi$ is empty). As  the automaton $\mathfrak A$ is finite, then there are only finitely many
pairwise distinct $s(i)$, say these are $s^\prime(1),\ldots, s^\prime(n)$. Now we consider automata $\mathfrak A(s^\prime(1)),\ldots, \mathfrak A(s^\prime(n))$ and apply the same argument as in the proof
of the second statement of Proposition \ref{prop:auto-Q} thus proving that
$y\in\Z_p\cap\Q$.  The same argument can be applied for the case when $x=0$
but there exists an infinite sequence of words $u_i$ whose lengths are increasing
unboundedly while $i\to\infty$. Therefore the only rest case is now $x=0$
and once the word sequence $(u_i)_{i=0}^\infty$ is such that $\lim_{i\to\infty}0.u_i=x$
and lengths of $u_i$ are bounded; $\Lambda(u_i)\le K$ for all $i\in\N_0$.
But this just means that for all sufficiently large $i$ all the words $u_i$
are $K$-letter zero words: $u_i=(0)^K$ for a suitable $K\in\N$. But then
$y=0.\mathfrak a((0)^K)$; thus $y\in[0,1)\cap\Q$. The last claim of the corollary
trivially follows from Proposition \ref{prop:auto-Q} since once $y\in\mathfrak
a(0)$ then $(0;y)\in\mathbf P(\mathfrak A)$ by the definition of $\mathbf P(\mathfrak A)$. 
\end{proof}

The following definition introduces an important  technical notion, the mark-up,
 which will be used
in further proofs: 
\begin{defn}
\label{def:mark-up}
Given a function $g\:D
\>[0,1]$ defined on the domain 
$D
\subset[0,1]$, 
an automaton $\mathfrak A$,  and a point $x\in D
$
consider  a right-infinite word $w$
such that $x=0.w$,
cf. \eqref{eq:base-p},
\eqref{eq:base-p_r}, \eqref{eq:base-p_l}. An infinite strictly increasing sequence $i_0,i_1,i_2,\ldots$
over $\N_0$ is called an \emph{$s$-mark-up of the 
\textup{(right-infinite)}  word $w=\gamma_0\gamma_1\ldots$
w.r.t. $g$ and $\mathfrak A$} (or briefly a \emph{mark-up} when it is clear what
$g$, $\mathfrak A$ and $s$ are meant) if there exists a state $s\in\Cal S$ of the automaton $\mathfrak A$ such
that $\lim_{k\to\infty}\rho(\mathfrak a_s(\gamma_0\gamma_1\ldots\gamma_{i_k}))=g(x)$. \end{defn}
\begin{rmk}
Given $x\in D
\subset\mathbb S$, the mark-ups of $x$ are defined
exactly in the same way. Note only that the point $x=0$ of $\mathbb S$ coincides
with the point $x=1$ and $1=0=0.000\ldots=0.(p-1)(p-1)(p-1)\ldots\in\mathbb
S$ as $\mathbb S=\R\md 1$. In a similar way we define the mark-up when $g\:D\>\mathbb
S$.
\end{rmk}
The following proposition shows, speaking loosely, that if a continuous real function
is finitely computable, then all base-$p$ expansions of all its arguments
can be marked-up:
\begin{prop}
\label{prop:mark-up}
Given a continuous function $g\colon(a,b)\>[0,1]$ \textup{(or, which makes no
difference, $g\colon
(a,b)
\>\mathbb S$ where
$
(a,b)
\subset\mathbb S$ is an arc of the unit circle $\mathbb
S$)}
let $\mathbf G(g)\subset\mathbf P(\mathfrak A)$ 
for a suitable finite automaton $\mathfrak
A$. Then for every $x\in(a,b)$ and every infinite word $w$ such
that $x=0.w$ there exists an  $s$-mark-up, for a suitable
state $s=s(w)\in\Cal S$ of the automaton $\mathfrak A$.
\end{prop}
\begin{proof}[Proof of Proposition \ref{prop:mark-up}]
The idea of the proof is as follows: Once feeding a finite automaton $\mathfrak
A$ by infinite
sequence of finite words $\gamma_0, \gamma_0\gamma_1, \gamma_0\gamma_1\gamma_2,
\ldots$ over $\Z_p$, the automaton reaches some of its states infinitely
many times; this state $s$ specifies a mark-up $(i(s))$ since corresponding sequences
of approximations $0.\gamma_0\ldots\gamma_{i(s)}$ and $\mathfrak a_s(0.\gamma_0\ldots\gamma_{i(s)})$
tend accordingly to $x=0.\gamma_0\gamma_1\gamma_2\ldots$ and to $g(x)$ due to the continuity of $g$ at $x$. Now we  prove the proposition rigorously.

Firstly consider the case when $x\in(a,b)$ 
has a unique base-$p$ expansion, say,
$x=0.\gamma_0\gamma_1\ldots=0.w$, where 
$w=\gamma_0\gamma_1\ldots$.
As $g(x)$ is $\mathfrak A$-computable, there exists a sequence $w_0,w_1,w_2,\ldots\in\Cal W$ of finite non-empty words such that 
$\rho(w_0),\rho(w_1),\rho(w_2),\ldots\in(a,b)$,
$\lim_{i\to\infty}\rho(w_i)=x$ and 
$\lim_{j\to\infty}\rho(\mathfrak a(w_i))=g(x)$.
Note that 
the sequence $(\Lambda(w_i))_{i=0}^\infty$
is increasing since otherwise $x=n/p^r$
for suitable $n,r\in\N_0$ as $\lim_{i\to\infty}\rho(w_i)=x$ and thus
$x$
has a non-unique base-$p$ expansion. We may
assume that $(\Lambda(\omega_i))_{i=0}^\infty$ is a strictly increasing sequence
since otherwise we just take a suitable subsequence of $(w_i)_{i=0}^\infty$.
Moreover, by the same reason we may assume that $\Lambda(w_i)>i$.

Consider a word sequence $\bar w_0=\gamma_0, \bar w_1=\gamma_0\gamma_1,
\bar w_2=\gamma_0\gamma_1\gamma_2, \ldots$. As $x\in(a,b)$, there exists $N\in\N_0$ 
such that $0.\bar w_i\in(a,b)$ once $i\ge N$. Without loss
of generality we may assume that $N=1$. 
As $\lim_{i\to\infty}0.w_i=0.\gamma_0\gamma_1\ldots=x$,
for every $n\in\N_0$ there exists $M(n)\in\N_0$, $M(n)\ge n$, such that 
$|0.w_i-0.\bar w_j|<p^{-n}$ provided $i,j\ge M(n)$; therefore 
as $\Lambda(w_i)>i\ge M(n)\ge n$, we conclude that 
 $w_i=\gamma_0\ldots\gamma_n v_i$ 
for $i\ge M(n)$ and suitable
finite  word $v_i$.
Given $n\in\N_0$, let $M(n)$ be the
smallest with the said property; this way we obtain an increasing sequence 
$M(0)<M(1)<M(2)<\ldots$. Considering the  subsequence $(w_{M(j)})_{j=0}^\infty$
of the word sequence $(w_i)_{i=0}^\infty$,
we see that $w_{M(j)}=\gamma_0\ldots\gamma_j r_{j}$ 
for $j=0,1,2,\ldots$ and suitable finite  non-empty words $r_j$, that
$\lim_{j\to\infty}\rho(w_{M(j)})=x$ and that $\lim_{j\to\infty}\rho(\mathfrak
a(w_{M(j)}))=g(x)$. Now to the word sequence $(r_{j})_{j=0}^\infty$
we put into the correspondence the sequence $(s(r_j)))_{j=0}^\infty$ of states of
the automaton $\mathfrak A$, where $s(r_j)$ is the state the automaton
$\mathfrak A$ reaches after being feeded by the input word $r_j$. As
the number of different states of $\mathfrak A$ is finite, in the sequence
$(s(r_j))_{j=0}^\infty$ at least one state, say $s$, occurs in the sequence
$(s(r_j))_{j=0}^\infty$ infinitely
many times; say, for $j=j_0,j_1,j_2,\ldots$ ($j_0<j_1<j_2<\ldots$). Therefore
\begin{align}
\label{eq:mark-up1}
\lim_{k\to\infty}\rho(\gamma_0\ldots\gamma_{j_k})&=0.\gamma_0\gamma_1\ldots=x\\
\label{eq:mark-up2}
\lim_{k\to\infty}\rho(\mathfrak a_s(\gamma_0\ldots\gamma_{j_k}))&= 
\lim_{k\to\infty}\rho(\mathfrak a(w_{M(i_k)}))=g(x)
\end{align}
Recall that $\mathfrak a=\mathfrak a_{s_0}$ where $s_0$ is the initial state
of the automaton $\mathfrak A=\mathfrak A(s_0)$.
Note that nowhere in the argument above we have used that $g$ is continuous.

Now consider the case when $x=n/p^r$
for suitable $n,r\in\N_0$, $n\in\{0,1,\ldots,p^r-1\}$. In this case, let
$0.\gamma_0\gamma_1\ldots$ be 
either of base-$p$ expansions 
$\wrd^r(x)$, $\wrd^l(x)$.
We will show that similarly to the case when a base-$p$ expansion is
unique, in the case under consideration there also exist $s\in\Cal S$ and sequences $i_0<i_1<i_2<\ldots$, $M(i_0)<M(i_1)<M(i_2)<\ldots$ over $\N_0$ such 
that $M(i_k)\ge i_k$ for all $k\in\N_0$ and both \eqref{eq:mark-up1}
and \eqref{eq:mark-up2} hold.

For this purpose, consider arbitrary sequence 
$w_0,w_1,w_2,\ldots$ of
right-infinite words over $\{0,1,\ldots,p-1\}$   which are not eventually
periodic and such that $\gamma_0\ldots\gamma_i$ is a prefix of $w_i$ for
all $i=0,1,2,\ldots$. Then $\lim_{i\to\infty}0.w_i=x$ and
therefore all $0.w_i\in (a,b)$ once $i\ge I$ where $I$ is large enough (we may assume that $I=0$; otherwise we consider a subsequence
$(w_i)_{i=I}^\infty$ rather than the whole sequence $(w_i)_{i=0}^\infty$).
Note that then $\lim_{i\to\infty}g(0.w_i)=g(x)$ as $g$ is continuous on $(a,b)$;  therefore there exists a sequence $(S(i))_{i=0}^\infty$
over $\N$ such that for all $i=0,1,2,\ldots$ the following inequality holds: 
\begin{equation}
\label{eq:lim-4}
|g(0.w_i)-g(x)|<p^{-S(i)}
\end{equation} 
Moreover, we may assume that the sequence $(S(i))_{i=0}^\infty$ is strictly
increasing (if not, we consider a corresponding infinite subsequence of the sequence
$(w_i)_{i=0}^\infty$ rather than the whole sequence).

Consider now a word $w_i=\gamma_{i0}\gamma_{i1}\ldots$ from
the above word sequence (note that $\gamma_{i\ell}=\gamma_\ell$ for $\ell=1,2,\ldots,i$).
As every $0.w_i$ 
is a unique
base-$p$ expansion of the corresponding real number from $(a,b)$, there exists
a state $s(i)$ of the automaton $\mathfrak A$ and a strictly increasing sequence $\Cal J(s(i))=
(j_{ik})_{k=0}^\infty$
of numbers from $\N_0$ such that
\begin{equation}
\label{eq:lim-1} 
\lim_{k\to\infty}\rho(\mathfrak a_{s(i)}(\gamma_{i0}\gamma_{i1}\ldots\gamma_{ij_{ik}}))=
g(0.w_i), 
\end{equation}
cf. the case we just have considered above
at the beginning of the proof of the proposition. Therefore, for any $k\in\N_0$
there exists $K(j_{ik})\in\N$ such that
\begin{equation}
\label{eq:lim-2}
|\rho(\mathfrak a_{s(i)}(\gamma_{i0}\gamma_{i1}\ldots\gamma_{ij_{ik}}))
-g(0.w_i)|<p^{-K(j_{ik})}, 
\end{equation}
and there exists a strictly increasing sequence of $k$ such that the corresponding
sequence of $K(j_{ik})$ is also strictly increasing in force of \eqref{eq:lim-1}.
Without loss of generality we may assume that the sequence $(K(j_{ik}))_{k=0}^\infty$
is strictly increasing (otherwise we consider a corresponding subsequence
of the sequence $\Cal J(s(i))$).

As a total number of states of $\mathfrak
A$ is finite, in the infinite sequence $(s(i))_{i=0}^\infty$ at least one
state, say $s$, occurs infinitely many times. We may assume that $s(i)=s$
for all $i\in\N_0$; 
otherwise we just consider respective subsequence of
$(w_i)_{i=0}^\infty$ rather than the whole sequence. As
the sequence $(j_{ik})_{k=0}^\infty$ is strictly increasing, all $j_{ik}>i$ once
$k$ is large enough. Given $i\in\N_0$, denote via $N(i)\in\N$  the smallest number
such that  $j_{ik}>i$ once $k\ge N(i)$. We again may assume that $N(i)=0$;
if otherwise we will just consider the subsequence $(j_{ik})_{k=N(i)}^\infty$
rather than the whole sequence $(j_{ik})_{k=0}^\infty$.  Then 
$\gamma_{i0}\gamma_{i1}\ldots\gamma_{ij_{ik}}=\gamma_0\ldots\gamma_ir_{ik}$
for all $k$ where $r_{ik}$ is a non-empty finite word.

Let $s^\prime(i,k)\in\Cal S$ be a
state the automaton $
\mathfrak A(s)$ reaches after being feeded by the input
word $r_{ik}$ 
(the latter state $s$ is
 defined above). 
As the total number of states of $\mathfrak A(s)$ is finite,
in the sequence $(s^\prime(i,k))_{k=0}^\infty$ at least one state, say $s^\prime(i)$,
occurs infinitely many times. Moreover, by the same reason at least one state,
say $s^\prime$, occurs in the sequence $(s^\prime(i))_{i=0}^\infty$ infinitely
many times. And again, without loss of generality we may assume that $s^\prime(i,k)=s^\prime$
for all $i,k$; otherwise we  consider corresponding subsequences of the sequences
$(w_i)_{i=0}^\infty$ and $(j_{ik})_{k=0}^\infty$. Note that being feeded
by the input word $\gamma_{i0}\gamma_{i1}\ldots\gamma_{ij_{ik}}=\gamma_0\ldots\gamma_ir_{ik}$,
the automaton $\mathfrak A(s)$ (rightmost letters are feeded prior to leftmost
ones), the automaton outputs a word of length $j_{ik}+1$ whose (left) suffix
of length $i+1$ is the word which outputs the automaton $\mathfrak A^\prime=\mathfrak
A(s^\prime)$
if being feeded by the word $\gamma_0\ldots\gamma_i$.  Therefore
\begin{equation}
\label{eq:lim-3}
|\rho(\mathfrak a^\prime(\gamma_0\ldots\gamma_i))-\rho(\mathfrak a_s(\gamma_{i0}\gamma_{i1}\ldots\gamma_{ij_{ik}}))|<p^{-i+1}
\end{equation}
Now combining \eqref{eq:lim-4}, \eqref{eq:lim-2} and \eqref{eq:lim-3}
we conclude that $\lim_{i\to\infty}\rho(\mathfrak a^\prime(\gamma_0\ldots\gamma_i))=g(x)$;
but
$\lim_{i\to\infty}\rho(\gamma_0\ldots\gamma_i)=x$ since $0.\gamma_0\gamma_1\ldots$
is a base-$p$ expansion of $x$. This finally proves the proposition.
%
%
%
%
%
\end{proof}
\begin{note}
\label{note:mark-up}
Actually during the proof of Proposition \ref{prop:mark-up} we have shown
that the following claim is true: \emph{Let the function $g\:U\>\mathbb S$ \textup{(or,
$g\:U\>\mathbb I$)} be
defined on an open neighbourhood $U\subset\mathbb S$ \textup{(or,
$U\subset\mathbb I$)} of a point $x$, let $g$ be continuous at $x$, and let
$\mathbf G(g)\subset\mathbf P(\mathfrak A)$ for a suitable
finite automaton $\mathfrak A$; 
then there exists a mark-up for every base-$p$ expansion of $x$.
Moreover, if $g\:[a,b]\>[0,1]$ is a continuous function on the closed segment
$[a,b]$ then there exist an $s$-mark-up for right base-$p$ expansion of $a$
and for left base-$p$ expansion of $b$.
}
\end{note}
\begin{cor}
\label{cor:mark-up}
In conditions of Proposition \ref{prop:mark-up}, if the automaton $\mathfrak
A$ is minimal then $\mathbf G(g)\subset\mathbf{LP}(\mathfrak A)$.
\end{cor}
\begin{proof}[Proof of Corollary \ref{cor:mark-up}]
Follows immediately from Corollary \ref{cor:AP=LP} by the definition of mark-up.
\end{proof}
The following proposition reduces examination of continuous functions computable by a finite
automaton $\mathfrak A$ for the case when the function is defined on a segment for which  there exists a state $s$ of the automaton such that the set of
all points from the segment that has $s$-mark-ups, is dense in the segment;
and so values at these  points completely specify the function on the segment.
\begin{prop}
\label{prop:cover}
Let  $g\:[a,b]\>[0,1]$, $[a,b]\subset[0,1]$, be a continuous function; let
$\mathbf G(g)\subset\mathbf P(\mathfrak A)$ for
a suitable automaton $\mathfrak A$ whose set of states $\Cal S$ is finite.
Then $[a,b]$ is a union of a countably
many
sub-segments $[a_j^\prime,b_j^\prime]\subset[a,b]$, $a_j^\prime <b_j^\prime$,
$j=1,2,\ldots$ having the following property: For every $j=1,2,\ldots$ 
there exists a state $s_q\in\Cal S$, $q=q(j)$, such that
the set  $M_q([a_j^\prime,b_j^\prime])$
of all points from $[a_j^\prime,b_j^\prime]$ that  have  $s_q$-mark-ups is dense in $[a_j^\prime,b_j^\prime]$.
\end{prop}
\begin{proof}[Proof of Proposition \ref{prop:cover}]
For $s\in\Cal S$ denote via $\Cal W^\infty(s)$  the set of all right-infinite words 
$w\in\Cal
W^\infty$ such that $0.w
\in(a,b)$ and an $s$-mark-up for $w$  
exists; put $0.\Cal W^\infty(s)=
\{0.w
\:w
\in\Cal W^\infty(s)\}$.

Given $s\in\Cal S$ such that $\Cal W^\infty(s)\ne\emptyset$, let $\mathbf
W(s)$ be  intersection of the closure $\overline{\mathbf W}(s)$ of $0.\Cal W^\infty(s)$ with $(a,b)$;
so $\mathbf W(s)$ is closed in $(a,b)$ w.r.t. the induced topology on $(a,b)$
and there are only finitely many pairwise distinct $\mathbf W(s)$; say, these
are $\mathbf W(s_1),\ldots,\mathbf W(s_k)$.  Proposition \ref{prop:mark-up} implies that 
$\mathbf W(s_1)\cup\ldots\cup\mathbf W(s_k)=(a,b)$. We argue that for some
$\mathbf W(s_1),\ldots,\mathbf W(s_k)$ their interiors $\mathbf W(s_1)^o,\ldots,\mathbf W(s_k)^o$ are not empty. 
Indeed, $\mathbf W(s_i)^o=\overline{\mathbf W}(s_i)^o\cap(a,b)$ for all $i=1,2,\ldots,k$;
but from Proposition \ref{prop:mark-up} it follows that 
$\overline{\mathbf W}(s_1)\cup\ldots\cup\overline{\mathbf W}(s_k)=[a,b]$
and therefore $\overline{\mathbf W}(s_1)^o\cup\ldots\cup\overline{\mathbf W}(s_k)^o$ is dense in $[a,b]$ as $[a,b]$ is Baire, cf. e.g. \cite[Theorems
6.16--6.17]{RealAnalys}. 

As some (without loss of generality we may assume
that all) of the interiors 
$\overline{\mathbf W}(s_1)^o,\ldots,\overline{\mathbf W}(s_k)^o$ are non-empty, the interiors are countable unions of open intervals: $\overline{\mathbf W}(s_i)^o=\cup_{\ell=1}^\infty(a^\prime_{i\ell},b^\prime_{i\ell})$, $a^\prime_{i\ell}<b^\prime_{i\ell}$,
$(i=1,2,\ldots,k)$. Therefore 
$\overline{\mathbf W}(s_i)=\cup_{\ell=1}^\infty[a^\prime_{i\ell},b^\prime_{i\ell}]$.
This completes the proof as $\overline{\mathbf W}(s_1)\cup\ldots\cup\overline{\mathbf W}(s_k)=[a,b]$.
\end{proof}
\begin{cor}
\label{cor:cover}
In conditions of Proposition \ref{prop:cover}, the segment $[a,b]$
admits a countable covering by closed sub-segments $[a_j^\prime,b_j^\prime]\subset[a,b]$
such that  the graph  $\mathbf G(g_j)$ of the restriction of the function $g$ to the
sub-segment $[a_j^\prime,b_j^\prime]$ lies in $\mathbf P(\mathfrak
A(s))\subset\mathbf P(\mathfrak A)$
for a suitable
sub-automaton $\mathfrak A(s)$ of the automaton $\mathfrak A$, $s=s(j)\in\Cal S$. 
\end{cor}
\begin{proof}[Proof of Corollary \ref{cor:cover}]
Indeed, the closure of the point
set $\{(x;g(x))\colon x\in M_q([a_j^\prime,b_j^\prime])\}$ in $\R^2$ is a graph of the restriction $g_j$ of the
function $g$ to the segment $[a_j^\prime,b_j^\prime]$ as $g_j$ is continuous
on $[a_j^\prime,b_j^\prime]$. On the
other hand the closure must lie in $\mathbf P(\mathfrak A(s))$ for $s=s_q$
as
$g(x)=\lim_{k\to\infty}\rho(\mathfrak a_s(\gamma_0\ldots\gamma_{i_k}))$
where $x=0.\gamma_0\gamma_1\ldots$ and $i_0,i_1,\ldots$ is an $s$-mark-up.
Note also that $\mathbf P(\mathfrak A(s))\subset\mathbf P(\mathfrak A)$ as
every state of the automaton $\mathfrak A$ is reachable from its initial
state $s_0$ (since we consider reachable automata only, cf. Subsection \ref{ssec:auto}).
\end{proof}
\begin{note}
\label{note:cover}
From the respective proofs it follows that both Proposition \ref{prop:cover}
and Corollary \ref{cor:cover} remain true for a continuous function $g\:[a,b]\>\mathbb
S$ as well as for the case when $[a,b]\subset\mathbb S$.
\end{note}
The following theorem shows that we may restrict our considerations of finitely
computable continuous functions to the case when computing automata are minimal.
\begin{thm}
\label{thm:eval-erg}
Given a continuous function $g\:[a,b]\to[0,1]$, $[a,b]\subset[0,1]$
such that $\mathbf G(g)\subset\mathbf P(\mathfrak A)$ for a  finite
automaton $\mathfrak A$, there exists a 
countable covering $\{[a_j^\prime,b_j^\prime]\subset[a,b]:j=1,2,\ldots; a_j^\prime<b_j^\prime\}$
of the segment $[a,b]$ such that for every $j$ the graph $\mathbf G(g_j)$
of the   restriction $g_j$ of the function $g$ to the segment $[a_j^\prime,b_j^\prime]$
lies in
$\mathbf{LP}(\mathfrak
A_n)$ for a suitable 
 minimal sub-automaton $\mathfrak A_n$ of $\mathfrak A$, $n=n(j)$.
\end{thm}
\begin{proof}[Proof of Theorem \ref{thm:eval-erg}]
The state $s_q$ from Proposition \ref{prop:cover} is either ergodic or transient,
see Subsection \ref{ssec:auto}.
We consider these two cases separately.
 
\textbf{Case 1:} The state $s_q$ is ergodic. As the set $M_q([a_j^\prime,b_j^\prime])$ from Proposition \ref{prop:cover}
is dense in $[a_j^\prime,b_j^\prime]$ and $g_j$ is continuous, every point $g_j(x)$
for $x\in[a_j^\prime,b_j^\prime]$ is a limit of a sequence $(g(x_i))_{i=0}^\infty$
where $(x_i)_{i=0}^\infty$ is a sequence  of points from $M_q([a_j^\prime,b_j^\prime])$ and  $(x_i)_{i=0}^\infty$ tends to $x$
as $i$ tends to infinity:
\begin{align} 
\label{eq:MIN-1}
x=&\lim_{i\to\infty}x_i;\\
\label{eq:MIN-2}
g_j(x)=&\lim_{i\to\infty}g_j(x_i).
\end{align} 
But $x_i=0.w_i$ where $w_i$ is a
right-infinite word for which there exists an $s_q$-mark-up (cf. the construction
of the set $M_q([a_j^\prime,b_j^\prime])$); therefore from \eqref{eq:MIN-1}--\eqref{eq:MIN-2}
it follows  now that there exists a sequence $(h_\ell)_{\ell=0}^\infty$ of finite words $h_\ell$ of strictly increasing lengths such that 
\begin{align} 
\label{eq:MIN-1.1}
x=&\lim_{\ell\to\infty}0.h_\ell;\\
\label{eq:MIN-2.1}
g_j(x)=&\lim_{\ell\to\infty}\rho(\mathfrak a_{s_q}(h_\ell)).
\end{align}
Indeed, the words $h_\ell$ are (left) prefixes of words $w_q=\omega_0^{(q)}\omega_1^{(q)}\ldots$
that correspond to $s_q$-mark-up; that is,  $h_\ell$ are of the form 
$\omega_0^{(q_\ell)}\omega_1^{(q_\ell)}\ldots\omega_{r_{q_\ell,k_\ell}}^{(q_\ell)}$
where the sequence $(r_{q_\ell,k})_{k=0}^\infty$ is the $s_q$-mark-up of the word
$w_{q_\ell}$. Now, as the state $s_q$ 
is ergodic (that is, $s_q$ a state of a certain minimal sub-automaton, say
$\mathfrak A_q=\mathfrak A(s_q)$, of the automaton
$\mathfrak A$, cf. Subsection \ref{ssec:auto}), then we just mimic the proof of Theorem \ref{thm:AP=LP} starting
with \eqref{eq:AP=LP-1}--\eqref{eq:AP=LP-2} and show that $(x,g_j(x))\in\mathbf{LP}(\mathfrak
A_q)$.

\textbf{Case 2:} Now let the state $s_q$ from Proposition \ref{prop:cover} 
be not ergodic (whence 
transient). Thus there exists a finite
word $u=\alpha_0\ldots\alpha_{k-1}$
such that after the automaton $\tilde{\mathfrak A}=\mathfrak A(s_q)$ has been feeded by the word
$u$ (rightmost letters are feeded to the automaton prior to leftmost ones),
the automaton reaches some ergodic state (say $t$) which is a state of a minimal
sub-automaton $\mathfrak A^\prime=\mathfrak A(t)$, see Subsection \ref{ssec:auto}. Note that then all words
of the form $vu$ have the same property, for all $v\in\Cal W_\phi$: After
being
feeded by $vu$, the automaton reaches some state from the set of states of
$\mathfrak A^\prime$ due to the minimality of $\mathfrak A^\prime$.
Therefore, the set $B_q\subset\Z_p$ of all $p$-adic integers whose base-$p$ expansions (cf. Subsection \ref{ssec:p-adic})
are left-infinite words $w\in\Cal W^\infty$ such that if the automaton $\mathfrak
A(s_q)$ while being feeded by the word $w$ reaches at a finite step some ergodic
state (that is, reaches a state which is a state of some minimal automaton
of $\mathfrak A$) is a union of balls of non-zero radii in $\Z_p$; thus, the set $B_q$ is an
open subset in $\Z_p$ since every ball of a non-zero radius is open in $\Z_p$
w.r.t. the $p$-adic topology, cf. Subsection \ref{ssec:p-adic}. Hence the set $A_q=\Z_p\setminus B_q$ is a closed
subset of $\Z_p$; and the set $A_q$ consists of all $p$-adic integers such that
if the automaton $\mathfrak A(s_q)$ is being feeded by a left-infinite word
that is a base-$p$ expansion of some $p$-adic integer from $A_q$, the automaton
$\mathfrak A$
never reaches an ergodic state. Let $P_q$ be the set of all finite prefixes
of words from $A_q$; denote $\mathbf P_q$ a closure of the set $0.P_q=\{0.w\: w\in P_q\}$ in $\R$. 

\underline{Claim:}\emph{The interior of $\mathbf P_q$ is empty} (therefore $\mathbf
P_q$ is nowhere dense in $[0,1]$). 

Indeed, if not
then $\mathbf P_q$ contains an open interval $(a_q,b_q)$.  Take a finite
non-empty word $u$ such that $a_q<0.u<b_q$. As  $\mathbf P_q\supset (a_q,b_q)$
then, given an arbitrary finite non-empty word $v=\alpha_1\ldots\alpha_k$
where $\alpha_k\ne 0$, there exists a sequence
$\Cal W(v)=(w_i)_{i=1}^\infty$ of finite non-empty words $w_i\in P_q$ such that $\lim_{i\to\infty}0.w_i=0.uv\in(a_q,b_q)$
(recall that $uv$ is a concatenation of words $u$ and $v$). Therefore either
$uv\in P_q$ (thus $v\in P_q$ by the construction of $P_q$) or 
$\Cal
W(v)$ contains an infinite subsequence of words of the form 
$w_i^\prime =u\alpha_1\ldots\alpha_k^\prime(p-1)^{r_i}$ where $\alpha_k^\prime=\alpha_k-1$,
$r_1<r_2<\ldots$ (recall that $(p-1)^{r_i}$ is a word of length $r_i$ all
whose letters are $p-1$); hence by the construction of $P_q$ there exists an infinite sequence $w_i^{\prime\prime} =\alpha_1\ldots\alpha_k^\prime(p-1)^{r_i}$
over $P_q$. Thus we conclude that once 
$n\in\N$, the closure of $A_q$  in $\Z_p$ (thus, the very set $A_q$ itself  as it is closed) must either contain $n$ or $-n$ (recall that negative rational
integers
in $\Z_p$ are exactly that  ones whose canonical $p$-adic expansions
have only a finitely many terms with coefficients other than $p-1$, cf.
Subsection \ref{ssec:p-adic}). But this implies that $A_q=\Z_p$ as the set
$\{\pm n\:n\in\N\}$ (where $+$ or $-$ are taken in arbitrary order) is dense
in $\Z_p$. 
On the other hand, by the construction the set $A_q$ consists of all $p$-adic integers such that
if the automaton $\mathfrak A(s_q)$ is being feeded by a left-infinite word
that is a base-$p$ expansion of some $p$-adic integer from $A_q$, the automaton
$\mathfrak A$
never reaches an ergodic state; therefore the equality $A_q=\Z_p$ contradicts our assumption that $s_q$
is transient (since then there must exist a left-infinite word $w$ such that at
a finite step the automaton $\tilde{\mathfrak A}=\mathfrak A(s_q)$ reaches an ergodic state if
being feeded by $w$). This proves our claim.

Denote now via $\tilde f$  an automaton function of the automaton $\tilde{\mathfrak A}=\mathfrak A(s_q)$; and for $k=1,2,\ldots$ put
\begin{equation}
\label{eq:plot-1}
E_k^\prime(\tilde f)=\left\{
\left({\frac{{z\md p^k}
}{p^k};\frac{{\tilde f(z)\md p^k}
}{p^k}
}\right)\in\mathbb I^2\: z\in \Z_p\setminus A_q=B_q\right\}
\end{equation} 
a point set in the unit real square $\mathbb I^2=[0,1]\times[0,1]$; then take a
union $E^\prime(\tilde f)=\cup_{k=1}^\infty E_k^\prime(\tilde f)$; denote
via $\mathbf P^\prime(\tilde{\mathfrak A})=\mathbf
P^\prime (\tilde f)$  a closure (in topology of $\R^2$) of the set $E^\prime(\tilde
f)$ (cf. \eqref{eq:plot}). Denote via $g_j$ a restriction of the function
$g$ to
$[a_j^\prime,b_j^\prime]$. 
As the function $g_j$ is continuous on $[a_j^\prime,b_j^\prime]$
(cf. Corollary \ref{cor:cover}) and $\mathbf G(g)\subset \mathbf P(\mathfrak A)$, then necessarily $\mathbf G(g_j)\subset
\mathbf P^\prime(\tilde{\mathfrak A})$ since the set $\mathbf P_q$ is nowhere
dense in $[a_j^\prime,b_j^\prime]$ by Claim 1. 

As  the set  $\Cal S$ of all states of the automaton $\mathfrak A$ is finite, there are only finitely many ergodic components in $\Cal S$; say they are $\Cal S_1,\ldots, \Cal S_m\subset \Cal S$. Given an ergodic component $\Cal S_n$ $(n=1,2,\ldots,m)$ 
denote
\begin{equation*}
E_n
=\left\{
\left({\frac{{z\md p^k}
}{p^k};\frac{{\tilde f(z)\md p^k}
}{p^k}
}\right)\in\mathbb I^2\: z\in B_q, k> k_n(z)
\right\}
\end{equation*}
where $k_n(z)$ is the smallest $k\in\N$
such
that after the automaton $\tilde{\mathfrak A}=\mathfrak A(s_q)$ has been
feeded by the word $\wrd(z\md p^{k_n(z)})$, the automaton reaches a state from $\Cal S_n$. Then the union $E=\cup_{n=1}^m E_n$ is disjoint and $\mathbf
P^\prime(\tilde{\mathfrak A})$ is a closure of $E$ by \eqref{eq:plot-1}. Therefore
$\mathbf
P^\prime(\tilde{\mathfrak A})=\cup_{n=1}^m \mathbf E_n$ where $\mathbf
E_n$ is a closure of $E_n$ in $\mathbb I^2$; hence $\mathbf G(g_j)=\cup_{n=1}^m (\mathbf G(g_j)\cap\mathbf E_n)$. 
Note that  from  the
definition of $\mathbf P(\mathfrak A)$ (cf. Subsection \ref{ssec:plots})
it follows that $\mathbf
E_n=\mathbf P(\mathfrak A_n)$ where $\mathfrak A_n$ is a minimal
sub-automaton (of the automaton $\mathfrak A$) whose set of states is $\Cal S_n$. 

Further, as the function $g_j$ is continuous on  $[a_j^\prime,b_j^\prime]$,
the set $\mathbf G(g_j)$ is closed in $\mathbb I^2$; therefore the set
$\mathbf G_n=\mathbf G(g_j)\cap\mathbf E_n$ is closed in $\R^2$. Hence,
the set $\mathbf R_n=\{x\in[a_j^\prime,b_j^\prime]\:(x,g(x))\in\mathbf E_n\}$
is closed in $\R$ and $[a_j^\prime,b_j^\prime]=\cup_{n=1}^m\mathbf R_n$.
Now by argument similar to that from the proof of Proposition \ref{prop:cover}
we conclude that some of the interiors $\mathbf R_n^o$ must be non-empty
and hence either of the non-empty interiors is a union of a countably many
open intervals. By taking closures of the intervals we see that $[a_j^\prime,b_j^\prime]$
is a union of the closures, that is, $[a_j^\prime,b_j^\prime]$ is a union
of a countably many its closed sub-segments $[a_{j.i}^\prime,b_{j.i}^\prime]$
$(i\in\N_0)$ of non-zero lengths, and the
graph of
the restriction $g_{j.i}$ of $g_j$ to either of the sub-segments lies in $\mathbf
E_n=\mathbf P(\mathfrak A_n)$ for a suitable $n\in\{1,2,\ldots,m\}$. Now
we apply Proposition \ref{prop:cover} substituting 
$g_{j.i}$ for $g$ and $[a_{j.i}^\prime,b_{j.i}^\prime]$ for $[a,b]$; but
as 
every $s_q$ from the statement of Proposition \ref{prop:cover} is now a state
of the minimal sub-automaton $\mathfrak A_n$, we now are in conditions of  Case
1. Therefore $\mathbf G(g_{j.i})\subset \mathbf{LP}(\mathfrak A_n(s_q))$;
but $\mathbf{LP}(\mathfrak A_n(s))=\mathbf{LP}(\mathfrak A_n(t))$ for all
states $s,t$ of the automaton $\mathbf{LP}(\mathfrak A_n(s_q))$ due to the
minimality of the automaton, cf. Note \ref{note:AP=LP}.

This
finally proves the theorem.

\end{proof}
\begin{note}
\label{note:eval-erg}
From the  proof of Theorem \ref{thm:eval-erg} it follows that the theorem
remains true for a continuous function $g\:[a,b]\>\mathbb
S$ as well as for the case when $[a,b]\subset\mathbb S$.
\end{note}
The following proposition shows that we may if necessary consider
only finitely computable continuous functions defined everywhere on the unit segment
$[0,1]$ 
rather than on sub-segments of $[0,1]$. 

\begin{prop}[The similarity]
\label{prop:ext}
If a continuous function 
$g\:[a,b]\>\mathbb S$, $[a,b]\subset[0,1]$, is 
such that $\mathbf G_{[a,b]}(g)\subset \mathbf P(\mathfrak A)$ 
for a suitable
finite automaton $\mathfrak A=\mathfrak A(s_0)$ 
then for every $n,m\in\N_0$ such that 
$m\ge\lfloor\log_pn\rfloor+1$ and $n/p^m,(n+1)/p^m\in[a,b]$
the function $g_d(x)=(p^mg(d+p^{-m}x))\md1$,  where $d=np^{-m}$, 
is continuous 
on $[0,1]$, 
and
$
\mathbf G_{[0,1]}(g_d)\subset\mathbf P(\mathfrak A)$.
\end{prop}
\begin{proof}[Proof of  Proposition \ref{prop:ext}]
As
a base-$p$ expansion of $d$ is $d=0.\chi_0\ldots\chi_{m-1}00\ldots$
then, given a  base-$p$ expansion for $x=0.\zeta_0\zeta_1\ldots\in[0,1]$,
a base-$p$ expansion for $d+xp^{-m}$ is $d+xp^{-m}=
0.
\chi_0\ldots\chi_{m-1}\zeta_0\zeta_1\ldots$
and $d+xp^{-m}\in[a,b]$ for all right-infinite words $\zeta_0\zeta_1\ldots$
(thus, for all $x\in[0,1]$).
Therefore if 
$i_0<i_1<i_2<\ldots$ is a mark-up for 
$\chi_0\ldots\chi_{m-1}\zeta_0\zeta_1\ldots$ (cf. Proposition \ref{prop:mark-up})
then 
$(j_m=i_{r+m}-m)_{m=0}^\infty$, where $r=\min\{\ell\:i_\ell>m\}$, is an $s$-mark-up
of the infinite word $\zeta_0\zeta_1\ldots$ for a suitable state $s\in\Cal
S$ of the automaton $\mathfrak A=\mathfrak A(s_0)$ w.r.t. the function $g$. Hence,
\begin{equation}
\label{eq:ext} 
 \lim_{k\to\infty}\rho(\mathfrak a_s(\zeta_0\zeta_1\ldots\zeta_{j_k}))\equiv
\left(p^m\cdot\left(g\left(d+\frac{x}{p^m}\right)\right)\right)\pmod1
\end{equation}
as
$
\rho(\mathfrak a_s(\zeta_0\zeta_1\ldots\zeta_{j_k}))=
(p^m(\rho(\mathfrak a_s(\chi_0\ldots\chi_{m-1}\zeta_0\zeta_1\ldots\zeta_{j_k}))
))\md1
$.
By our assumption on reachability of the automaton $\mathfrak A$ (cf. Subsection
\ref{ssec:auto}), there exists a finite word $u=u(s)$ such that the automaton $\mathfrak A$ being feeded by
$u$  reaches the state $s$ and outputs the corresponding finite word $u^\prime=\mathfrak
a_{s_0}(u)$; therefore
the automaton $\mathfrak A=\mathfrak A(s_0)$ being feeded by a concatenated finite word 
$\zeta_0\zeta_1\ldots\zeta_{j_k}u$ outputs the concatenated finite word 
$\mathfrak a_s(\zeta_0\zeta_1\ldots\zeta_{j_k})u^\prime$. 
But $\lim_{k\to\infty}\rho(\mathfrak a_{s_0}(\zeta_0\zeta_1\ldots\zeta_{j_k}u))
=\lim_{k\to\infty}\rho(\mathfrak a_s(\zeta_0\zeta_1\ldots\zeta_{j_k})u^\prime)
=\lim_{k\to\infty}\rho(\mathfrak a_s(\zeta_0\zeta_1\ldots\zeta_{j_k}))$ and
simultaneously
$x=\lim_{k\to\infty}\rho(\zeta_0\zeta_1\ldots\zeta_{j_k}u)=
\lim_{k\to\infty}\rho(\zeta_0\zeta_1\ldots\zeta_{j_k})$ since the words $u,u^\prime$
are finite and fixed; therefore
$(x;g_d(x)\md1)\in\mathbf{P}(\mathfrak A)$ for all $x\in[0,1]$ 
in view of
\eqref{eq:ext}.

The function $g_d$ is conjugated to a continuous function by a continuous
map and therefore is also continuous: Once $e,h\in[0,1]$ are such that 
$|e-h|<p^{-K(L)}$ to ensure that $|g(d+p^{-m}e)-g(d+p^{-m})h|<p^{-L}$ for
a sufficiently large $L\in\N$ then $|g_d(e)-g_d(h)|<p^{-L+m}$.


\end{proof}
\begin{cor}
\label{cor:ext}
If a continuous function 
$g\:[a,b]\>\mathbb S$, $[a,b]\subset[0,1]$, is  
such that $\mathbf G_{(a,b)}(g)\subset \mathbf P(\mathfrak A)$ 
for a suitable
finite automaton $\mathfrak A=\mathfrak A(s_0)$ 
then for every $n,m\in\N_0$ such that 
$m\ge\lfloor\log_pn\rfloor+1$ and $d=n/p^m
\in[a,b)$
\begin{itemize}
\item the function $g_{d,M}(x)=(p^Mg(d+p^{-M}x))\md1$  
is continuous 
on $[0,1]$ for all sufficiently large $M\ge m$,
and
\item $\mathbf G_{[0,1]}(g_{d,M})\subset\mathbf P(\mathfrak A)$.
\end{itemize}

\end{cor}
\begin{proof}[Proof of Corollary \ref{cor:ext}]
Indeed, in the proof of Proposition \ref{prop:ext} as a base-$p$ expansion for $d=np^{-m}$ just use  
$0.\chi_0\ldots\chi_{m-1}(0)^{M-m}$ 
where
$M\ge m$ is large enough so that $0.\chi_0\ldots\chi_{m-1}(0)^{M-m-1}1\in[a,b]$.
Note that nowhere in the proof of the proposition we used that some  of $\chi_0,\ldots,\chi_{m-1}$ are not zero.
\end{proof}
\begin{note}
\label{note:ext}
Corollary \ref{cor:ext} shows that given any point $d^\prime\in[a,b)$ and a rational approximation $d=np^{-m}$ of
$d^\prime$, the graph
of the function $g$ on
a sufficiently small closed neighbourhood $[a^\prime,b^\prime]$ of the point $d^\prime\ne b^\prime$ is similar to
the graph of the function $g_{d,M}$ on $[0,1]$ where $d=np^{-m}$ and $M$ is large enough.
\end{note}
Summarizing results of the current subsection we may say that while considering
a continuous function $g\:[a,b]\>\mathbb S$ (where $[a,b]\subset[0,1]$ or
$[a,b]\subset\mathbb S$) whose graph $\mathbf G(g)$ lies in $\mathbf P(\mathfrak A)$ for some finite automaton $\mathfrak A$  one can if necessary assume that the function is defined and continuous 
on $[0,1]$ (or on $\mathbb S$ except for maybe a single point),  the automaton $\mathfrak A$ is minimal, the function $g$ is
ultimately computable by $\mathfrak A$ and that for some state $s$ of $\mathfrak
A$ the set of all points from $[0,1]$ which  have  base-$p$ expansions admitting
$s$-mark-ups is dense in $[0,1]$ (respectively, in $\mathbb S$).
\subsection{Finite computability of compositions}
It is clear that 
a composition of finitely computable continuous
functions
should be  a finitely computable continuous function.
The following proposition states this formally and gives some extra information
about the graph of a composite finitely computable function.
\begin{prop}
\label{prop:comp}
Let $[a,b],[c,d]\subset[0,1]$ and let $g\:[a,b]\>[0,1]$, $f\:[c,d]\>[0,1]$ be two continuous functions such
that $g([a,b])\subset[c,d]$ and there exist finite automata $\mathfrak A$
and $\mathfrak B$ such that $\mathbf G_{[a,b]}(g)\subset\mathbf P(\mathfrak A)$,
$\mathbf G_{[c,d]}(f)\subset\mathbf P(\mathfrak B)$. Then 
%
there exists a 
covering $\{[a_j^\prime,b_j^\prime]\subset[a,b]\:j\in J\}$ 
such that
if   $h_j$ is a restriction of the composite function $f(g)$ to the sub-interval
$[a_j^\prime,b_j^\prime]$ then $\mathbf G_{[a_j^\prime,b_j^\prime]}(h_j)\subset\mathbf P(\mathfrak
C_j)$ for every $j\in J$, where $\mathfrak C_j$ is a sequential composition of the automaton $\mathfrak
A(s_j)$ with the automaton  $\mathfrak B(t_j)$ and $s_j, t_j$ are suitable \textup{(depending
on $j$)}
states
of the automata $\mathfrak A$, $\mathfrak B$ accordingly.
%
\end{prop}
\begin{proof}[Proof of Proposition \ref{prop:comp}]
By Note \ref{note:mark-up}, for every right-infinite word
$w=\gamma_0\gamma_1\ldots\in\Cal
W^\infty$ such that $x=0.w\in(a,b)$ there
exists a mark-up (w.r.t. some state $s$ of the finite automaton $\mathfrak
A$) $i_0,i_1,i_2,\ldots$; i.e., $\lim_{k\to\infty}\rho(\mathfrak
a(w_k))=g(x)$, where $w_k=\gamma_0\gamma_1\ldots\gamma_{i_k}\in\Cal
W$; and if $x=a$ (respectively, $x=b$) then the mark-up exists at least for right (respectively, left) base-$p$ expansion.  By the same reason, for $y=g(x)=
0.v$, where 
$v=\nu_0\nu_1\ldots\in\Cal W^\infty$, there exists
a mark-up $j_0,j_1,\ldots$ (w.r.t. some state $t$ of the finite automaton
$\mathfrak B$) such that 
$\lim_{n\to\infty}\rho(\mathfrak
b(v_n))=g(y)$, where $v_n=\nu_0\nu_1\ldots\nu_{i_n}\in\Cal
W$. Now for $m\in\N_0$ denote $N(m)=\min\{k\:i_k\ge j_m\}$, consider the
sequence $(N(m))_{m=0}^\infty$ and let $q(0)=N(m_0),q(1)=N(m_1),\ldots$ be a strictly increasing
subsequence of $(N(m))_{m=0}^\infty$. Denote $s(\ell)$ the state the automaton $\mathfrak
A$ reaches after being feeded by the word 
$\gamma_{j_{m_\ell}+1}\gamma_{j_{m_\ell}+2}\ldots\gamma_{i_{q(\ell)}}$;
put $s(\ell)=s$ if the latter word is empty. As the automaton $\mathfrak
A$ is finite, there is a state, say $s^\prime$, that occurs in the sequence
$(s(\ell))$ infinitely often. Then the sequence $({j_{m(\ell)}}\: s(\ell)=s^\prime)$
is a mark-up of the word $w$ w.r.t. the automaton $\mathfrak
A(s^\prime)$, and simultaneously the same sequence is a mark-up of the word
$v$ w.r.t. the automaton $\mathfrak B(t)$. Therefore
$\lim_{\ell\to\infty}\rho(\mathfrak c^\prime(\gamma_0\gamma_1\ldots\gamma_{j_{m(\ell)}})))
=f(y)=f(g(x))$, where $\mathfrak C^\prime$ is a sequential composition of automata
$\mathfrak A(s^\prime)$ and $\mathfrak B(t)$. 
%

By Corollary \ref{cor:cover}, the segment $g([a,b])$ can be covered by a
countably many segments $[c_k,d_k]$, $k\in\N$ where for every $k$ there
exists a state $t_k$ of the automaton $\mathfrak B$ such that the set of
all points from $[c_k,d_k]$ whose base-$p$ expansions (w.r.t. the function
$f$) admit $t_k$-mark-ups
is dense in $[c_k,d_k]$. Given a real number $y\in[c_k,d_k]$ and its base-$p$
expansion, in view of Proposition \ref{prop:mark-up} there exists a $t_k$-mark-up
of the base-$p$-expansion. Having this mark-up and by acting as above, we,
given $x\in g^{-1}(y)$ find corresponding state $s^\prime_k$
of the automaton $\mathfrak A$  and construct a strictly increasing sequence
over $\N$ such that the sequence is simultaneously a mark-up for $y$ (w.r.t.
$t_k$ and the function $f$) and for $x$ (w.r.t. $s^\prime_k$ and the function
$g$).

Let $s^\prime_1,\ldots,s^\prime_r$ be all pairwise distinct states of
the automaton $\mathfrak A$ that satisfy the following condition:  For every
$w,v\in\Cal W^\infty$ such
that $0.v\in[a,b]$, 
$g(0.w)=0.v$
there exists an $s_i^\prime$-mark-up (for suitable $i\in\{1,2,\ldots,r\}$)
such that the mark-up is a mark-up both
for $w$ (w.r.t. $s_i^\prime$ and $g$) and for
$v$ (w.r.t. $t_k$ and $f$) simultaneously. 
For  $i\in\{1,2,\ldots,r\}$
denote via $\Cal W^\infty(s_i^\prime)$  the set of all infinite words 
$w\in\Cal
W^\infty$ such that 
there exists
an $s$-mark-up which is a mark-up both for $w$ and for $v$ simultaneously;
then proceeding in the same
way as in the proof of Proposition \ref{prop:cover} we conclude that
there exists $s^\prime=s^\prime_i$ and a closed subinterval $[a^\prime,b^\prime]$
such that  $W=\Cal W^\infty(s_i^\prime)\cap[a^\prime,b^\prime]$ is dense in
$[a^\prime,b^\prime]$. But then $g(W)$ is dense in $g([a^\prime,b^\prime])$
and $f(g(W))$ is dense in $[f(g([a^\prime,b^\prime]))$
as $g$ is continuous on $[a^\prime,b^\prime]$ and $f$ is continuous on $g([a^\prime,b^\prime])$.
Therefore for a finite automaton $\mathfrak C_{ik}^\prime$ which is a sequential composition of the automata $\mathfrak A(s_i^\prime)$ and  $\mathfrak B(t_k)$ we have that the graph of the restriction $h$ of
the function $f(g)$ to $[a^\prime,b^\prime]$ 
lies in $\mathbf P(\mathfrak C^\prime_{ik})$.  

\end{proof}

\begin{note}
\label{note:comp}
By arguing as in the proof of Proposition \ref{prop:comp} the following can
be shown: 
\emph{
Let $[a,b]\subset[0,1]$, let $g\:[a,b]\>\mathbb S$, $f\:[a,b]\>\mathbb
S$ be two continuous functions, and let
there exist finite automata $\mathfrak A$
and $\mathfrak B$ such that $\mathbf G_{[a,b]}(g)\subset\mathbf P(\mathfrak A)$,
$\mathbf G_{[a,b]}(f)\subset\mathbf P(\mathfrak B)$. Then 
there exists a
covering $\{[a_j^\prime,b_j^\prime]\subset[a,b]\:j\in J\}$ 
such that
if   $h_j$ is a restriction of the function $(f+g)\md1$ to the sub-interval
$[a_j^\prime,b_j^\prime]$ then $\mathbf G_{[a_j^\prime,b_j^\prime]}(h_j)\subset\mathbf P(\mathfrak
C_j)$ for every $j\in J$, where $\mathfrak C_j$ is a sum of the automaton $\mathfrak
A(s_j)$ with the automaton  $\mathfrak B(t_j)$ and $s_j, t_j$ are suitable \textup{(depending
on $j$)}
states
of the automata $\mathfrak A$, $\mathfrak B$ accordingly. 
}
Here by the sum of automata $\mathfrak A$ and $\mathfrak B$ we mean a sequential
composition of the automata by automaton which has two inputs and a single
output and performs addition of $p$-adic integers. The latter automaton is
finite, see Subsection \ref{ssec:a-map} and Proposition \ref{prop:fin-auto}.
Note also that we may assume that both $f$ and $g$ are defined on an arc
of $\mathbb S$ rather than on $[a,b]$.
\end{note}
\begin{cor}
\label{cor:sum}
Given $A,B\in\Z_p\cap\Q$ and continuous finitely computable functions $f,g\:[a,b]\>\mathbb
S$, there exists a covering $\{[a_j^\prime,b_j^\prime]\subset[a,b]\:j\in J\}$ such that the function $Af+Bg$ is finitely computable on every $[a^\prime_j,b^\prime_j]$.
\end{cor} 

Comparing Theorem \ref{thm:eval-erg} with Proposition \ref{prop:comp} we
see that in the class of continuous
functions there is no big difference between finite computability and ultimate
finite computability since given a finitely computable continuous function
on a segment there exists a  covering of the segment by sub-segments
such that the function is ultimately finitely computable on either of the sub-segments.

\section{Main theorems}
\label{sec:main}
In this section we prove that a graph of any $C^2$-smooth finitely computable function $g\:[a,b]\>\mathbb
S$, $[a,b]\subset[0,1)$, lies (under a natural association of the half-open interval
$[0,1)$ with the unit circle $\mathbb S$) on a torus winding with a $p$-adic rational slope; and if
$\mathfrak A$ is a finite automaton that computes $g$ then necessarily the
graph of the automaton contains the whole winding. Moreover, we prove a generalization
of this theorem for  multivariate functions. To make further proofs  (which
are somewhat involved) more
transparent  we begin with a brief (and no too rigorous) outline  of  their general underlying idea.

Given $g$ as above, fix $x=n p^{-m}\in[a,b]$; then for $h\in[0,1]$ and all
sufficiently large $\ell$ from  the
differentiability of $g$ it follows  that
$g(x+p^{-m-\ell}h)=
g(x)+g^\prime(x)\cdot p^{-m-\ell}h+p^{-m-t(\ell)}\theta(\ell,h)$, where $|\theta(\ell,h)|\le 1$  
and $t$ is a map from $\N_0$ to $\N_0$ such
that $p^{t(\ell)}\to\infty$ faster than $p^{\ell}\to\infty$ while $\ell\to\infty$.
Once $h$ is fixed (say, $h=p^{-1}$) then the above equality for large $\ell$
implies (in view of Proposition \ref{prop:comp} and Corollary \ref{cor:sum}) that there exists a finite automaton  $\mathfrak B_{x}$ which computes
$(g^\prime(x))\md1=((g(x+p^{-m-\ell-1})-
g(x))p^{m+\ell+1}-p^{1+\ell-t(\ell)}\theta(\ell,p^{-1}))\md1$ being feeded by an   infinite sequence of zero words whose lengths increase unboundedly,
i.e.,  $g^\prime(x)\in\mathbf
P(\mathfrak B_{x})$: This is because, speaking loosely, the error term $p^{1+\ell-t(\ell)}\theta(\ell,h)$
makes no perturbations of the infinite output sequence due to the fast growth
of $t(\ell)$. But then necessarily $g^\prime(x)\in\Z_p\cap\Q$ by Proposition
\ref{prop:auto-Q}. 
Further Lemma \ref{le:der} proves this fact rigorously. 

We then (see Lemma \ref{le:der2} below) play similar trick with the second derivative $g^{\prime\prime}(x)$: As $g$ is two times differentiable and $g^\prime(x)\in\Z_p\cap\Q$,
the function $g_1(u)=g(u)-g^\prime(x)\cdot u+c$ 
of argument $u\in[a,b]$ 
is also a $C^2$-smooth finitely computable function for every $c\in\Q\cap\Z_p$.
As $g_1^\prime(x)=0$,
$g_1^{\prime\prime}(u)=g^{\prime\prime}(u)$,
we have (for all sufficiently large $\ell$) that 
$g_1(x+p^{-m-\ell}h)=g_1(x)+
\frac{g^{\prime\prime}(x)}{2}
\cdot p^{-2m-2\ell}h^2
+p^{ -2m-t_1(\ell)}\theta_1(\ell,h)$ where $|\theta_1(\ell,h)|\le 1$, $t(\ell)=2\ell+w(\ell)$,  and $w$ is a map from $\N_0$ to $\N_0$ such that
$w(\ell)\to\infty$ as $\ell\to\infty$. From here in a way similar to that of above we deduce that $\frac{g^{\prime\prime}(x)}{2}\in\Z_p\cap\Q$. But
then, if $g^{\prime\prime}(x)\ne 0$, the argument means that there exists a finite automaton which performs squaring $h\to
h^2$ of every $h\in[0,1]$ with arbitrarily high accuracy. However as it is well known (cf.
Subsection \ref{ssec:auto})
no finite automaton can do such squaring; so necessarily $g^{\prime\prime}(x)=0$
for all $x=n p^{-m}\in[a,b]$. But the set of these $x$ is dense in $[a,b]$;
therefore $g^{\prime\prime}(x)=0$ for all $x\in[a,b]$ as $g^{\prime\prime}$ is
continuous on $[a,b]$. Hence $g$ must be affine: $g(u)=g^{\prime}(x)u+e$
for all $u\in[a,b]$.
Note that then necessarily $e\in\Z_p\cap\Q$ since $e=g(0)$ and $g$ is finitely computable, cf. Proposition \ref{prop:auto-Q}.  After that by Proposition
\ref{prop:ext} we can `stretch' the graph of the function $g$ from $[a,b]$
to the whole unit circle $\mathbb S$ and thus finally obtain a whole cable
 which lies in the plot of the finite automaton which calculates $g$.
But then by Theorem \ref{thm:mult-add-p} the plot must contain the whole
link of torus windings; and the graph $\mathbf G_{[a,b]}(g)$ must lie completely
on some of these windings. The number of links is finite since every link
corresponds to some minimal sub-automaton  (see Subsection \ref{ssec:auto}
and Theorem \ref{thm:eval-erg})
of the automaton which computes
$g$; and the number of minimal sub-automata of a finite automaton is clearly
a finite. Finally, every such link corresponds to a finite family of complex-valued
exponential functions of the form $
\psi_k(y)=e^{i(Ay-2\pi p^kB)}$, 
$k=0,1,2,\ldots$,
for suitable $A,B\in\Z_p\cap\Q$ as shown in Corollary \ref{cor:mult-add-compl}.
Figures \ref{fig:2link-square} and \ref{fig:2link-torus} illustrate how
the graphs of $C^2$-functions from the plots of finite automata look like.

Now we proceed with rigorous assertions and proofs.
\subsection{The univariate case}
\label{ssec:main-uni}
Here we show that $C^2$-smooth finitely computable  functions defined on
$[a,b]\subset[0,1)$ and valuated in $[0,1)$ are only affine ones. Once we associate
the half-open interval
$[0,1)$ with a unit  circle $\mathbb S$ under a natural bijection we may
consider  graphs of the functions as  subsets on a surface of the unit torus
$\mathbb T^2=\mathbb S\times\mathbb S$. We show that then the graphs lie
only on cables of the torus $\mathbb T^2$, and the slopes of the cables must
be $p$-adic rational integers (i.e., must lie in $\Z_p\cap\Q$), see Subsection
\ref{ssec:knot} for definitions of torus knots, cables of torus,  and links
of knots.
\begin{thm}
\label{thm:main}
Consider a finite automaton $\mathfrak A$ 
and 
a continuous 
function $g$ with domain $[a,b]\subset[0,1)$,  valuated in $\mathbb [0,1)$.  Let 
$
\mathbf G(g)\subset \mathbf
{P}(\mathfrak A)$, 
let $g$ be two times 
differentiable on $[a,b]$, and let the second derivative $g^{\prime\prime}$ of $g$
be continuous on $[a,b]$. Then there exist $A,B\in\Q\cap\Z_p$
such that
$g(x)=(Ax+B)\md1$ for all
$x\in[a,b]$; moreover, the graph $\mathbf G_{[a,b]}(g)$ of the function $g$ lies completely
in the cable $\mathbf C(A,B)\subset\mathbf{LP}(\mathfrak A)$ and
$\mathbf C(A,\bar B)\subset\mathbf{LP}(\mathfrak A)$  for all $\bar B\in\mathbf
C(B\md1)$.

Given a finite automaton $\mathfrak A$, there are no more than a finite number of pairwise distinct  cables $\mathbf C(A,B)$ 
of the
unit torus
$\mathbb T^2$ such that $\mathbf C(A,B)\subset
\mathbf {P}(\mathfrak
A)$ 
\textup{(note that $A,B\in\Z_p\cap\Q$ then)}.
\end{thm}
\begin{figure}[ht]
\begin{minipage}[b]{.45\linewidth}
\begin{quote}
\psset{unit=0.3cm}
 \begin{pspicture}(0,0)(15,15)
 \newrgbcolor{emerald}{.31 .78 .47}
 \newrgbcolor{el-blue}{.49 .98 1}
 \newrgbcolor{baby-blue}{.54 .81 .94}
 \newrgbcolor{copper}{.722 .451 .2}
 \newrgbcolor{ocher}{.8 .467 .133}
 \newrgbcolor{gold}{1 .843 .0}
 \newrgbcolor{erin}{0 1 .247}
 \newrgbcolor{r-carmine}{.95 0 .094}
  \psframe[
  linewidth=0.5pt,linecolor=green](-4.8,0)(13.2,18)
  \pscircle[fillstyle=solid,fillcolor=yellow,linewidth=.4pt](-4.8,0.1){.3}
  \psline[unit=0.359cm,linecolor=r-carmine,linewidth=2pt](-4,10)(1,0)
  \psline[unit=0.359cm,linecolor=r-carmine,linewidth=2pt](1,15)(8.5,0)
  \psline[unit=0.359cm,linecolor=r-carmine,linewidth=2pt](8.5,15)(11,10)
  \psline[unit=0.359cm,linecolor=erin,linewidth=2pt](-4,5)(-1.5,0)
  \psline[unit=0.359cm,linecolor=erin,linewidth=2pt](-1.5,15)(6,0)
  \psline[unit=0.359cm,linecolor=erin,linewidth=2pt](6,15)(11,5)
  \psline[unit=0.359cm,linecolor=baby-blue,linewidth=2pt](-4,1.71)(11,10.71)
  \psline[unit=0.359cm,linecolor=gold,linewidth=2pt](-4,0.43)(11,9.43)
  \psline[unit=0.359cm,linecolor=ocher,linewidth=2pt](-4,0.86)(11,9.86)
  \psline[unit=0.359cm,linecolor=baby-blue,linewidth=2pt](-4,4.71)(11,13.71)
  \psline[unit=0.359cm,linecolor=gold,linewidth=2pt](-4,3.43)(11,12.43)
  \psline[unit=0.359cm,linecolor=ocher,linewidth=2pt](-4,3.86)(11,12.86)
  \psline[unit=0.359cm,linecolor=baby-blue,linewidth=2pt](-4,7.71)(8.14,15)
  \psline[unit=0.359cm,linecolor=gold,linewidth=2pt](-4,6.43)(10.29,15)
  \psline[unit=0.359cm,linecolor=ocher,linewidth=2pt](-4,6.86)(9.3,15)
 \psline[unit=0.359cm,linecolor=baby-blue,linewidth=2pt](-4,10.71)(3.14,15)
  \psline[unit=0.359cm,linecolor=gold,linewidth=2pt](-4,9.43)(5.29,15)
  \psline[unit=0.359cm,linecolor=ocher,linewidth=2pt](-4,9.86)(4.3,15)
 \psline[unit=0.359cm,linecolor=baby-blue,linewidth=2pt](-4,13.71)(-1.86,15)
  \psline[unit=0.359cm,linecolor=gold,linewidth=2pt](-4,12.43)(0.29,15)
  \psline[unit=0.359cm,linecolor=ocher,linewidth=2pt](-4,12.86)(-0.3,15)
  \psline[unit=0.359cm,linecolor=baby-blue,linewidth=2pt](-1.86,0)(11,7.71)
  \psline[unit=0.359cm,linecolor=gold,linewidth=2pt](0.29,0)(11,6.43)
  \psline[unit=0.359cm,linecolor=ocher,linewidth=2pt](-0.3,0)(11,6.86)
  \psline[unit=0.359cm,linecolor=baby-blue,linewidth=2pt](3.14,0)(11,4.71)
  \psline[unit=0.359cm,linecolor=gold,linewidth=2pt](5.29,0)(11,3.43)
  \psline[unit=0.359cm,linecolor=ocher,linewidth=2pt](4.3,0)(11,3.86)
  \psline[unit=0.359cm,linecolor=baby-blue,linewidth=2pt](8.14,0)(11,1.71)
  \psline[unit=0.359cm,linecolor=gold,linewidth=2pt](10.29,0)(11,.43)
  \psline[unit=0.359cm,linecolor=ocher,linewidth=2pt](9.3,0)(11,.86)
 \end{pspicture}
\end{quote}
\caption{%
{The limit plot in $\R^2$ of an automaton that has two affine subautomata $\mathfrak
A$ and $\mathfrak B$;
$f_{\mathfrak A}(z)=-2z+\frac{1}{3}$ and $f_{\mathfrak B}(z)=\frac{3}{5}z+\frac{2}{7}$,
where $z\in\Z_2$.\qquad \qquad \qquad \qquad \qquad \qquad \qquad \qquad \qquad \qquad \qquad \qquad \qquad \qquad \qquad \qquad \qquad \qquad
  }}
\label{fig:2link-square}
\end{minipage}\hfill
\begin{minipage}[b]{.45\linewidth}
\includegraphics[width=0.92\textwidth,natwidth=610,natheight=642]{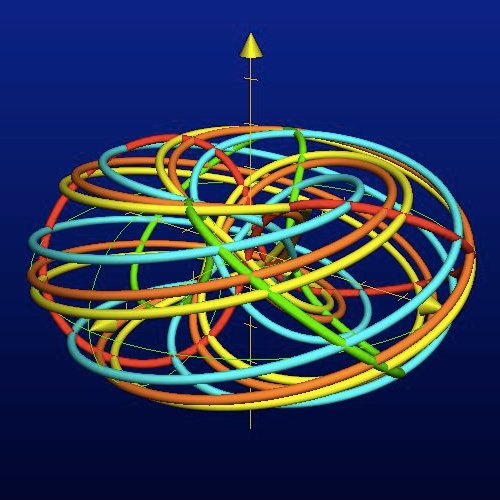}
\caption{The limit plot  of the same automaton on the torus $\mathbb T^2$ in $\R^3$. The
plot consists of two torus links; the links consist of 2 and of 3 knots accordingly.\qquad}
\label{fig:2link-torus}
\end{minipage}
\end{figure}

\begin{lem}
\label{le:der}
Consider a finite automaton $\mathfrak A$ 
and 
a continuous 
function $g$ with domain $[a,b]\subset[0,1)$  valuated in $\mathbb [0,1)$. Let 
$\mathbf G(g)\subset \mathbf
{P}(\mathfrak A)$ and let $g$ be differentiable at the point
$x=n p^{-m}\in[a,b)$ 
where 
$n\in\N_0$. 
Then 
$g^\prime(x)\in\Z_p\cap\Q$.
\end{lem}
\begin{lem}
\label{le:der2}
Under conditions of Theorem \ref{thm:main} let $x$ be the same as in the
statement of Lemma \ref{le:der}%
; then $g^{\prime\prime}(x)=0$. 
\end{lem}
\begin{proof}[Proof of Lemma \ref{le:der}]

Under
conditions of the lemma,
the right base-$p$ expansion of $x$ is
$x=0.\gamma_0\ldots\gamma_{m-1}00\ldots$, for suitable 
$\gamma_0,\ldots,\gamma_{m-1}\in\{0,1,\ldots,p-1\}$. 
We claim that 
$(p^mg(x))\md1\in\Z_p\cap\Q$.
Indeed, as $g$ is continuous and as $x=0.v0^\infty$ where $v=\gamma_0\ldots\gamma_{m-1}$, 
there exists an $s$-mark-up
of the right-infinite word $v0^\infty$ w.r.t. some state $s\in\Cal S$ (cf. Note
\ref{note:mark-up}). That is, there exists a strictly increasing sequence
$k_0<k_1<\ldots$ over $\N$ such that for the  infinite sequence of words
$w_i=v0^{k_i-m}$ of strictly increasing lengths $k_i$ (where $i\ge K$ and
$K$ is large enough so that $k_i-m>0$) the following is true:
\begin{align*}
\lim_{i\to\infty}0.w_i&=x;\\
\lim_{i\to\infty}0.\mathfrak a_s(w_i)&=g(x).
\end{align*}
Therefore, the mark-up $(\bar k_j=k_{K+j}-m)_{j=0}^\infty$  of the
zero right-infinite word $0^\infty$ is such that (in the notation of Proposition
\ref{prop:ext}) the following equalities hold simultaneously
\begin{align*}
\lim_{j\to\infty}0.0^{\bar k_j}&=0;\\
\lim_{j\to\infty}0.\mathfrak a_s(0^{\bar k_j})&=g_d(0).
\end{align*} 
Now by combining Proposition \ref{prop:ext} (or Corollary \ref{cor:ext} if
necessary)  and Proposition \ref{prop:auto-Q}
we conclude that $g_d(0)\in\Z_p\cap\Q$ where $d=x=0.\gamma_0\ldots\gamma_{m-1}$; therefore $(p^mg(x))\md1\in\Z_p\cap\Q$
as $g_d(0)=(p^mg(d))\md1$. Note that if $(n+1)p^{-m}\not\in[a,b]$ then we apply Corollary
\ref{cor:ext} rather than Proposition \ref{prop:ext} and use $g_{d,M}$ instead of $g_d$ and $M$ instead of $m$ here
and after.  

Take $\ell\in\N$; then
by differentiability of $g$, for all $0\le h<1$  and all sufficiently large $\ell\in\N_0$ we can represent $g(x+p^{-m-\ell}h)$ as
\begin{equation}
\label{eq:der}
g(x+p^{-m-\ell}h)=
g(x)+c(x)\cdot p^{-m-\ell}h+p^{-m-t(\ell)}\theta(\ell,h),
\end{equation}
where $c(x)=g^\prime(x)$, $|\theta(\ell,h)|\le 1$ 
and $t$ is a map from $\N_0$ to $\N_0$ such
that $p^{t(\ell)}\to\infty$ faster than $p^{\ell}\to\infty$ while $\ell\to\infty$. That is, for all sufficiently
large $\ell$ we may represent $t(\ell)$ as
$t(\ell)=\ell+w(\ell)$, where $w$ is a map from $\N_0$ to $\N_0$ such that
$w(\ell)\to\infty$ as $\ell\to\infty$.

%
%

Further, by Proposition \ref{prop:ext}, the function $\tilde g(y)=(p^mg(x+p^{-m}y))\md1$
is continuous on $[0,1]$ and $\mathbf G_{[0,1]}(\tilde g)\subset\mathbf{P}(\mathfrak
A)$. From here  by combining Proposition \ref{prop:comp}
and Theorem \ref{thm:mult-add-p} we conclude that there exists a finite automaton
$\mathfrak C$ 
such that the graph $\mathbf G_{[0,1]}(\bar g)$ of the function 
$$\bar g(y)=(p^mg(x+p^{-m}y)-p^mg(x))\md1=((p^mg(x+p^{-m}y))\md1-(p^mg(x))\md1)\md1$$
lies
completely in  $\mathbf P({\mathfrak C})$. 

Indeed, as $(p^mg(x))\md1\in\Z_p\cap\Q$
then by \eqref{eq:thm:mult-add-p} the graph of the continuous function $y\mapsto
y-(p^mg(x))\md1$ on $[0,1]$ lies completely in  $\mathbf{LP}(\mathfrak B)\subset\mathbf{P}(\mathfrak
B)$ for  a finite automaton $\mathfrak B$ whose automaton function is $f_{\mathfrak
B}(z)=z-(p^mg(x))\md1$, $(z\in\Z_p)$, and therefore the composite function
$\bar g(y)$ is finitely computable on $[0,1]$, cf. 
Proposition \ref{prop:comp}. We proceed
with this in mind.

We  see from \eqref{eq:der}  that for all  sufficiently large $\ell$ and all $h$
the following is true:
\begin{equation}
\label{eq:der-m}
\bar g(p^{-\ell}h)=(p^mg(x+p^{-m-\ell}h)-
p^mg(x)
)\md1=c(x)\cdot p^{-\ell}h+p^{-t(\ell)}\theta(\ell,h).
\end{equation}


%


Now for the rest of the proof we take (and fix) $h=p^{-1}=0.1$. 
Let $0.\alpha_s\alpha_{s+1}\ldots$
be a base-$p$ expansion of $(c(x))\md1$;  so  $c(x)=\alpha_0\ldots\alpha_{s-1}.\alpha_{s}\alpha_{s+1}\ldots$
is a base-$p$ expansion of $c(x)$. We may assume that $c(x)\ge 0$ since if
otherwise we
consider the function $(-g)\md1$ which  satisfies conditions of the lemma
as $g$ satisfies these conditions.
If there exists two different base-$p$ expansions for $(c(x))\md1$
we will consider only one of these. Recall that
these expansions are  of the form $0.\zeta_1\ldots\zeta_n0^{\infty}$
and $0.\zeta_1\ldots\zeta_{n-1}\zeta_n^\prime(p-1)^{\infty}$ where $\zeta_1,\ldots,\zeta_n\in\{0,1,\ldots,p-1\}$,
$\zeta_n\ne 0$ and $\zeta_n^\prime=\zeta_n-1$. Now, if the function $\theta(\ell,p^{-1})$
is non-negative for an infinite number of $\ell\in\N$, then we take the first
of the base-$p$ expansions; and we take the second one in the opposite case.

We claim that in all cases mentioned above there exists a strictly
increasing sequence $\Cal L$ of $\ell\in\N$  such that, speaking loosely, the term $p^{-t(\ell)}\theta(\ell,h)$ has no affect on
higher
order digits of  the base-$p$ expansion of the right-hand part of \eqref{eq:der-m}.
In the case when $(c(x))\md1$ admits only one base-$p$ expansion this follows
from the fact that  $p^{-t(\ell)}$ tends to 0 faster than $p^{-\ell}$ as
we may take for $\Cal L$ all sufficiently large $\ell$.
In the case when $(c(x))\md1$
admits two base-$p$ expansions  the claim is also true since we 
consider a right
base-$p$ expansion  $(c(x))\md1=0.\zeta_1\ldots\zeta_n0^{\infty}$ and assume that the function $\theta(\ell,p^{-1})$
in \eqref{eq:der-m}
is non-negative for an infinite number of $\ell\in\N$: In that case we take for $\Cal
L$ all sufficiently large $\ell$ such that $\theta(\ell,p^{-1})\ge 0$.
When $(c(x))\md1$ admits two base-$p$ expansions and the function $\theta(\ell,p^{-1})$
is non-negative only for a finite number of $\ell\in\N$, we
consider a left base-$p$ expansion $(c(x))\md1=0.\zeta_1\ldots\zeta_{n-1}\zeta_n^\prime(p-1)^{\infty}$ where $\zeta_1,\ldots,\zeta_n\in\{0,1,\ldots,p-1\}$,
$\zeta_n\ne 0$ and $\zeta_n^\prime=\zeta_n-1$. Then there
exists infinitely many $\ell\in\N_0$ such that $\theta(\ell,p^-1)\le 0$,
and we take for $\Cal L$  all these sufficiently large
$\ell$. 

In other words, if we take $\ell\in\Cal L$, substitute $y=p^{-\ell-1}$ to $\bar g(y)$ and
apply \eqref{eq:der-m} then we  get 
\begin{equation}
\label{eq:right}
\bar g(0.(0)^{\ell}1(0)^\infty)=0.\underbrace{0\ldots0}_{\ell-s+1}\alpha_0\ldots\alpha_{t_1(\ell)}\delta_{t_1(\ell)+1}\delta_{t_1(\ell)+2}\ldots,
\end{equation}
where $\delta_{j}\in\{0,1,\ldots,p-1\}$ for $j\ge t_1(\ell)+1$,  $t_1(\ell)=-\ell+s+t(\ell)=s+w(\ell)$
(note that $\delta_{j}$ depends on $\ell$). 
Further,  by 
Note \ref{note:mark-up}  we
conclude now that given  a right-infinite word $u$ and  $\ell\in\N$ there exists an 
$s$-mark-up of the word $0^\ell u$ w.r.t. the function $\bar g$
where  $s$ is a suitable (depending on $u$ and
$\ell$) state of a finite automaton $\mathfrak C$ which is a sequential
composition of the automaton $\mathfrak A$ with the automaton $\mathfrak
B$ 
This means in particular that given an infinite word 
$v(\ell)=0^\ell100\ldots$, for any  $\ell\in\Cal L$ 
there exists a mark-up 
(w.r.t. a suitable state $s=s(v(\ell))$ of the automaton $\mathfrak
C$, cf. Proposition \ref{prop:mark-up}) 
 $i_0(\ell),i_1(\ell),i_2(\ell),\ldots$ 
of the word $v(\ell)$. As a total number of
states of the automaton $\mathfrak C$ is finite, at least one state, say
$s^\prime$, in
the sequence $(s(v(\ell))\:\ell\in\Cal L)$ occurs infinitely many times. Denote $\mathfrak
C^\prime=\mathfrak C(s^\prime)$ (then $\mathfrak C^\prime$ is a finite automaton as well) and consider
an infinite strictly increasing sequence 
$\Cal L^\prime=(\ell^\prime\:s(v(\ell^\prime))=s^\prime;\ell^\prime\in\Cal
L)$.

Given a term $\ell^\prime$ of the latter sequence $\Cal L^\prime$ take the smallest $k\in\N_0$ such that  $i_k(\ell^\prime)>t(\ell^\prime)$; denote via $s(\ell^\prime)$
the state the automaton $\mathfrak C^\prime$ reaches after being feeded by the word
$0^{i_k(\ell^\prime)-t(\ell^\prime)}$. As the number of states of the automaton
$\mathfrak C^\prime$ is finite, in
the infinite sequence $(s(\ell^\prime))$ at least one term, say $\bar s$, occurs infinitely
many times. Consider an infinite sequence $(\ell_j^\prime)_{j=0}^\infty$ such that $s(\ell_j^\prime)=\bar s$ and consider an automaton $\mathfrak C^\prime(\bar s)$ (whence the latter automaton is finite
also). 
If the automaton $\mathfrak C^\prime(\bar s)$ is being feeded
by the word $0^{\ell_j^\prime}10^{w(\ell^\prime_j)}$ then the automaton outputs the word 
$q_j=\sigma_{0}^{(j)}\ldots\sigma_{t(\ell^\prime_j)}^{(j)}$; therefore being
feeded by the word $0^{w(\ell^\prime_j)}$ the automaton outputs the word
$q_j^\prime=\sigma_{\ell^\prime_j-s+1}^{(j)}\ldots\sigma_{t(\ell^\prime_j)}^{(j)}$
of length $L_j=w(\ell_j^\prime)-s$. As $w(\ell)\to\infty$ while $\ell\to\infty$
we may assume without loss of generality that the sequence $(L_j)$ is strictly
increasing (since if otherwise we consider a subsequence $(j_i)_{i=0}^\infty$  of the sequence $(j)$ such that the sequence $(L_{j_i})_{i=0}^\infty$).
Now by mimic the proof of Proposition \ref{prop:auto-Q} we show that the
sequence $(0.q_j^\prime)$ has only finitely many limit points and all these
limit points are in $\Z_p\cap\Q$. But from   \eqref{eq:right}
it follows that $\lim_{j\to\infty}0.q_j^\prime=(c(x))\md1$; therefore 
$(c(x))\md1\in\Z_p\cap\Q$ and thus $c(x)=g^\prime(x)\in\Z_p\cap\Q$.
\end{proof}
%
%

\begin{proof}[Proof of Lemma \ref{le:der2}]
Let $x$ be as in the statement of Lemma \ref{le:der}; i.e.,
let the right
base-$p$ expansion of $x$ be as in the proof of  Lemma \ref{le:der}. 
Then   $g^\prime(x)\in\Q\cap\Z_p$ by Lemma \ref{le:der}.

Consider the function $g_1(u)=g(u)-g^\prime(x)\cdot u+c$ of argument $u\in[a,b]$
where $c\in\Q\cap\Z_p$;
then $g_1$ is two times differentiable on $[a,b]$ and $g_1^\prime(x)=0$,
$g_1^{\prime\prime}(u)=g^{\prime\prime}(u)$ for all $u\in[a,b]$.
As $g_1$ is continuous on $[a,b]$, the constant $c$ may be taken so that
$g_1(u)\in[0,1]$ for all $u$ from a sufficiently small  closed  neighborhood $[a_1,b_1]$ of $x$. 

We are going to prove that $g_1^{\prime\prime}(x)=0$ for all $x\in[a,b]$.
Note that if $g_1^{\prime\prime}(x)\ne 0$ for some $x\in[a,b]$ then 
$g_1^{\prime\prime}$ does not change its
sign  on a  sufficiently small neighborhood $(a_2,b_2)\subset[a_1,b_1]$ of $x$. Indeed, if not, then there exist two infinite sequences, $(\check x_i)_{i=0}^\infty$
and $(\hat x_i)_{i=0}^\infty$ such that all the terms of either sequence
are pairwise distinct, $\lim_{i\to\infty}\check x_i=\lim_{i\to\infty}\hat x_i=x$, and $g_1^{\prime\prime}(\check x_i)\ge0$, $g_1^{\prime\prime}(\hat x_i)\le0$ for all $i\in\N_0$. But as $\lim_{i\to\infty}
g_1^{\prime\prime}(\check x_i)=\lim_{i\to\infty}g_1^{\prime\prime}(\hat x_i)=g_1^{\prime\prime}(x)$
(since $g_1^{\prime\prime}$ is continuous at $x$) then necessarily $g_1^{\prime\prime}(x)=0$;
but this contradicts our assumption that $g_1^{\prime\prime}(x)\ne 0$.
We therefore may assume that $g_1^{\prime\prime}(u)\ge 0$ for all $u\in[a_2,b_2]$; otherwise
consider the function $1-g_1$ rather than $g_1$.  

%
Finally by
Corollary
\ref{cor:sum} we conclude that $g_1$ is finitely computable on a sufficiently
small closed neighbourhood $U\subset[a,b]$ of $x$.
Further we use $g$ for $g_1$
and $[a,b]$ for $U$
without risk of misunderstanding. Thus we have:
\begin{enumerate}
\item $g$ is finitely computable on $[a,b]\ni x$; 
\item $g$ is two times differentiable on $[a,b]$;
\item $g^{\prime\prime}$ is continuous on $[a,b]$;
\item $g^{\prime\prime}\ge 0$ on $[a,b]$;
\item $g^{\prime}(x)=0$;
\end{enumerate}

Now, since $g$ is two times differentiable on $[a,b]$, 
for all $0\le h\le1$  and all sufficiently large $\ell\in\N_0$ we can represent $g(x+p^{-m-\ell-1}h)$ as
\begin{equation}
\label{eq:der2}
g(x+p^{-m-\ell}h)=g(x)+
C(x)
\cdot p^{-2m-2\ell}h^2
+p^{-2m-t(\ell)}\theta(\ell,h),
\end{equation}
where $C(x)$ stands for $\frac{g^{\prime\prime}(x)}{2}$, $|\theta(\ell,h)|\le 1$, 
$t(\ell)=2\ell+w(\ell)$,  and $w$ is a map from $\N_0$ to $\N_0$ such that
$w(\ell)\to\infty$ as $\ell\to\infty$.

\underline{Claim 1:} $C(x)\in\Z_p\cap\Q$. We  prove that  by mimic of the respective part of the proof of Lemma \ref{le:der}. 
%
Firstly we show that $(p^mg(x))\md1
\in\Z_p\cap\Q$ 
as in the proof of Lemma \ref{le:der}; thus considering a function $\bar g(y)=(p^mg(x+p^{-m}y)-p^mg(x))\md1$
we see that
\begin{equation}
\label{eq:der-m-2}
\bar g(p^{-\ell}h)=(p^mg(x+p^{-m-\ell}h)-p^mg(x)
)\md1=C(x)\cdot p^{-m-2\ell}h^2+p^{-m-t(\ell)}\theta(\ell,h)
\end{equation}
for all sufficiently large $\ell$.

Let $0.\alpha_s\alpha_{s+1}\ldots$
be a base-$p$ expansion of $(C(x))\md1$ (we may take either of the expansions if
there exist two different ones); so  $C(x)=\alpha_0\ldots\alpha_{s-1}.\alpha_{s}\alpha_{s+1}\ldots$
is a base-$p$ expansion of $C(x)$.  Take $h=1$; then for all sufficiently
large $\ell$ the base-$p$ expansion of the right-hand part of \eqref{eq:der-m-2}
is of the form 
\begin{equation}
\label{eq:right2}
0.\underbrace{0\ldots0}_{m+2\ell-s}\alpha_0\ldots\alpha_{t_1(\ell)}\delta_{t_1(\ell)+1}\delta_{t_1(\ell)+2}\ldots
\end{equation}
where $\delta_{j}\in\{0,1,\ldots,p-1\}$ for 
$j\ge t_1(\ell)+1=t(\ell)-m-2\ell+s=w(\ell)-m+s$
depend on $\ell$. 

The function $(p^mg(x+p^{-m}z)-p^mg(x))\md1$ of argument $z$ is continuous and finitely computable on $[0,1]$ by 
a finite automaton $\mathfrak C$.  Now considering an infinite word
$0^\ell10^\infty$  with the corresponding mark-up we prove in the same way as in Lemma \ref{le:der} that the corresponding
sequence of finite output words of the automaton $\mathfrak C$ is a sequence of initial
finite sub-words of the infinite word   $0^{m+2\ell-s}\alpha_0\alpha_1\ldots$
and then deduce as in the proof of Lemma \ref{le:der} that $C(x)\in\Z_p\cap\Q$
(note that given $x=np^{-m}$ we may always take $m$ so that $m+2\ell-s>\ell$
without altering the value of $x$ just by multiplying both numerator and
denominator by a suitable power
of $p$). 

\underline{Claim 2:} Now we prove that $C(x)=0$.
Assume that $C(x)\ne0$; that is, that $g^{\prime\prime}(x)\ne0$. Let $p^c=|C(x)|_p$,
$c\in\Z$, be a $p$-adic absolute value of $C(x)$;
therefore $C(x)=qp^{-c}$, where $q\in\Z_p$, $q$ is a unity of $\Z_p$, cf. Subsection \ref{ssec:p-adic}. By
Claim 1, $C(x)\in\Z_p\cap\Q$, so necessarily $c\le 0$ and $q\in\Q$; whence
$q\in\Z_p\cap\Q$ and thus  $q^{-1}\in\Z_p\cap\Q$ as $q$ is a unity. Therefore
the function $\check g(y)=(q^{-1}(p^mg(x+p^{-m}y)+p^mg(x))\md1)\md 1$ of argument $y\in[0,1]$
is finitely computable
on $[0,1]$ (say, by the automaton $\mathfrak B$):
This follows from Propositions \ref{prop:ext}, \ref{prop:comp}, and Corollary
\ref{cor:sum}.
Now by \eqref{eq:der-m-2} we conclude that

%
%
\begin{equation}
\label{eq:der2-1}
(q^{-1}(p^mg(x+p^{-m-\ell}h)+p^mg(x))\md1)\md 1=
p^{-k
-2\ell
}h^2
+p^{-m-t_1(\ell)}\theta_1(\ell,h),
\end{equation}
where $k=c+m\ge0$ (we may assume that the inequality is true just by taking $m$ sufficiently large by multiplying both numerator and denominator by a
suitable power of $p$ and thus
without altering the value of $x$), 
$|\theta_1(\ell,h)|\le 1$, 
$t_1(\ell)=2\ell+w_1(\ell)$,  and $w_1$ is a map from $\N_0$ to $\N_0$ such that
$w_1(\ell)\to\infty$ as $\ell\to\infty$. 

Further, given $n\in \N$ and a word $v_n=\chi_0\ldots\chi_{n-1}\in\Cal W$
and taking $y=p^{-\ell}h$ with $h=0.v_n0^\infty$ we have that
$(p^{-\ell}h;\check g(p^{-\ell}h))\in\mathbf{P}(\mathfrak B)$. Let
$i_0(\ell,v_n)<i_1(\ell,v_n)<i_2(\ell,v_n)<\ldots$   be 
corresponding mark-up of the
infinite 
word $0^\ell v_n0^\infty$ (the latter infinite word corresponds to $y=p^{-\ell}h$
once $h=0.v_n0^\infty$).
Take $r$  the smallest so that $i_r=i_r(\ell,v_n)>t_1(\ell)+m$;  denote 
$s(\ell,v_n)$  the  state the automaton $\mathfrak B$ reaches after
being feeded by the word $0^{i_r(\ell,v_n)-k-2\ell-2n}$.  
As the number of states of the
automaton $\mathfrak B$ is finite, in the sequence $(s(\ell,v_n))_{\ell=1}^\infty$
at least one state, say $\bar s(v_n)$, occurs infinitely many times. Consider
a strictly increasing sequence $\Cal
L=(\ell_j\in\N\:s(\ell_j,v_n)=\bar s(v_n))_{j=0}^\infty$; then
 once $w_1(\ell_j)>2n+r$, the automaton $\mathfrak B(s(\ell_j,v_n))=\mathfrak
 B(\bar s(v_n))$, being
feeded by the word $0^{\ell_j}v_n0^{k+\ell_j+n}$, outputs the word
$
\zeta_0^{(j)}\ldots\zeta_{k+2\ell_j-1}^{(j)}\xi_0^{(j)}\ldots\xi_{2n-1}^{(j)}$.
From \eqref{eq:der2-1}
it follows  that $(\lim_{j\to\infty}
(p^{k+2\ell_j}\cdot0.\zeta_0^{(j)}\ldots\zeta_{k+2\ell_j-1}^{(j)}
\xi_0^{(j)}\ldots\xi_{2n-1}^{(j)}))\md1=h^2$.
Therefore $(\lim_{j\to\infty}
0.\xi_0^{(j)}\ldots\xi_{2n-1}^{(j)})\md1=h^2=0.\xi_0\ldots\xi_{2n-1}$ where $$
\xi_0p^{2n-1}+\xi_1p^{2n-2}+\cdots+\xi_{2n-1}=
\left(\chi_0p^{n-1}+\chi_1p^{n-2}+\cdots+\chi_{n-1}\right)^2
$$
(in other words,
$\xi_0\xi_1\ldots\xi_{2n-1}$ is a base-$p$ expansion of the square of the
number whose base-$p$ expansion is $v_n=\chi_0\ldots\chi_{n-1}$). Thus necessarily
$\xi_0^{(j)}=\xi_0,\ldots,\xi_{2n-1}^{(j)}=\xi_{2n-1}$ for all sufficiently
large $j$. But  $(p^{-\ell}h;\check g(p^{-\ell}h))\in\mathbf{P}(\mathfrak B)$ for all $\ell$; thus by Proposition \ref{prop:p-shift},  
$((p^{k+\ell_j}h)\md1;(p^{k+2\ell_j}\check g(p^{-\ell_j}h))\md1\in\mathbf{P}(\mathfrak B)$ for all $j\in\N_0$. Therefore for all sufficiently large $j$ (such that
$k+\ell_j>n$) we have
that $(0;0.\xi_0\ldots\xi_{2n-1})\in\mathbf{P}(\mathfrak B)$ since $h=0.v_n0^\infty$.
In other words, as the automaton $\mathfrak B(s(\ell_j,v_n))=\mathfrak
 B(\bar s(v_n))$, being
feeded by the word $0^{\ell_j}v_n0^{k+\ell_j+n}$, outputs the word
$
\zeta_0^{(j)}\ldots\zeta_{k+2\ell_j-1}^{(j)}\xi_0\ldots\xi_{2n-1}$ once $j$
is sufficiently large, $k+\ell_j>n$, the automaton  $\mathfrak B(\bar s(v_n))$, being feeded by the zero word $0^{2n}$, outputs the word $\xi_0\xi_1\ldots\xi_{2n-1}$.
This
means that, given an arbitrary number $N\in\N_0$ whose base-$p$ expansion
(where higher order digits might be 0) is of length $n$, $n$ being sufficiently large, there exists
a state $s$ of the finite automaton $\mathfrak B$ such that the automaton
$\mathfrak B(s)$, being feeded by a zero sequence of length $2n$, outputs a
word (of length $2n$) that is a base-$p$ expansion of $N^2$.
But this is not possible since it is well known that  squaring is not possible by a finite automaton (cf.,e.g., \cite[Theorem
2.2.3]{Bra} or \cite[Proposition 7.1.6]{AlgCombinWords});
however, a short proof follows.

%
%
%
%

As the automaton $\mathfrak B$
is finite, then there are only finitely many sub-automata  $\mathfrak B(s(\ell,v_n))$.
But any finite automaton, being feeded by a sufficiently long zero word $0^L$
outputs the word of the form $u_1(u_2)^M u_3$, where $M=M(L)\in\N$,
$u_2\in\Cal W$,
$u_1,u_3\in\Cal W_\phi$, and the words $u_1,u_2,u_3$ are completely determined
by the finite automaton, $u_1$ is a right prefix of $u_2$, cf. Lemma
\ref{le:per}. 
But given finitely many words $u_{1.i},u_{2.i},u_{3.i}$
of that sort, $i=0,1,2,\ldots,K$, there exist infinitely many words $\xi_0\xi_1\ldots\xi_{2n-1}$  which are base-$p$ expansions  of squares of numbers from $\N_0$ 
and which are not of the form $u_{1.i} u_{2.i}^M u_{3.i}$, $i=1,2,\ldots,K$,
$M=1,2,\ldots$. 

The  contradiction proves that $C(x)=0$; therefore,  
$g^{\prime\prime}(x)=0$.

 
\end{proof}

\begin{proof}[Proof of Theorem \ref{thm:main}]
We already
have proved that 
$g^{\prime\prime}(x)=0$ if
$x=n p^{-m}\in[a,b)$, where
$n\in\N_0$, $m\ge\lfloor\log_pn\rfloor+1$. But the set of all these $x$ is
dense in $[a,b]$, so, as the second 
derivative $g^{\prime\prime}$ is
is continuous
on $[a,b]$ by the condition of the theorem under proof, 
$g^{\prime\prime}$ must vanish
everywhere on $[a,b]$; therefore, $g^\prime=const$. But this implies that 
there exist $A,B\in\R$ such
that $g(x)=Ax+B$ for all $x\in[a,b]$ and $g^\prime(x)=A$ for all $x\in[a,b]$.
From Lemma \ref{le:der} it follows now that $A\in\Z_p\cap \Q$. Now taking
an arbitrary number $y\in(a,b)\cap\Z_p\cap\Q$ we see that $g(y)\in\Z_p\cap\Q$ by Proposition
\ref{prop:auto-Q}; 
hence $g(y)-Ay=B$ must
be also in $\Z_p\cap\Q$ as $Ay\in\Z_p\cap\Q$. 

Now we will prove that $\mathbf C(A,B)\subset\mathbf{LP}(\mathfrak A)$. To
begin with,
we note that by Theorem \ref{thm:eval-erg} there exists a minimal sub-automaton
$\mathfrak A^\prime$ and a segment $[a^\prime,b^\prime]\subset[a,b]$ such
that $\mathbf G_{[a^\prime,b^\prime]}(g)\subset\mathbf{LP}(\mathfrak A^\prime)$.
Taking $d\in(a^\prime,b^\prime)$ as in the statement of Proposition \ref{prop:ext},
we conclude that the graph $\mathbf G_{[0,1]}(g_d)$ of the function $g_d(x)=(Ax
+An
+p^mB)\md1$
 on $[0,1]$ lies completely in $\mathbf P(\mathfrak A^\prime)$; thus  $\mathbf G_{[0,1]}(g_d)=\{(x;g_d(x))\:x\in[0,1]\}\subset
\mathbf{LP}(\mathfrak A^\prime)$ by Corollary \ref{cor:mark-up}. 
$An\in\Z$; that is, $p^ma\le n<p^mb$. 
As $A\in\Z_p\cap\Q$ then $A=P/Q$ for
suitable $P\in\Z$, $Q\in\N$. 
Now given arbitrary $R\in\{0,1,\ldots,Q-1\}$ we take $n$ and $m$ so that
$d=p^{-m}n$ satisfies conditions of Proposition \ref{prop:ext} (that is, $p^ma\le n<p^mb$) and $n=LQ+R\in\{0,1,\ldots,p^m-1\}$ for a suitable $L\in\N_0$
and conclude that
$$
\{(x;g_d(x))\:x\in[0,1]\}=\{(x;(Ax+AR+p^mB)\md1)\:x\in[0,1]\} \subset
\mathbf{LP}(\mathfrak A^\prime)
$$ 
Given arbitrary $R\in\{0,1,\ldots,Q-1\}$, the above inclusion holds for all sufficiently large $m$; therefore due the structure of $\mathbf C (B)$ (cf. Subsection
\ref{ssec:const}) the following inclusion holds for every $R\in\{0,1,\ldots,Q-1\}$
and every $B^\prime\in\mathbf C (B)$:
$$
\{(x;(Ax+AR+B^\prime)\md1)\:x\in[0,1]\} \subset\mathbf{LP}(\mathfrak A^\prime).
$$
But $\cup_{R=0}^{Q-1}\{(x;(Ax+AR+B^\prime)\md1)\:x\in[0,1]\}=\{((x\md1;(Ax+B^\prime)\md1)\:x\in\R\}
=\mathbf C(A,B^\prime)$; therefore we have shown that $\mathbf C(A,B^\prime)\subset\mathbf{LP}(\mathfrak A^\prime)$ for all $B^\prime\in\mathbf C(B\md1)$. That is,  $\mathbf{LP}(\mathfrak A^\prime)$ contains the whole link of cables $\mathbf C(A,B^\prime)$ for all $B^\prime\in\mathbf C(B)$ (i.e., contains $\mathbf{LP}(F)$ where $F\:z\mapsto
Az+B$, $z\in\Z_p$, cf. Theorem \ref{thm:mult-add-p}) and
$\mathbf G_{[a,b]}(g)$ lies completely in a suitable cable
of the link.
This proves the first
claim of Theorem \ref{thm:main} since $\mathbf{LP}(\mathfrak A^\prime)
\subset\mathbf{LP}(\mathfrak A)$, cf. Note \ref{note:sub-auto}.

To prove the second claim, given a finite automaton $\mathfrak
A$ consider all cables $\mathbf C(A,B)$ 
such that $\mathbf C(A,B)=\{(y\md1;(Ay+B)\md1)\colon
y\in\R\}\subset\mathbf{P}(\mathfrak A)$; whence by the first claim of the theorem all these cables lie in $\mathbf{LP}(\mathfrak A)$. Moreover, as
we have shown during the proof of the first claim of the theorem, for
either
of the cables $\mathbf C(A,B)$ there exists a minimal sub-automaton $\mathfrak A^\prime_{A,B}$ of the automaton $\mathfrak A$
such that $\mathbf C(A,B)\subset\mathfrak A^\prime_{A,B}$.
The cables cross  zero meridian $\mathbf O=\{(0;t\md1)\:t\in\R\}\subset\T^2$
of the torus $\T^2$ only when $y\md1=0$; therefore the point set $\mathbf
S$ of all
the points where the cables
cross   zero meridian consists of the points of the form $(0;e)$ where $e\in\mathfrak a(0)$ and $\mathbf S$ contains all the points of the
form $(0;B\md1)$ where $B$ are constant terms of the cables. As  $\mathfrak
a(0)$ is a finite set (cf. Proposition \ref{prop:auto-Q}), there are no more
than a finite number of pairwise distinct numbers $B\md1$ (note that cables with
equal slopes whose
constant terms are congruent modulo 1 coincide). Now taking $y\in\Z$ we see
that all the points of the form $(0;(Ay+B)\md1)$ of the cables belong to zero
meridian and therefore to the finite set $\{(0;r)\:r\in\mathfrak a(0)\}$; hence, there exist no more
than a finite number of pairwise distinct numbers $Ay\md 1$ where $y$ ranges over
rational integers $\Z$ and $A$ are slopes of the cables from $\mathbf{P}(\mathfrak
A)$. Thus if there exists an infinite number of cables in $\mathbf{P}(\mathfrak A)$ then there exists a minimal sub-automaton $\mathfrak A^\prime$ of the
automaton $\mathfrak A$ such that $\mathbf {LP}(\mathfrak A^\prime)$ contains 
an infinite number of cables of the form $\mathbf C(AC,B)$ with $A,B$
fixed and $C$ ranging through an infinite subset $\mathbf C$ of $\Z$ so that $AC\md1$
are all equal one to another. Therefore for the rest of the proof we may
(and will)
assume that the automaton $\mathfrak A$ is minimal. 

By the first claim of the theorem, $A\in\Z_p\cap\Q$;
so there exists a unique representation of $A$ in the form
$$
A=c+\frac{d}{p^t-1}
$$
where $t\in\N$ is a period length of $A$,  $c\in\Z$, and $d\in\{0,1,\ldots,p^t-2\}$,
cf. Proposition \ref{prop:p-repr} and Note \ref{note:pq-frac-rep}. Therefore
 $\mathbf C$ must contain an infinite subset of numbers from the coset
$q+(p^t-1)\cdot\Z$ for a suitable $q\in\{0,1,\ldots,p^t-2\}$
since $AC_1\equiv AC_2\pmod1$
implies $A(C_1-C_2)\equiv 0\pmod 1$, i.e., $A(C_1-C_2)\in\Z$.  Thus from
the assumption that there are infinitely many pairwise distinct
cables in $\mathbf{P}(\mathfrak A)$ it follows that then in $\mathbf{LP}(\mathfrak A)$ there exist infinitely many cables of the form $\mathbf C(D+E,B)$
with $B,E$ fixed ($B\in\Z_p\cap\Q$, $E\in\Z_p\cap\Q\cap[0,1)$) and $D$ running
through an infinite subset $\mathbf D\subset\Z$. 
By considering $-f_{\mathfrak A}$ (and the
corresponding finite automaton) if necessary
we may assume 
that $\mathbf D$ is an infinite subset of $\N$. 
Therefore,  $\mathbf D$ constitutes a strictly increasing sequence $(D_i)_{i=0}^\infty$
of natural numbers.
Now take arbitrary $u\in[0,1)$ and consider a sequence $x_i=uD_i^{-1}$.
As  the sequence $(D_i)$ is strictly increasing, $\lim_{i\to\infty}x_i=0$;
therefore $\lim_{i\to\infty}(x_i;(D_ix_i+Ex_i+B)\md1)=(0;(u+B)\md1)
\in\mathbf{LP}(\mathfrak A)$ as  
$(x_i;(u+Ex_i+B)\md1)\in\mathbf C(D_i+E,B)\subset\mathbf{LP}(\mathfrak A)$ and $\mathbf{LP}(\mathfrak A)$ is closed
in $\T^2$, cf. Corollary \ref{cor:AP=LP}. Thus we have proved that zero meridian $\mathbf O=\{(0;y)\:y\in[0,1)\}$
of the torus $\T^2$ lies completely in $\mathbf{LP}(\mathfrak A)$.

On the other hand, if $(0;y)\in\mathbf{LP}(\mathfrak A)$ then $y\in\mathfrak
a(0)$ by definitions of  $\mathbf{LP}(\mathfrak A)$ and $\mathfrak a(0)$,
see Subsections \ref{ssec:plots} and \ref{ssec:mark-up}; but there are only finitely many points
in $\mathfrak a(0)$ by Proposition \ref{prop:auto-Q}. The contradiction proves
the second claim of the theorem.


%

\end{proof}
\subsection{The multivariate case}
\label{ssec:main-mult}
In this subsection we are going to extend Theorem \ref{thm:main} for the
case of finite automata with multiply inputs/outputs. Note that actually
an automaton over alphabet $\F_p=\{0,1,\ldots, p-1\}$ with $m$ inputs and
$n$ outputs can be considered as a letter-to-letter transducer with a single input over the
alphabet
$\{0,1,\ldots,p^m-1\}$ and a single output over the alphabet $\{0,1,\ldots,p^n-1\}$;
therefore 
the plot of that automaton is a closed subset of the unit square $\mathbb I^2$. We
however are going to consider plots of automata of that sort as subsets of multidimensional
unit hypercube $\mathbb I^{m+n}$. Therefore automata functions of such automata
are 1-Lipschitz mappings from $\Z_p^m$ to $\Z_p^n$, see Subsection \ref{ssec:a-map};
and vice versa, every 1-Lipschitz mapping from $F\:\Z_p^m\to\Z_p^n$ is an
automaton function of a suitable automaton $\mathfrak A$ with $m$ inputs and $n$ outputs
over the alphabet $\F_p$. Note that $F=(F_1;\ldots;F_m)$ where $F_k\:\Z_p^m\>\Z_p$
($k=1,2,\ldots,m$)
is 1-Lipschitz and therefore is an automaton function of an automaton with $m$ inputs and
a single output.

Now we re-state
our definition of a (limit) plot for that case of automata with $m$ inputs
and $n$ outputs.

\begin{defn}[Automata plots, the multivariate case]
\label{def:plot-auto-mult}
Given an automaton function $F=F_{\mathfrak A}\:\Z_p^m\>\Z_p^n$ define a  set
$\mathbf P(F_{\mathfrak A})$ of points of
 $\R^{n+m}$ 
as follows: For $k=1,2,\ldots$ denote
\begin{equation}
\label{eq:plot-mult}
E_k(F)=\left\{
\left({\frac{{\mathbf z\md p^k}
}{p^k};\frac{{F(\mathbf z)\md p^k}
}{p^k}
}\right)\in\mathbb I^{m+n}\: \mathbf z\in \Z_p^m\right\}
\end{equation} 
a point set in a unit real hypercube $\mathbb I^{m+n}$; here given $\mathbf
y=(y_1;\ldots;y_q)\in\Z_p^q$ we put
$$
\frac{\mathbf y\md p^k}{p^k}=\left(\frac{y_1\md p^k}{p^k};\ldots;\frac{y_q\md p^k}{p^k}\right)\in(\Z/p^k\Z)^q.
$$
Then take a
union $E(F)=\cup_{k=1}^\infty E_k(f)$ and denote via $\mathbf
P(F)=\mathbf P(\mathfrak A)$  a closure (in topology of $\R^{m+n}$) of the set $E(F)$.

Given an 
automaton $\mathfrak A$,
we call a \emph{plot of the automaton} $\mathfrak A$ 
the set $\mathbf P(\mathfrak A)$. We call a \emph{limit plot} 
of the automaton $\mathfrak A$
the point set $\mathbf{LP}(\mathfrak A)$ 
which is defined as follows: A point $(\mathbf x;\mathbf y)\in\R^{m+n}$ lies in $\mathbf{LP}(\mathfrak A)$ if and only if there exist $\mathbf z\in\Z_p^m$ and a strictly increasing
infinite sequence $k_1<k_2<\ldots$ of numbers from $\N$ such that simultaneously
\begin{equation}
\label{eq:def-LP-mult}
\lim_{i\to\infty}\frac{\mathbf z\md p^{k_i}}{p^{k_i}}=\mathbf x;\ \lim_{i\to\infty}\frac{F_\mathfrak
A(\mathbf z)\md p^{k_i}}{p^{k_i}}=\mathbf y.
\end{equation}  
\end{defn}
To put it in other words, at every step a letter-to-letter transducer $\mathfrak A$ (which has $m$
inputs and $n$ outputs over a $p$-symbol alphabet $\F_p$)  
\begin{itemize}
\item obtains
a vector $\mathbf a=(\alpha^{(1)};\ldots,\alpha^{(m)})\in\F_p^m$  (each $i$-th letter
$\alpha^{(i)}$ is sent accordingly to the $i$-th input of the automaton, $i=1,2,\ldots,m$), 
\item outputs a vector $\mathbf b=(\beta^{(1)};\ldots,\beta^{(n)})\in\F_p^n$  (each $j$-th output
of the automaton outputs accordingly the letter
$\beta^{(j)}$, $i=1,2,\ldots,n$) which depends both on the current state
and on the input vector $\mathbf a$,
\item reaches the next state (which depends both on $\mathbf
a$ and on the current state).
\end{itemize}
Then the routine repeats. Therefore after $k$ steps the automaton $\mathfrak
A$ transforms the input $m$-tuple $\mathbf w=(w_1;\ldots;w_m)$ of $k$-letter
words $w_i=\alpha_k^{(i)}\ldots\alpha_1^{(i)}$ ($i=1,2,\ldots,m$) into the
output $n$-tuple $\mathbf v=\mathfrak a(\mathbf w)=(v_1;\ldots;v_n)$ of $k$-letter
words $v_j=\mathfrak a^{(j)}(\mathbf w)=\beta_k^{(j)}\ldots\beta_1^{(j)}$ ($j=1,2,\ldots,n$). For $\mathbf w$ running over all $m$-tuples of $k$-letter
words, $k=1,2,\ldots$ we consider the set $E(\mathfrak A)$ of all points $(0.\mathbf w;0.\mathfrak a(\mathbf
w))\in\R^{m+n}$; here $0.\mathbf u$ stands for $(0.u_1;\ldots;0.u_\ell)$
where $u_1,\ldots,u_\ell$ are $k$-letter words. Then we define $\mathbf P(\mathfrak
A)$ as a closure in $\R^{m+n}$ of the set $E(\mathfrak A)$. Following the
lines of Note \ref{note:plot-auto-in} it can be shown that $\mathbf P(\mathfrak
A)=\mathbf P(F_{\mathfrak A})$. We stress that $\mathfrak A$ is a synchronous
letter-to-letter
transducer; that is why in the definition of the plot all $m$ input words
as well as corresponding $n$ output words of the automaton must have pairwise equal lengths.

Given a real function $G\colon D\>\R^n$ with the domain $D\subset \R^m$,
by the
\emph{graph 
of the function} (on the torus $\mathbb T^{m+n}$) we mean
the point subset $\mathbf G_D(g)=\{(\mathbf x\md1;G(\mathbf x)\md1)\colon \mathbf x\in D\}\subset\mathbb
T^{m+n}$. Note that if $\mathbf y=(y_1;\ldots;y_k)\in\R^k$ then $\mathbf
y\md1$ stands for $(y_1\md 1;\ldots;y_k\md 1)$.

\begin{thm}
\label{thm:main-mult}
Let  $\mathfrak A$ be a finite automaton over  the alphabet $\{0,1,\ldots,p-1\}$,
let $\mathfrak A$ have $m$ inputs and $n$ outputs, and let $G=(G_1;\ldots;G_n)\:[\mathbf a,\mathbf
b]= [a_1,b_1]\times\cdots\times[a_m,b_m]\>[0,1)^n$
\textup(where $[a_i,b_i]\subset[0,1)$, $G_i\:[\mathbf a,\mathbf b]\>[0,1)$,
$i=1,2,\ldots,m$\textup) be a two times differentiable
function such that all its second partial derivatives are continuous on $[\mathbf
a,\mathbf b]$.
If $\mathbf G(G)\subset \mathbf
{P}(\mathfrak A)\subset \T^{m+n}$ then there exist an $m\times n$ matrix $\mathbf A=(A_{ij})$
and a vector $\mathbf B=(B_1;\ldots; B_n)$ such that $A_{ij}\in\Q\cap\Z_p$,
$B_j\in\Q\cap\Z_p\cap[0,1)$ \textup($i=1,2,\ldots,m$;
$j=1,2,\ldots,n$\textup) 
and
$G(\mathbf x)=(\mathbf x\mathbf A+\mathbf B)\md1$ for all
$\mathbf x\in[\mathbf a,\mathbf b]$. There are not more than a
finitely many  $\mathbf A$ and $\mathbf B$  such that $A_{ij}\in\Q\cap\Z_p$,
$B_j\in\Q\cap\Z_p\cap[0,1)$ \textup($i=1,2,\ldots,m$;
$j=1,2,\ldots,n$\textup)  and $\mathbf G_{[\mathbf a,\mathbf b]}((\mathbf
x \mathbf
A+\mathbf B)\md1)\subset\mathbf P(\mathfrak
A)$ for some $[\mathbf a,\mathbf b]\subset[0,1)^m$; moreover, if 
$\mathbf G_{[\mathbf a,\mathbf b]}(\mathbf
x \mathbf
A+\mathbf B)\subset\mathbf P(\mathfrak
A)$ for some $[\mathbf a,\mathbf b]\subset[0,1)^m$ then $\mathbf G_{\R^m}((\mathbf
x \mathbf
A+\mathbf B)\md1)\subset\mathbf P(\mathfrak
A)\subset\T^{n+m}$. 

%
%
\end{thm}
\begin{proof}[Proof of Theorem \ref{thm:main-mult}]
Let $F_\mathfrak A=(F_1;\ldots;F_n)\:\Z_p^m\>\Z_p^n$ be automaton function of the automaton
$\mathfrak A$. 
Having  $i\in\{1,2,\ldots,m\}$
and $j\in\{1,2,\ldots,n\}$ fixed, take arbitrary numbers 
$z_k\in\Z_p\cap\Q\cap[a_k,b_k]$, $k=1,2,\ldots,i-1,i+1,\ldots,m$,
consider the map
$\bar F_{ij}(z)=F_j(z_1;\ldots;z_{i-1};z; z_{i+1};\ldots;z_m)$ and the function
$\bar G_{ij}(x)=G(z_1;\ldots;z_{i-1};x;z_{i+1};\ldots;z_m)$. 

As 
$z_k\in\Z_p\cap\Q\cap[a_k,b_k]$ and $[a_k,b_k]\subset[0,1)$ then the map $\bar F_{ij}\:\Z_p\>\Z_p$
is a finite automaton function: Actually the corresponding automaton $\bar{\mathfrak
A}_{ij}$ is a sequential
composition of the automaton $\mathfrak A$ with autonomous automata $\mathfrak
B_k$ 
which produce accordingly purely periodic output words
$\wrd(z_k)
\in\Cal W^\infty$
(cf. Corollary \ref{cor:r-p-repr-qz})
and feed accordingly $k$-th inputs ($k=1,\ldots,{i-1},{i+1},\ldots,m$) of the automaton
$\mathfrak A$ while the output of the automaton $\bar{\mathfrak
A}_{ij}$ is the $j$-th output of the automaton $\mathfrak A$. 

\underline{Claim:} We assert that 
$\mathbf G_{[a_i,b_i]} (\bar G_{ij})\subset\mathbf P(\bar{\mathfrak
A}_{ij})$. 

To prove the claim, firstly note that by Corollary \ref{cor:r-p-repr-qz}, for every $k=1,\ldots,{i-1},{i+1},\ldots,m$
we have that $z_k=0.(\zeta_{T_k-1}^{(k)}\zeta_{T_k-2}^{(k)}\ldots\zeta_{0}^{(k)})^\infty$
where $T_k$ is a period length of $z_k$ (see Subsection \ref{ssec:p-adic}).
Let $T$  be the least common multiple of all $t_i$; then 
$z_k=0.(\eta_{T-1}^{(k)}\eta_{T-2}^{(k)}\ldots\eta_{0}^{(k)})^\infty$ for
all $k=1,\ldots,{i-1},{i+1},\ldots,m$. Denote the right-infinite purely periodic word
$\eta_{T-1}^{(k)}\eta_{T-2}^{(k)}\ldots\eta_{0}^{(k)})^\infty$ via 
$u(z_k)=\tau_1^{(k)}\tau_2^{(k)}\ldots$ for suitable $\tau_q^{(\ell)}\in\F_p$.

Take arbitrary $x\in[a_i,b_i]$ and put $\mathbf x=(z_1;\ldots;z_{i-1};x; z_{i+1};\ldots;z_m)\in[\mathbf
a,\mathbf b]$; then $(\mathbf x;G(\mathbf
x))\in\mathbf P(\mathfrak A)$. Let $x=0.\chi_1\chi_2\ldots$ be a base-$p$
expansion of $x$ (the word $u(x)=\chi_1\chi_2\ldots$ is right-infinite); then from the definition of the plot it follows that there exists a strictly  increasing sequence $\bar r_1<\bar r_2<\ldots$ over $\N$ such that 
\begin{align}
\label{eq:thm:mai-mult1}
\lim_{\ell\to\infty}&0.\bar\chi_1\bar\chi_2\ldots\bar\chi_{\bar r_\ell}=x; \\
\label{eq:thm:mai-mult2}
\lim_{\ell\to\infty}&0.\bar\tau_1^{(k)}\bar\tau_2^{(k)}\ldots\bar\tau_{\bar
r_\ell}^{(k)}=
z_k\ \ (k=1,\ldots,{i-1},{i+1},\ldots,m);\\
\label{eq:thm:mai-mult3}
\lim_{\ell\to\infty}&0.\mathfrak a(\mathbf u_{\bar r_\ell}(\bar x))=G(\mathbf x),  
\end{align}
where 
\begin{align*}
&\mathbf u_{\bar r_\ell}(\bar x)=(u_{\bar r_\ell}(\bar z_1);\ldots;u_{\bar
r_\ell}(\bar
z_{i-1});u_{\bar r_\ell}(\bar x);
u_{\bar r_\ell}(\bar z_{i+1});\ldots;u_{\bar r_\ell}(\bar z_m));\\
&u_{\bar r_\ell}(\bar z_k)=\bar\tau_1^{(k)}\bar\tau_2^{(k)}\ldots\bar\tau_{r_\ell}^{(k)} \ \ (k=1,\ldots,{i-1},{i+1},\ldots,m);\\
&u_{\bar r_\ell}(\bar x)=\bar\chi_1\bar\chi_2\ldots\bar\chi_{\bar r_\ell},\\
\end{align*}
see remarks which follow Definition \ref{def:plot-auto-mult} above.  Moreover,
since base-$p$ of all $z_k$ are unique, the arguing like in the first part of the proof of Proposition \ref{prop:mark-up}
we conclude that there exists a state $s$ of the automaton $\mathfrak A$
and a strictly increasing sequence $r_1< r_2<\ldots$ over $\N$
such that 
\begin{align}
\label{eq:thm:mai-mult1.1}
\lim_{\ell\to\infty}&0.\chi_1\chi_2\ldots\chi_{r_\ell}=x; \\
\label{eq:thm:mai-mult2.1}
\lim_{\ell\to\infty}&0.\tau_1^{(k)}\tau_2^{(k)}\ldots\tau_{r_\ell}^{(k)}=
z_k\ \ (k=1,\ldots,{i-1},{i+1},\ldots,m);\\
\label{eq:thm:mai-mult3.1}
\lim_{\ell\to\infty}&0.\mathfrak a_s(\mathbf u_{r_\ell}(x))=G(\mathbf x),  
\end{align}
where $\mathfrak A_s$ is  the automaton which differs from $\mathfrak A$ only maybe by the initial state (which is $s$ rather than $s_0$). Now recall
that $\tau_1^{(k)}\tau_2^{(k)}\ldots=(\eta_{T-1}^{(k)}\eta_{T-2}^{(k)}\ldots\eta_{0}^{(k)})^\infty$ for
all $k=1,\ldots,{i-1},{i+1},\ldots,m$; so given $\ell\in\N$ let $q(\ell)\in\N$
be the largest such that $q_\ell< r_\ell$ and $\tau_{q_\ell}^{(k)}=\eta_0^{(k)}$
for some (thus, for all) $k=1,\ldots,{i-1},{i+1},\ldots,m$. Since all the
words $\tau_1^{(k)}\tau_2^{(k)}\ldots$ are periodic with a period of length
$T$ such $q_\ell$ exists for all sufficiently large $\ell\ge N$.  Denote via $s_\ell$
the state the automaton $\mathfrak A(s)$ reaches after being feeded (via
respective inputs) by words
$\tau_{q_\ell+1}^{(k)}\ldots\tau_{r_\ell}^{(k)}$ ($k=1,\ldots,{i-1},{i+1},\ldots,m$)
and $\chi_{q_\ell+1}^{(k)}\ldots\chi_{r_\ell}^{(k)}$. By the finiteness of
the automaton, in the sequence $(s_\ell)_{ell=N}^\infty$ at least one state,
say $\hat s$, occurs infinitely many times; therefore from 
\eqref{eq:thm:mai-mult1.1}--\eqref{eq:thm:mai-mult3.1} it follows that
\begin{align}
\label{eq:thm:mai-mult1.2}
\lim_{\ell\to\infty}&0.\chi_1\chi_2\ldots\chi_{q_\ell}=x; \\
\label{eq:thm:mai-mult2.2}
\lim_{\ell\to\infty}&0.(\eta_{T-1}^{(k)}\eta_{T-2}^{(k)}\ldots\eta_{0}^{(k)})^{q_\ell/T}
=
z_k\ \ (k=1,\ldots,{i-1},{i+1},\ldots,m);\\
\label{eq:thm:mai-mult3.2}
\lim_{\ell\to\infty}&0.\mathfrak a_{\hat s}(\mathbf u_{q_\ell}(x))=G(\mathbf x),  
\end{align}
where $\mathfrak A_{\hat s}$ is  the automaton which differs from $\mathfrak A$ only maybe by the initial state (which is $\hat s$ rather than $s_0$).
Note that $T$ is a divisor of $q_\ell$ by the construction of $q_\ell$ since
all the
words $\tau_1^{(k)}\tau_2^{(k)}\ldots$ are periodic with a period of length
$T$. By the definition of the plot we conclude that \eqref{eq:thm:mai-mult1.2}--\eqref{eq:thm:mai-mult3.2}
prove our claim.

Thus the function $\bar G_{ij}$ satisfies all conditions of Theorem \ref{thm:main};
therefore  the second derivative of $\bar G_{ij}$ is zero. But this means by the
construction of $G_{ij}$ that every second partial derivative $\partial^2 G_j/\partial^2 x_i$ is zero for all $z_k\in\Z_p\cap\Q\cap[a_k,b_k]$
($k=1,\ldots,{i-1},{i+1},\ldots,m$) and all $x\in[a_i,b_i]$. As $\Z_p\cap\Q\cap[a_k,b_k]$
is dense in $[a_k,b_k]$ for all $k=1,\ldots,{i-1},{i+1},\ldots,m$ we conclude
that $\partial^2 G_j/\partial x_i^2=0$ everywhere on $[\mathbf a,\mathbf
b]$ and for all $j=1,2,\ldots,n$, $i=1,2,\ldots,m$.

Now we are going to prove that $\partial^2 G_j/\partial x_i\partial x_t$
vanishes everywhere on  $[\mathbf a,\mathbf
b]$ and for all $j=1,2,\ldots,n$, $i,t=1,2,\ldots,m$, $i\ne t$ (without loss
of generality, let $t>i$ in what follows). Assume that
the opposite is true, that is, that there exist $i,j,t$ and a point $\mathbf
x\in[\mathbf a, \mathbf b]$ such that $\partial^2 G_j(\mathbf x)/\partial x_i\partial x_t\ne 0$. Then due to the continuity of second partial derivatives of the function $G$, 
by using the argument similar to that  from the beginning of the proof of Lemma \ref{le:der2} we conclude that there exist a point
(which without risk of misunderstanding we denote by the same symbol $\mathbf x$)
in $(\mathbf a,\mathbf b)=(a_1,b_1)\times\cdots\times(a_m,b_m)$ and a neighborhood
$U$ of that point  such that $\partial^2 G_j/\partial x_i\partial x_t>0$
everywhere on $U$. Therefore we always may take $z_k\in\Z_p\cap\Q\cap[a_k,b_k]$;
$k\ne i,t$; $M\in\N$ and $c,d\in\{0,1,\ldots, p^M-1\}$ such that 
the point $\mathbf z(x,y)=(z_1;\ldots;z_{i-1};x;z_{i+1};\ldots;z_{t-1};y;z_{t+1};\ldots;z_m)$
lies in $U$ for all $x=p^{-M}(c+e)$, $y=p^{-M}(d+h)$ and all
$e,h\in[0,1)$. Arguing like in the proof of Proposition \ref{prop:ext}
we see that
the following inclusion holds:
$$
\left\{(\bar{\mathbf z}(e,h);(p^MG(\bar{\mathbf z}(e,h)))\md1)\:e,h\in[0,1]\right\}\subset\mathbf
P(\mathfrak A),
$$
where  $\bar{\mathbf z}(e,h)=(\mathbf z(x,y))\md1$ (we reduce all coordinates
modulo 1).

Consider a finite automaton $\tilde{\mathfrak A}$ which is obtained by `gluing
together' the $i$-th and the $t$-th inputs of the automaton $\mathfrak A$ while feeding
the rest $k$-th inputs with infinite words $\wrd((p^Mz_k)\md1)$; that is, the
automaton function of the automaton $\tilde{\mathfrak A}$ is 
$$
f_{\tilde{\mathfrak A}}(v)=f_{\mathfrak A}(w_1;\ldots;w_{i-1};v;w_{i+1};\ldots;w_{t-1};v;w_{t+1};\ldots;w_m)
$$ 
where $w_\ell=\wrd((p^Mz_k)\md1)\in\Cal W^\infty$, $\ell\in\{1,2,\ldots,m\}\setminus\{i,t\}$.
By argument similar to that for the case $i=t$ (see
the proof of the Claim above) we conclude that the automaton $\tilde{\mathfrak A}$ is
finite and that the graph of
the function $\bar G_j(h,h)=(p^MG_j(\mathbf z(p^{-M}(c+h),p^{-M}(d+h)))\md1\:[0,1]^2\>[0,1)$
when  $h$ is running through $[0,1)$ lies in $\mathbf P(\tilde{\mathfrak A})$. But on the other hand we have that 
$$
\partial^2\bar G_j(h,h)/\partial h^2=(\partial/\partial
x_i+\partial/\partial x_t)^2 G_j({\mathbf z}(x,y))=2\cdot\partial^2G_j(\mathbf z(x,y))/\partial
x_i\partial x_t
$$
since $\partial^2G_j(\mathbf z(x,y))/\partial x_i^2=\partial^2G_j(\mathbf z(x,y))/\partial x_t^2=0$ by what we have already proved above. But this
is a contradiction to Theorem \ref{thm:main} since  the function $\bar G_j(h,h)$
of argument $h$
satisfies all conditions of the theorem and has a non-zero second derivative.
Thus we have proved that under conditions of Theorem \ref{thm:main-mult}
the function $G$ must be affine:  $G(\mathbf x)=\mathbf x\mathbf A+\mathbf
B$ for all $\mathbf x\in[\mathbf a,\mathbf b]$.

Now fix arbitrary $i\in\{1,2,\ldots,m\}$, $j\in\{1,2\ldots,n\}$, and 
$z_k\in[a_k,b_k]\cap\Z_p\cap\Q$ for $k=1,2,\ldots,m$, $k\ne
i$; consider the function $\bar G_{ij}$ and the automaton $\bar{\mathfrak
A}_{ij}$ as in the beginning of the proof of Theorem \ref{thm:main-mult}.
Then from the affinity of the function $G$ it follows that $\bar G_{ij}(x)=xA_{ij}
+B_j$. Since $\mathbf G_{[a_i,b_i]}(\bar G_{ij})\subset\mathbf
P(\bar{\mathfrak A}_{ij})$ by the Claim above, Theorem \ref{thm:main} implies
that $A_{ij}, B_j\in\Z_p\cap\Q$.

Further, arguing like in the proof of Proposition \ref{prop:ext}
we conclude that for  suitable $M\in\N$ and $\mathbf h\in\{0,1,\ldots,p^M-1\}^m$ the graph $\mathbf G_{[0,1]^m}(H(\mathbf v))$ 
of the function $H(\mathbf v)=H_{\mathbf h,M}(\mathbf v)=(p^M((p^{-M}(\mathbf h+\mathbf
v)\mathbf A+\mathbf B)\md1=(\mathbf v\mathbf A+(\mathbf h\mathbf A+p^M\mathbf
B))\md1$ lies completely in $\mathbf P(\mathfrak A)$. Now considering the
function $H$ and the corresponding
automaton $\bar{\mathfrak A}_{ij}$ as above  for $z_k=0$, $k=1,2,\ldots,m$, $k\ne
i$, $G=H$, we conclude by Theorem \ref{thm:main} that there are only finitely many
$A_{ij}$; whence finitely many $\mathbf A$. 

If for some of these $\mathbf A$ there were infinitely many
$\mathbf B\md1$ 
such that $\mathbf G_{[0,1]^m}(H(\mathbf v))\subset\mathbf P(\mathfrak A)$
then  for some $j\in\{1,2,\ldots,n\}$ there were infinitely many pairwise
distinct $B_j\md1$.
But given arbitrary $z_k\in[a_k,b_k]\cap\Z_p\cap\Q$ for $k=2,3,\ldots,m$ and considering
corresponding automata $\bar {\mathfrak A}_{1j}$ for various $(m-1)$-tuples
$(z_2,\ldots,z_m)$ (cf. the beginning of the proof of Theorem \ref{thm:main-mult}), from the construction of  $\bar {\mathfrak A}_{1j}$ it follows (cf. the proof
of the Claim) that there are only finitely many these automata
$\bar {\mathfrak A}_{1j}$ since the automaton $\mathfrak A$ is finite. Therefore
applying Theorem \ref{thm:main} to every automaton $\bar {\mathfrak A}_{1j}$
we finally conclude that there are only finitely many $B_j\md 1$; a contradiction
to our assumption.

 Therefore
there are only finitely many pairwise distinct functions $H_{\mathbf h,M}$ as above. Now by mimic the respective part of the proof of the first assertion of Theorem \ref{thm:main} we conclude that given an $(m\times n)$-matrix
$\mathbf A$ and a vector $\mathbf B_j$ over $\Z_p\cap\Q$ such that
the graph of the function $G(\mathbf x)=\mathbf x\mathbf A+\mathbf B$ on
$[\mathbf a,\mathbf b]\subset[0,1]^m$ lies completely in $\mathbf P(\mathfrak
A)$ then necessarily 
$\mathbf G_{\R^m}((\mathbf
x \mathbf
A+\mathbf B)\md1)\subset\mathbf P(\mathfrak
A)\subset\T^{n+m}$.

\end{proof}

\begin{note}
\label{note:main-alph}
An automaton with a single input and a single output over  respective
alphabets 
$\{0,1,\ldots,p^n-1\}$ and $\{0,1,\ldots,p^k-1\}$, $(n,k\ge1)$, can be considered as an automaton
with $n$ inputs and $k$ outputs over an alphabet $\{0,1,\ldots,p-1\}$ and
therefore
Theorem \ref{thm:main-mult} can be applied to automata of that sort as well.
\end{note}

\section{Discussion: \emph{It from bit}, indeed}
\label{sec:Concl}
Now we are going to outline possible relations of main results of preceding
section to quantum theory leaving apart applications to cryptography (the
latter are subject of future paper).
Although further physical interpretation of the results is highly speculative, it
reveals deep analogies between automata and quantum systems and thus
worth a short discussion to explain a direction in which it is reasonable
to develop the results in order to derive some physically meaningful assertions (and maybe
models) from mathematical theorems of the paper.

We start with some remarks on what is `physical law'. Let us (somewhat naively) think of a physical law as of  mathematical correspondence
 between
quantities which express  impacts  a physical system is exposed to and quantities which express responses the system exhibits. Suppose for simplicity that both impacts and responses are scalars.
As  the measured experimental values of physical quantities are rational numbers (since there is no possibility to obtain during
measurements an exact value of irrational number, cf.
\cite{Vladimirov/Volovich/Zelenov:1994,Khrennikov:1996A,Khrennikov:1997}) the result of measurements are points
in $\R^2$, the experimental points.  To find  a particular physical law one seeks
 for a correspondence between cluster points (w.r.t. the metrics in $\R$) of experimental values and tries to draw
an experimental curve. The latter curve is a (piecewise) smooth curve (the $C^2$-smoothness is common)  which is the
best approximation of the set of the experimental points.
A physical law is then  a curve which  approximate with the highest achievable accuracy (w.r.t. metric
in $\R^2$) the experimental curves obtained during series of measurements.

Let physical quantities which correspond to impacts and reactions 
be quantized; i.e, let they take only values (measured in suitable units
and properly normalized), say, $0,1,\ldots,p-1$, where $p>1$ is an
integer. Then, once the system is exposed to
a sequence of  $k$ of impacts, it produces corresponding sequence
of  $k$ reactions. Every impact changes current state of the system to a
new one; therefore provided the systems is causal, both the next state and the reaction (effect) depends only on impacts
(causes) the system has already  been exposed to; so an automaton $\mathfrak A$ is an adequate model
of the system\footnote{We stress that we are  \emph{not} speaking  here about the so-called memory effect of
the macroscopic measurement equipment which may `remember' its previous interactions
with particles, cf. \cite{Khren-mem-effect}; we only say that every interaction
(impact)
forces the system (e.g. a particle) to change its state to some another one.
We do  \emph{not} discuss the nature of these states which are \emph{not}
necessarily quantum states; we just say that every interaction changes something in a
system and refer to this `something' as to a `state' of the system, and nothing
more.}. 
Every finite  sequence $\alpha_{k-1},\ldots,\alpha_0$ of impacts/reactions corresponds to
a base-$p$ expansion of natural number $z=\alpha_{k-1}p^{k-1}+\cdots+\alpha_0$
to which after normalization there corresponds a rational number $\frac{z}{p^k}$.
Every measurement is a sequence of interactions $\alpha_{k-1},\ldots,\alpha_0$ of the measurement instrument
with the system, and if the accuracy of the instrument is not
better than $p^{-N}$, then the result of a single measurement lies within
the segment $[\frac{z}{p^k}-p^{-N},\frac{z}{p^k}+p^{-N}]$. Assuming that
$k\gg N$ we see that even if the system before every measurement has been prepared in a fixed
state $s_0$ (the initial state of the automaton)  during a single measurement
the system $\mathfrak A(s_0)$ will be exposed to a random sequences of impacts $\alpha_{k-M-1},\ldots,\alpha_0$
which switches the system to a new state $s=s(\alpha_{k-1},\ldots,\alpha_0)$;
so actually as a result of the measurement  due  to its limited accuracy we obtain an experimental point
 $(0.\alpha_k\ldots\alpha_{k-M};0.\beta_k\ldots\beta_{k-M})\in\R^2$
where $\beta_k\ldots\beta_{k-M}$ is the output of the automaton $\mathfrak
A(s)$ (whose initial state is $s=s(\alpha_{k-1},\ldots,\alpha_0)$) feeded
by the sequence $\alpha_k,\ldots,\alpha_{k-M}$.


Theorem \ref{thm:main} shows that if the number of states of the system $\mathfrak
A$ is much less than the length
of input sequence of impacts then experimental curves necessarily tend
to straight
lines (or torus windings, under a natural map of the unit square onto a torus),
cf. Figures \ref{fig:Plot-16}, \ref{fig:Plot-17}, and \ref{fig:Plot}.
This may be judged  as linearity of corresponding physical law and, what
is even more important, the way experimental points are clustering  on the unit square
is very much alike  to that of  the points where electrons hit target screen in a double-slit experiment, cf. Figures \ref{fig:Plot-16}--\ref{fig:Plot-17} and Figure \ref{fig:2slit}.
\begin{figure}
\includegraphics[width=0.8\textwidth,natwidth=610,natheight=642]{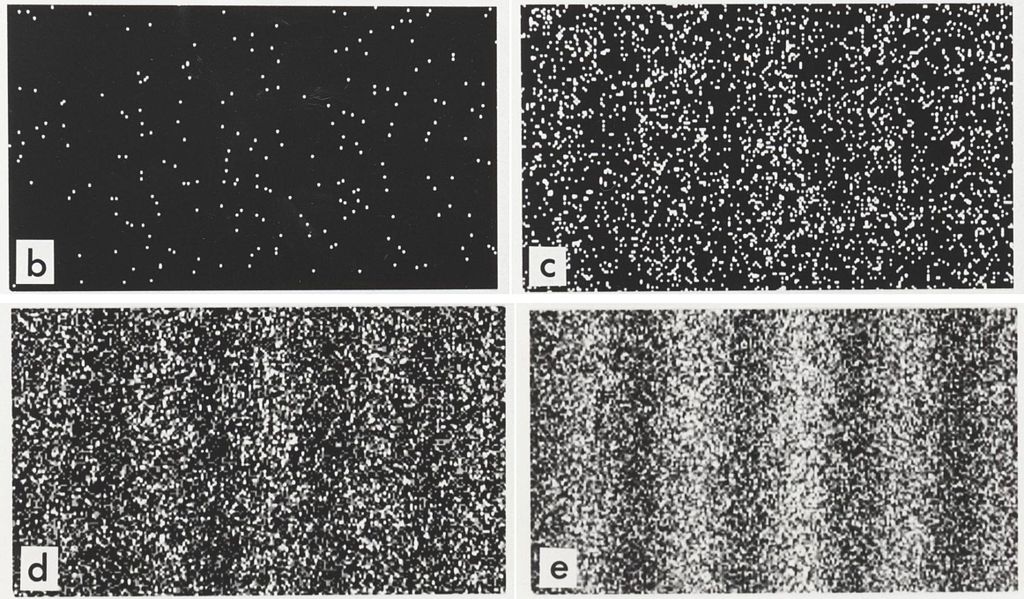}
\caption{Interference pattern of the double slit experiment.
{\tiny From  Wikimedia Commons, the free media repository
http://commons.wikimedia.org/wiki/File:Double-slit experiment results Tanamura
four.jpg
}
}
\label{fig:2slit}
\end{figure}
We are not going here to discuss further parallels of the computer experiments
with automata and behaviour of quantum systems such as analogies between
transition and ergodic states of automata and mixed and pure quantum states
respectively, or probabilities of Markov chain related to an automaton and
probabilities in quantum systems, etc.: Although we believe that the analogies
are not  external but reflect deep relations between quantum systems and
automata, the
issues are far from the subject of the paper and that's why the discussion
is postponed to further relevant papers. Here we briefly touch only an interesting
analogy between smooth curves in plots of finite automata and matter waves
of quantum theory.

By Theorem \ref{thm:main}, the smooth curves from the plot of a finite automaton
$\mathfrak A$
can be described by families of complex-valued exponential functions of the form $
\psi_k(y)=e^{i(Ay-2\pi p^kB)}$, 
$k=0,1,2,\ldots$,
for suitable $A,B\in\Z_p\cap\Q$, cf. Corollary \ref{cor:mult-add-compl}.
The wave function of a particle is of the form $ce^{i(mx-t\omega)}$ where
$m$ is momentum, $x$ position, $\omega$ angular frequency, and $c$ is a complex
amplitude. Comparing
the two expressions we see that $p^k$ may serve as a time  for the
automaton $\mathfrak A$ since multiplication by $p^k$ is a $k$-step shift
of a base-$p$ expansion of a number. But can we someway associate it to physical
time $t$ of quantum theory? In what follows we argue that yes, there is a
natural way to do this.

Let us forget for a moment that $p$ is a positive integer and suppose that $p=1+\tau$
where $1\gg\tau>0$ is a small real number; then $p^k\approx 1+k\tau$ and
if $\tau$ is a small time interval which is out of accuracy of measurements
(e.g., let $\tau$ be  Planck time which is approximately $10^{-43}$ s.).
Therefore the torus link $
\psi_k(y)=e^{i(Ay-2\pi p^kB)}$, 
$k=0,1,2,\ldots$ can be approximately described by $
\Psi(y,t)=e^{-i\cdot 2\pi B}e^{i(Ay-2\pi tB)}$, 
$y,t\in\R$ since it is reasonable to assume that $k\tau$  is just a time $t$
as $\tau$ is a small time interval, a time quantum, the Planck time.
But  $\Psi(y,t)$ is a wave function of a particle with
momentum $A$, angular frequency $2\pi B$ and amplitude $e^{-i\cdot 2\pi B}$.
Is this mathematically correct to substitute $1+\tau$ for $p$ in our reasoning?
Yes, this is correct; but to explain why this is correct we need to recall a notion of \emph{$\beta$-expansion}
of real number.   

The $\beta$-expansions are radix expansions in non-integer
bases; they were first introduced more than half-century
ago, see \cite{Renyi-beta,Parry-beta}, and now $\beta$-expansions are a substantial
part of
dynamics, see e.g. survey  \cite{Sidorov-beta}. Following \cite{Sidorov-beta},
given $x\in[0,$ and $\beta\in\R$, $\beta>1$ we call a sequence $(\chi_i)_{i=1}^\infty$
over the alphabet $\{0,1,\ldots, \lfloor\beta\rfloor\}$ a \emph{$\beta$-expansion}  of $x$ once $x=\sum_{i=-N}^\infty
\chi_i\beta^{-1}$ for suitable $N\in\Z$. Note that sometimes the term $\beta$-expansion is used in
a narrower meaning, when the `digits' $\chi_i$ are obtained by the so-called
`greedy algorithm' only, cf. \cite[Section 7.2]{AlgCombinWords} but this
is not important at the moment: In what follows we just sketch the way how the results of current paper can
be modified to handle the case of $\beta$-expansions rather than the case
of base-$p$ expansions only. We leave details and rigorous
proofs for further paper.

From the definition we see that the notion of $\beta$-expansion is a generalization
of the notion of base-$p$ expansion: It is clear that for $\beta=p$ the $\beta$-expansion
of $x$ is just base-$p$ expansion of $x$, and that is why both  $\beta$-expansions
and base-$p$ expansions share some common properties. For instance, given
$\beta$-expansion
of reals it is
possible to perform arithmetic operations with reals in a way similar to
that of school-textbook algorithms for base-$p$ expansions of reals. However,
differences between base-$p$ expansions and $\beta$-expansions should
also be taken into the account
since when $\beta$ is not an integer, a $\beta$-expansion of a real number
is generally
not unique; moreover a real number may have a continuum of different $\beta$-expansions for $\beta$ fixed. Nonetheless, we can perform arithmetic operations with
numbers represented by $\beta$-expansions, i.e., with words over
the alphabet $\{0,1,\ldots,\lfloor\beta\rfloor\}$. These operations for
some non-integer $\beta$ may be represented by finite automata as well. For
instance, if $\beta=\sqrt[n]2$ then  arithmetic operations with numbers  represented
by $\sqrt[n]2$-expansions $\ldots\alpha_2\alpha_1\alpha_0$ and 
$\ldots\gamma_2\gamma_1\gamma_0$ (which are binary words over the alphabet $\{0,1\}$
since $\lfloor\sqrt[n]2\rfloor=1$)
can be  performed in a manner  similar to that when one applies school-textbook 
algorithms for base-$p$ expansions, with the only difference: A  `carry' from $i$-th
position should
be added to $(n+i+1)$-th position; e.g. for $\beta=\sqrt 2$ we have that $11+01=110$ while in the case $\beta=2$ we have that $11+01=100$. Note 
that $01=1$, $11=\sqrt2+1$ (and thus $110=(\sqrt 2)^2+(\sqrt 2)^1+0=2+\sqrt
2$) when $\beta=\sqrt 2$; and $01=1$, $11=3$ when $\beta=2$.

When an automaton  $\mathfrak A$ proceeds a word (or, a corresponding system reacts to impacts) it just
evaluates step-by-step a $p$-adic 1-Lipschitz function  $f_\mathfrak A\:\Z_p\to\Z_p$
(cf. Subsection \ref{ssec:a-map}),
and no $\beta$ appears at this moment. But we need to specify $\beta$ when
we `visualize' the function $f_\mathfrak A$ in $\R^2$: To every  word 
$\alpha_{k-1}\ldots\alpha_0$
over
the alphabet $\F_p=\{0,1,\ldots,p-1\}$ we put into the correspondence a point
$(\beta^{-k}(\alpha_{k-1}\beta^{k-1}+\cdots+\alpha_1\beta+\alpha_0))\md1\in[0,1)$;
thus to every pair of input/output words of the automaton there corresponds
a point in the unit square $\mathbb I^2$(or, on the unit torus $\T^2\subset\R^3$). We then take a closure of all these
points and obtain a \emph{$\beta$-plot} of the automaton $\mathfrak A$ in
a way similar to that when we constructed a plot of the automaton (which
corresponds to the case when $\beta=p$), cf. Definition
\ref{def:plot-auto}. We then  consider smooth curves in the $\beta$-plots
of finite automata, in particular, the curves which correspond to affine
automata functions $z\mapsto Az+B$. To these functions there correspond torus
windings which can be expressed in a form of complex-valued functions
 $\psi_k(y)=e^{i(Ay-2\pi \beta^kB)}$, $k=0,1,2\ldots$, $y\in\R$;
and these functions can by approximated  with arbitrarily high accuracy by functions $\Psi(y,t)=e^{-i\cdot 2\pi B}e^{i(Ay-2\pi tB)}$, $t,y\in\R$, just
by taking $\beta>1$ sufficiently close to 1. Moreover, the case when $\beta$ is close to 1 is the only case when  approximations are of the form
of wave functions. But this means that the \emph{corresponding
automata must necessarily be binary}; i.e., their input/output alphabets are 
$\{0,1,\ldots,\lfloor\beta\rfloor\}=\{0,1\}$. So these automata (which are
just models of causal discrete systems) indeed produce
waves, the \emph{its}, \emph{from bits}.

From this view, main results of the current paper may be considered as a contribution to informational
interpretation
of quantum theory, namely, to
J.~A.~Wheeler's \emph{It from bit} doctrine which suggests that all things physical (`its') are
information-theoretic in origin (`from bits'), \cite{Wheeler_IFB}: We have
given
some evidence above that this is indeed so regarding particular `its', the
matter waves. We stress once again that our conclusion is based on the following
assumptions only: A quantum system is causal and discrete, whence is an
automaton; and the number of states of the automaton is finite.


\begin{thebibliography}{100}

\bibitem{RealAnalys}
Charalambos~D. Aliprantis and Owen Burkinshaw.
\newblock {\em Principles of real analysis}.
\newblock Academic Press, Inc., third edition, 1998.

\bibitem{Allouche-Shall}
J.-P. Allouche and J.~Shallit.
\newblock {\em Automatic Sequences. Theory, Applications, Generalizations}.
\newblock Cambridge Univ. Press, 2003.

\bibitem{AnKhr}
V.~Anashin and A.~Khrennikov.
\newblock {\em Applied Algebraic Dynamics}, volume~49 of {\em de Gruyter
  Expositions in Mathematics}.
\newblock Walter~de~Gruyter GmbH \& Co., Berlin---N.Y., 2009.

\bibitem{AKY-DAN}
V.~S. Anashin, A.~Yu. Khrennikov, and E.~I. Yurova.
\newblock Characterization of ergodicity of $p$-adic dynamical systems by using
  van der {P}ut basis.
\newblock {\em Doklady Mathematics}, 83(3):306--308, 2011.

\bibitem{me-auto_fin}
Vladimir Anashin.
\newblock Automata finiteness criterion in terms of van der {P}ut series of
  automata functions.
\newblock {\em $p$-Adic Numbers, Ultrametric Analysis and Applications},
  4(2):151--160, 2012.

\bibitem{me:Discr_Syst}
Vladimir Anashin.
\newblock The non-{A}rchimedean theory of discrete systems.
\newblock {\em Math. Comp. Sci.}, 6(4):375--393, 2012.

\bibitem{Haeseler_Barbe}
Andre Barb\'e and Friedrich von Haeseler.
\newblock Limit sets of automatic sequences.
\newblock {\em Adv. Math.}, 175:169--196, 2003.

\bibitem{Bra}
W.~Brauer.
\newblock {\em Automatentheorie}.
\newblock B.~G.~Teubner, Stuttgart, 1984.

\bibitem{Car-Long_Automata}
John Carroll and Darrell Long.
\newblock {\em Theory of Finite Automata}.
\newblock Prentice-Hall Inc., 1989.

\bibitem{Cherep_Approx}
A.~N. Cherepov.
\newblock On approximation of continuous functions by determinate functions
  with delay.
\newblock {\em Discrete Math. Appl.}, 22(1):1--24, 2010.

\bibitem{Cherepov_Approx-contf}
A.~N. Cherepov.
\newblock Approximation of continuous functions by finite automata.
\newblock {\em Discrete Math. Appl.}, 22(4):445--453, 2012.

\bibitem{Fox}
R.~Crowell and R.~Fox.
\newblock {\em Introduction to the Knot Theory}.
\newblock Ginu and Co., Boston, 1963.

\bibitem{Khren-mem-effect}
D.~Dubischar, V.~M. Gundlach, O.~Steinkamp, and A.~Khrennikov.
\newblock The interference phenomenon, memory effects in the equipment and
  random dynamical systems over the fields of $p$-adic numbers.
\newblock {\em Nuovo Cimento B}, 114(4):373--382, 1999.

\bibitem{DuFomNov_ModGeo}
B.~A. Dubrovin, A.~T. Fomenko, and S.~P. Novikov.
\newblock {\em Modern Geometry - Methods and Applications}, volume~II.
\newblock Springer-Verlag, NY--Berlin-Heidelberg-Tokyo, 1985.

\bibitem{Eilenberg_Auto}
Samuel Eilenberg.
\newblock {\em Automata, Languages, and Machines}, volume~A.
\newblock Academic Press, 1974.

\bibitem{Frougny_Rat-base-p}
C.~Frougny and K.~Klouda.
\newblock Rational base number systems for $p$-adic numbers.
\newblock {\em RAIPO Theor. Inform. Appl.}, 46(1):87--106, 2012.

\bibitem{Gouvea:1997}
F.~Q. Gouv{\^e}a.
\newblock {\em $p$-adic Numbers, An Introduction}.
\newblock Springer-Verlag, Berlin--Heidelberg--New York, second edition, 1997.

\bibitem{Grigorch_auto}
R.~I. Grigorchuk, V.~V. Nekrashevich, and V.~I. Sushchanskii.
\newblock Automata, dynamical systems, and groups.
\newblock {\em Proc. Steklov Institute Math.}, 231:128--203, 2000.

\bibitem{Hass-Katok_First}
B.~Hasselblatt and A.~Katok.
\newblock {\em A First Course in Dynamics}.
\newblock Cambridge Univ. Press, Cambridge, etc., 2003.

\bibitem{Kat}
S.~Katok.
\newblock {\em $p$-adic analysis in comparison with real}.
\newblock Mass. Selecta. American Mathematical Society, 2003.

\bibitem{Kem-Snell_FinMaCh}
John~G. Kemeny and J.~Laurie Snell.
\newblock {\em Finite Markov Chains}.
\newblock Springer-Verlag, 1976.

\bibitem{Khren-quick}
A.~Khrennikov.
\newblock Quantum mechanics from time scaling and random fluctuations at the
  "quick time scale".
\newblock {\em Nuovo Cimento B}, 121(9):1005--1021, 2006.

\bibitem{Khren-quant-av}
A.~Khrennikov.
\newblock To quantum averages through asymptotic expansion of classical
  averages on infinite-dimensional space.
\newblock {\em Math. Phys.}, 48(1), 2007.
\newblock Art. No. 013512.

\bibitem{Khrennikov:1996A}
A.~Yu. Khrennikov.
\newblock Ultrametric {H}ilbert space representation of quantum mechanics with
  a finite exactness.
\newblock {\em Found. Physics}, 26:1033--1054, 1996.

\bibitem{Khrennikov:1997}
A.~Yu. Khrennikov.
\newblock {\em Non-{A}rchimedean Analysis: {Q}uantum Paradoxes, Dynamical
  Systems and Biological Models}.
\newblock Kluwer Academic Publishers, Dordrecht, 1997.

\bibitem{Kobl}
N.~Koblitz.
\newblock {\em $p$-adic numbers, $p$-adic analysis, and zeta-functions},
  volume~58 of {\em Graduate texts in math.}
\newblock Springer-Verlag, second edition, 1984.

\bibitem{Konech_Aff}
Michal Kone{\v c}n{\' y}.
\newblock Real functions computable by finite automata using affine
  representations.
\newblock {\em Theor. Comput. Sci.}, 284:373--396, 2002.

\bibitem{Lisovik_Realfunk}
L.~P. Lisovik and O.~Yu. Shkaravskaya.
\newblock Real functions defined by transducers.
\newblock {\em Cybernetics and System Analysis}, 34(1):69--76, 1998.

\bibitem{AlgCombinWords}
M.~Lothaire.
\newblock {\em Algebraic Combinatorics on Words}.
\newblock Cambridge Univ. Press, 2002.

\bibitem{Lunts}
A.~G. Lunts.
\newblock The $p$-adic apparatus in the theory of finite automata.
\newblock {\em Problemy Kibernetiki}, 14:17--30, 1965.
\newblock In Russian.

\bibitem{Mah}
K.~Mahler.
\newblock {\em $p$-adic numbers and their functions}.
\newblock Cambridge Univ. Press, 1981.
\newblock (2nd edition).

\bibitem{Mansurov-knots}
V.~Mansurov.
\newblock {\em Knot theory}.
\newblock Chapman \& Hall/CRC, Boca Raton - London - NY - Washington, 2004.

\bibitem{Misch-Fom}
A.~Mishchenko and A.~Fomenko.
\newblock {\em A course of differential geometry and topology}.
\newblock Mir, Moscow, 1988.

\bibitem{Monna}
A.~F. Monna.
\newblock Sur une transformation simple des nombres $p$-adiques en nombres
  r\'eels.
\newblock {\em Indag. Math.}, 14:1--9, 1952.

\bibitem{Parry-beta}
W.~Parry.
\newblock On the $\beta$-expansions of real numbers.
\newblock {\em Acta Math. Acad. Sci. Hung.}, 11:401--416, 1960.

\bibitem{ComputAnalPhys}
M.~B. Pour-El and J.~I. Richards.
\newblock {\em Computability in Analysis and Physics}.
\newblock Springer-Verlag, 1989.

\bibitem{Renyi-beta}
A.~R\'enyi.
\newblock Representation for real numbers and their ergodic properties.
\newblock {\em Acta Math. Acad. Sci. Hung.}, 8:477--493, 1957.

\bibitem{Sch}
W.~H. Schikhof.
\newblock {\em Ultrametric calculus}.
\newblock Cambridge University Press, 1984.

\bibitem{Shkar_AffineAuto}
O.~Yu. Shkaravskaya.
\newblock Affine mappings defined by finite transducers.
\newblock {\em Cybernetics and System Analysis}, 34(5):781--783, 1998.

\bibitem{Sidorov-beta}
N.~Sidorov.
\newblock Arithmetic dynamics.
\newblock In S.~Bezuglyi and S.~Kolyada, editors, {\em Topics in dynamics and
  ergodic theory}, volume 310 of {\em London Math. Soc. Lecture Note Series},
  pages 145--189. Cambridge University Press, Cambridge, 2003.

\bibitem{Smyshl-fin-auto}
T.~I. Smyshlyaeva.
\newblock A criterion for functions defined by automata to be
  bounded-determinate.
\newblock {\em Diskret. Mat.}, 25(2):121--134, 2013.

\bibitem{Vladimirov/Volovich/Zelenov:1994}
V.~S. Vladimirov, I.~V. Volovich, and E.~I. Zelenov.
\newblock {\em $p$-adic Analysis and Mathematical Physics}.
\newblock World Scientific, Singapore, 1994.

\bibitem{Vuillem_circ}
J.~Vuillemin.
\newblock On circuits and numbers.
\newblock {\em IEEE Trans. on Computers}, 43(8):868--879, 1994.

\bibitem{Vuillem_fin}
J.~Vuillemin.
\newblock Finite digital synchronous circuits are characterized by 2-algebraic
  truth tables.
\newblock In {\em Advances in computing science - ASIAN 2000}, volume 1961 of
  {\em Lecture Notes in Computer Science}, pages 1--7, 2000.

\bibitem{Vuillem_DigNum}
J.~Vuillemin.
\newblock Digital algebra and circuits.
\newblock In {\em Verification:Theory and Practice}, volume 2772 of {\em
  Lecture Notes in Computer Science}, pages 733--746, 2003.

\bibitem{Wheeler_IFB}
John~A. Wheeler.
\newblock Information, physics, quantum: {T}he search for links.
\newblock In W.~H. Zurek, editor, {\em Complexity, Entropy, and the Physics of
  Information}, pages 309--336, Redwood City, Calif., 1990. Addison-Wesley Pub.
  Co.

\bibitem{Yb-eng}
S.~V. Yablonsky.
\newblock {\em Introduction to discrete mathematics}.
\newblock Mir, Moscow, 1989.

\end{thebibliography}
\end{document}